\documentclass[11pt,letterpaper]{amsart}

\usepackage[letterpaper,margin=1.3in,footskip=0.90in]{geometry}
\usepackage{microtype}
\usepackage{amsmath,amssymb,amsthm}
\usepackage{mathtools}
\usepackage[dvipsnames]{xcolor}
\usepackage{enumitem}
\setlist{noitemsep}
\usepackage{mathrsfs}
\usepackage{tikz-cd}
\usepackage{array}
\usepackage[pdfusetitle,colorlinks]{hyperref}
\hypersetup{citecolor=black,linkcolor=black,urlcolor=black}
\usepackage[capitalise,noabbrev]{cleveref}
\crefformat{equation}{\ensuremath{(#2#1#3)}}
\crefmultiformat{equation}{\ensuremath{(#2#1#3)}}{ and~\ensuremath{(#2#1#3)}}{, \ensuremath{(#2#1#3)}}{, and~\ensuremath{(#2#1#3)}}
\numberwithin{figure}{section}
\numberwithin{equation}{subsection}
\newtheorem{theorem}[figure]{Theorem}
\newtheorem{lemma}[figure]{Lemma}
\newtheorem{corollary}[figure]{Corollary}
\newtheorem{proposition}[figure]{Proposition}
\theoremstyle{definition}
\newtheorem{definition}[figure]{Definition}
\newtheorem{notation}[figure]{Notation}
\theoremstyle{definition}
\newtheorem{remark}[figure]{Remark}
\theoremstyle{definition}
\newtheorem{example}[figure]{Example}
\theoremstyle{definition}
\newtheorem{construction}[figure]{Construction}

\AddToHook{env/definition/begin}{\crefalias{equation}{definition}}
\AddToHook{env/question/begin}{\crefalias{equation}{question}}
\AddToHook{env/theorem/begin}{\crefalias{equation}{theorem}}
\AddToHook{env/proposition/begin}{\crefalias{equation}{proposition}}
\AddToHook{env/corollary/begin}{\crefalias{equation}{corollary}}
\AddToHook{env/observation/begin}{\crefalias{equation}{observation}}
\AddToHook{env/construction/begin}{\crefalias{equation}{construction}}
\AddToHook{env/example/begin}{\crefalias{equation}{example}}
\AddToHook{env/notation/begin}{\crefalias{equation}{notation}}
\AddToHook{env/lemma/begin}{\crefalias{equation}{lemma}}
\AddToHook{env/remark/begin}{\crefalias{equation}{remark}}
\AddToHook{env/question/begin}{\crefalias{equation}{question}}
\AddToHook{env/conjecture/begin}{\crefalias{equation}{conjecture}}

\usepackage[all]{xy}
\usepackage{relsize}
\usepackage[bbgreekl]{mathbbol}
\usepackage{amsfonts}
\DeclareSymbolFontAlphabet{\mathbb}{AMSb} 
\DeclareSymbolFontAlphabet{\mathbbl}{bbold}
\newcommand{\Prism}{{\mathlarger{\mathbbl{\Delta}}}}

\usepackage{lipsum}
\usepackage{fancyhdr}
\newcommand{\w}{\text}

\newcommand{\Fil}{\mathrm{Fil}}
\newcommand{\syn}{\mathrm{syn}}
\newcommand{\Nyg}{\mathrm{Nyg}}
\newcommand{\Spf}{\mathrm{Spf}}
\newcommand{\qrsp}{\mathrm{qrsp}}
\newcommand{\Ga}{\mathrm{Gauge}_{\Prism}}
\newcommand{\Te}{T_{\mathrm{\acute{e}t}}}
\newcommand{\BK}{\mathrm{BK}}
\newcommand{\qsyn}{\mathrm{qsyn}}
\newcommand{\Hom}{\mathrm{Hom}}
\newcommand{\gr}{\mathrm{gr}}

\newcolumntype{C}[1]{>{\centering\arraybackslash}p{#1}}

\numberwithin{figure}{subsection}
\numberwithin{equation}{subsection}

\title{{Dieudonn\'e theory via classifying stacks and prismatic $F$-gauges}}
\author{ Shubhodip Mondal}
\date{}

\address[Shubhodip Mondal]{\parbox{\linewidth}{Purdue University, West Lafayette, USA
} }
\email{mondalsh@purdue.edu}
\begin{document}

\begin{abstract}
In this paper, we apply stack theoretic ideas to the classification problem in Dieudonn\'e theory. First, we use crystalline cohomology of classifying stacks to directly reconstruct the classical Dieudonn\'e module of a finite, $p$-power rank, commutative group scheme $G$ over a perfect field $k$ of characteristic $p>0$. As a consequence, we give a new, much shorter proof of the isomorphism $\sigma^* M(G) \simeq \mathrm{Ext}^1 (G, \mathcal{O}^{\mathrm{crys}})$ due to Berthelot--Breen--Messing using stacky methods combined with the theory of de Rham--Witt complexes. Additionally, we show that finite locally free commutative group schemes of $p$-power rank over a quasisyntomic base can be classified in terms of ``prismatic Dieudonn\'e $F$-gauges", which we introduce by making constructions using (higher) classifying stacks. The latter generalizes the result of Ansch\"utz and Le Bras on classification of $p$-divisible groups, which we also reprove using our approach. Along the way, we prove a description of cohomology with coefficients in group schemes, compatibility with Cartier duality, and reconstruction of Galois representations in terms of our prismatic Dieudonn\'e $F$-gauges.
\end{abstract}

\maketitle

\tableofcontents

\newpage
\section{Introduction} Let $p$ be a fixed prime number. Let $k$ be a perfect field of characteristic $p.$ Classically, Dieudonn\'e theory offers a classification of arithmetic objects such as $p$-divisible groups and finite commutative group schemes of $p$-power rank in terms of certain linear algebraic data, called Dieudonn\'e modules. Grothendieck suggested that the Dieudonn\'e module should be viewed as a Dieudonn\'e crystal, which has been realized by the work of Messing \cite{mess1}, Mazur--Messing \cite{Mess2} and Berthelot--Breen--Messing \cite{bbm}, \cite{mess3}. This subject has received contributions from many authors and we refer to \cite{lau} for a survey. Recently, over mixed characteristic base rings, Ansch\"utz--Le Bras \cite{alb} defined a notion of prismatic Dieudonn\'e crystal associated to a $p$-divisible group and proved a classification result, using the theory of prismatic cohomology due to \cite{BS19}.

 In this paper, we revisit aspects of crystalline Dieudonn\'e theory \cite{bbm} from an entirely different perspective, using cohomology of classifying stacks, giving a short proof the main comparison result \cite[Thm.~4.2.14]{bbm}. Furthermore, using (higher) classifying stacks, we introduce a ``prismatic Dieudonn\'e $F$-gauge" $\mathcal{M}(G)$
to any $p$-divisible group or finite locally free commutative group scheme $G$ of $p$-power rank. Using our construction, we generalize the classification results in \cite{alb}, by giving a new classification in the case of finite locally free group schemes, which was not addressed in loc.~cit. Our approach also gives a different argument for the classification of $p$-divisible groups that appeared in \cite{alb}. We make use the notion of prismatic $F$-gauges, which gives a refined notion of coefficients for prismatic cohomology, recently introduced in the work of Drinfeld \cite{drinew}, and Bhatt--Lurie \cite{fg}\footnote{This may be regarded as a generalization of the notion of $F$-gauges due to Fontaine--Jannsen \cite{FJ13}.}. Our proofs are geometric in nature and circumvents certain technicalities that appear in many of the previously known cases. We also give an alternative approach to prismatic $F$-gauges, aiming to demystify them. 
 
Let us recall the classical construction (see \cite[\S III]{dem}) of Dieudonn\'e modules below.

\begin{definition}[Dieudonn\'e]\label{defofdieu}
  Let $G$ be a finite commutative group scheme of $p$-power rank over $k$. Then $G$ admits a canonical decomposition $G = G^{\mathrm{uni}}\oplus G^{\mathrm{mul}},$ where $G^{\mathrm{uni}}$ is unipotent and $G^{\mathrm{mul}}$ is a local group scheme whose Cartier dual $(G^{\mathrm{mul}})^\vee$ is \'etale (see \cite[\S II, p.~39]{dem}). One defines the (contravariant) Dieudonn\'e module of $G$ in the following manner. Let us first define $M(G^{\mathrm{uni}}):= \varinjlim_{n,V}\mathrm{Hom}(G,  W_n).$ One can now define
$M(G) := M(G^{\mathrm{uni}}) \oplus M((G^{\mathrm{mul}})^\vee)^*.$   
\end{definition}{}

A uniform construction without appealing to duality was first given by Fontaine \cite{fonta} using a more complicated formal group $CW$, which maybe realized as a completion of $\varinjlim_{n,V} W_n$ in a certain sense. Another uniform construction is due to the seminal work of Berthelot--Breen--Messing \cite{bbm} in terms of crystalline Dieudonn\'e theory; they proved that $\sigma^* M(G) \simeq \mathrm{Ext}^1 (G, \mathcal{O}^{\mathrm{crys}})$, where the last Ext group is computed in the large crystalline site and $\sigma$ is the Frobenous on $W(k)$. Their proof crucially relies on Fontaine's work and in particular certain explicit computations done in the crystalline site to understand the somewhat complicated object $CW$. In \cite[Thm.~1.2]{Mon1}, it was shown that $\sigma^* M(G) \simeq H^2_{\mathrm{crys}}(BG)$, where the proof relied on the work of Berthelot--Breen--Messing.

 In this paper, we directly prove that $\sigma^* M(G) \simeq H^2_{\mathrm{crys}}(BG)$ circumventing the use of \cite{fonta} or \cite{bbm}. Instead, our techniques use cohomology of algebraic stacks and the de Rham--Witt complex \cite{Luc1}. Broadly speaking, the main new ingredient is the usage of geometric techniques such as differential forms, deformation theory in the study of Dieudonn\'e modules by means of the classifying stack $BG.$
\begin{theorem}\label{die}
   Let $G$ be a finite commutative $p$-power rank group scheme over a perfect field $k$ of characteristic $p>0.$ We have a canonical isomorphism $\sigma^*M(G) \simeq H^2_{\mathrm{crys}}(BG).$ 
\end{theorem}{}
As a corollary, we obtain a new, much shorter proof of

\begin{corollary}[{\cite[Thm.~4.2.14]{bbm}}]\label{coyi}
 Let $G$ be a finite commutative group scheme over $k$ of $p$-power rank. Then $\sigma^* M(G) \simeq \mathrm{Ext}^1_{(k/ W(k))_{\mathrm{Crys}}}(G, \mathcal{O}^{\mathrm{crys}}).$   
\end{corollary}{}
\begin{remark}
We point out that while the statement of \cref{die} appeared in the author's previous work \cite{Mon1}, the arguments for the proof in \cite{Mon1} relied crucially on the seminal work of Berthelot--Breen--Messing \cite{bbm}, and essentially directly followed from {\cite[Thm.~4.2.14]{bbm}} via a spectral sequence argument. In contrast, the proof of \cref{die} in our current paper is entirely different from \cite{Mon1}, and it does not rely on \cite{bbm}. In particular, the main ideas of the new proof involves a subtle combination of certain constructions involving de Rham--Witt complexes, some direct calculations of cohomology of classifying stacks (which uses the cotangent complex), and the usage of the $2$-stack $B^2 G$ (\cref{2-stacks}). These techniques and ideas did not appear in \cite{Mon21} (or \cite{bbm}). In fact, as a consequence of this new argument, we are able to recover one of the main theorems of \cite{bbm} as stated in \cref{coyi}. 
\end{remark}{}

Obtaining a shorter proof of the above result of Berthelot--Breen--Messing (\cref{coyi}) was one of our main motivations to reprove \cref{die} using the current approach. Based on the above two results, one may regard the classical Dieudonn\'e module as de Rham--Witt realization of $H^2_{\mathrm{crys}}(BG)$ and $\mathrm{Ext}^1_{(k/ W(k))_{\mathrm{Crys}}}(G, \mathcal{O}^{\mathrm{crys}})$ as its sheaf cohomological incarnation in the crystalline site. 

Our approach in \cref{section2} leads to further developments, where the notion of $2$-stacks invoked in \cref{2-stacks}, and the theory of prismatic cohomology (generalizing crystalline cohomology) play an important role. This is the subject of \cref{covidy}, where we study Dieudonn\'e theory in the mixed characteristic set up, for $p$-divisible groups or finite locally free commutative group schemes of $p$-power rank over a quasisyntomic ring $S$. Let us recall the definition of such rings.

\begin{definition}[{see \cite[\S~4]{BMS2}}]
A ring $S$ is called quasisyntomic if $S$ is $p$-complete with bounded $p^\infty$-torsion and if the cotangent complex $\mathbb{L}_{S/\mathbb{Z}_p}$ (see \cite{Ill72}) has $p$-complete Tor-amplitude in homological degrees $[0,1].$ \end{definition}{}

\begin{example}Quasisyntomic rings form a fairly large class for rings. Any $p$-complete local complete intersection Noetherian ring is quasisyntomic. Any perfectoid ring, or $p$-completion of a smooth algebra over a perfectoid ring is also quasisyntomic. The class of rings in \cref{greentea} below form an important class, as they form a basis for the quasisyntomic topology on quasisyntomic rings (see \cite[\S~4.4]{BMS2}). \end{example}{}

\begin{definition}[{see \cite[Rmk.~4.22]{BMS2}}]\label{greentea}
 A ring $R$ is called quasiregular semiperfectoid if it is quasisyntomic and a quotient of some perfectoid ring.   
\end{definition}{}

In \cite{alb}, the authors introduced the notion of admissible prismatic Dieudonn\'e modules over $S$ in order to obtain a classification of $p$-divisible groups over $S$ (see \cite[Thm.~1.4.4]{alb}). However, the category of admissible prismatic Dieudonn\'e modules end up not being flexible enough to obtain a classification of finite locally free group schemes. In order to obtain a classification of the latter, one needs to work with a suitable category. Motivated by work of Lau \cite{lau}, in the desired category, one would like to have a notion of divided Frobenius in an appropriately filtered set up, which does not appear directly in a prismatic Dieudonn\'e module. Further, in \cite[Thm.~0.3]{kisin}, Kisin obtained a classification of $p$-divisible groups over $\mathcal{O}_K$ in terms of lattices in crystalline Galois representations with Hodge--Tate weights in $[0,1].$ In their recent work \cite[Thm.~1.2]{bscam}, Bhatt and Scholze explained how to recover lattices in crystalline Galois representations from prismatic $F$-crystals, or from the more refined notion of prismatic $F$-gauges (see \cite[Thm.~6.6.13]{fg}). Unwrapping the notion of a prismatic $F$-gauge algebraically, one naturally obtains certain filtered objects, and maps that play the role of divided Frobenii (see \cref{lion}). Moreover, as explained in \cref{comppp}, one can view admissible prismatic Dieudonn\'e modules as prismatic $F$-gauges with Hodge--Tate weights in $[0,1].$ As we will show, it turns out that prismatic $F$-gauges also provide the desired category for classifying locally free commutative group schemes of $p$-power rank. 

The category of prismatic $F$-gauges over $S$ is a rather elaborate piece of structure: they are defined to be the derived category of quasicoherent sheaves on certain $p$-adic formal stacks introduced in \cite{drinew} and \cite{fg}, called the ``syntomification of $S$'' and denoted by $\Spf (S)^{\syn}.$ By quasisyntomic descent (see \cite[Lem.~4.28, Lem.~4.30]{BMS2}), in order to understand $\Spf (S)^{\mathrm{syn}}$ one can restrict attention to the case when $S = R$ for a quasiregular semiperfectoid algebra $R$ (see \cref{presen}). In this case, we concretely spell out the stack $\mathrm{Spf} (R)^{\mathrm{syn}}$ to give a sense of what kind of objects we are working with.
\begin{construction}[{see \cite[Rmk.~5.5.16]{fg}}]\label{sc}Let $R$ be a quasiregular semiperfectoid algebra (qrsp). Define $\mathrm{Spf}(R)^{\Prism}:= \mathrm{Spf}(\Prism_R),$ where $\Prism_R$ is the prism associated to $R.$ Define $$\mathrm{Spf}(R)^{\mathrm{Nyg}}:= \mathrm{Spf} \left(\bigoplus_{i \in \mathbb Z} \mathrm{Fil}^i_{\mathrm{Nyg}} \Prism_R \left \{ i \right \} \right)/\mathbb{G}_m,$$ where $\Prism_R \left \{i \right \}$ denotes the Breuil--Kisin twist. The Nygaard filtration provides a map of graded rings $$\bigoplus_{i \in \mathbb Z} \mathrm{Fil}^i_{\mathrm{Nyg}} \Prism_R \left \{ i \right \} \to \bigoplus_{i \in \mathbb Z} \Prism_R \left \{ i \right \},$$ which induces a map $$\mathrm{can}: \mathrm{Spf}(R)^{\Prism} \to \mathrm{Spf}(R)^{\mathrm{Nyg}}.$$ Also, the divided Frobenius defines a map of graded rings
$$\oplus_{i \in \mathbb Z} (\varphi_i):  \bigoplus_{i \in \mathbb Z} \mathrm{Fil}^i_{\mathrm{Nyg}} \Prism_R \left \{ i \right \} \to \bigoplus_{i \in \mathbb Z} \Prism_R \left \{ i \right \},$$ which induces a map $$\mathrm{\varphi}: \mathrm{Spf}(R)^{\Prism} \to \mathrm{Spf}(R)^{\mathrm{Nyg}}.$$One defines
$$\mathrm{Spf}(R)^{\mathrm{syn}}:=  \mathrm{coeq} \left(\xymatrix{
  \mathrm{Spf}(R)^\Prism\ar@<1ex>[r]^{\mathrm{can}}\ar@<-1ex>[r]_{\varphi} & \mathrm{Spf}(R)^{\mathrm{Nyg}}
 }\right).$$ In this set up, the graded module $\bigoplus_{i \in \mathbb Z} \mathrm{Fil}^{i-1}_{\mathrm{Nyg}} \Prism_R \left \{ i-1 \right \}$ defines a vector bundle on $\mathrm{Spf}(R)^{\w{Nyg}},$ which descends to a vector bundle on $\mathrm{Spf}(R)^{\w{syn}}$-- this will be called the Breuil--Kisin twist and will be denoted by $\mathcal{O}\left \{-1 \right \}.$ For any $M \in \mathrm{Spf}(S)^{\w{syn}},$ we use $M \left \{-n\right \}$ to denote $M \otimes_{\mathcal{O}} \mathcal{O}\left\{-1\right \}^{\otimes n}.$  
\end{construction}{}

\begin{remark}
    Note that in \cite[Prop.~5.5.2]{fg}, a variant of the above construction is given without the Breuil--Kisin twists. However, we prefer the version with the Breuil--Kisin twists for the introduction, since it makes the divided Frobenius operator, and the line bundle $\mathcal{O}\left \{-1\right \}$ more transparent. The equivalence of these two constructions is discussed in \cite[Rmk.~5.5.16]{fg}.
\end{remark}

\begin{definition}[Prismatic $F$-gauges, see \cite{fg}]
The category of prismatic $F$-gauges over $S$, denoted as $F\text{-}\mathrm{Gauge}^{\mathrm{}}_{\Prism}(S)$, is defined to be the derived category of quasicoherent sheaves on $\Spf (S)^\syn$.
\end{definition}{}

In \cref{korma1}, we will explain how to think of prismatic $F$-gauges as a certain filtered module over a filtered ring equipped with some extra structure in the quasisyntomic site of $S$; this approach is closer in spirit to the notion of $F$-gauges in \cite{FJ13} and has certain technical advantages. This approach will also crucially be used in our paper. However, for the introduction, it would be easier to explain our constructions from the stacky point of view.

In \cref{prismaticdie}, we begin by defining the prismatic Dieudonn\'e $F$-gauge $\mathcal{M}(G)$ associated to a $p$-divisible group $G$ over $S$. Our definition is directly motivated by \cref{die}, which we roughly explain below.

\begin{definition}[Prismatic Dieudonn\'e $F$-gauge of a $p$-divisible group]\label{pdfg}
Let $G$ be a $p$-divisible group over $S.$ By functoriality of the syntomification construction, we obtain a natural map
$$v: BG^{\mathrm{syn}} \to S^{\mathrm{syn}}.$$ The prismatic Dieudonn\'e $F$-gauge of $G$, denoted as $\mathcal{M}(G)$, is defined to be $R^2v_* \mathcal{O}_{BG^{\mathrm{syn}}} \in F\text{-}\mathrm{Gauge}^{\mathrm{}}_{\Prism}(S).$       
\end{definition}{}

\begin{remark}\label{introgreentea}
 Since the theory of $t$-structures on $p$-complete derived category of (derived) $p$-adic formal stacks is not well-behaved in general, the above definition of $\mathcal{M}(G)$ should be regarded as heuristic and should be interpreted in the sense of \cref{prelimdef} (also see \cref{introdd}). In fact, some work is necessary to prove that $\mathcal{M}(G)$ is well-defined as a prismatic $F$-gauge; this is proven in \cref{dualo1}. If $G$ is of height $h,$ then $\mathcal{M}(G)$ is a vector bundle of rank $h$ when viewed as a quasicoherent sheaf on $\Spf(S)^\syn.$  
\end{remark}
We show the following:
\begin{theorem}[\cref{mainthm2}]\label{sna}
Let $\mathrm{BT}(S)$ denote the category of $p$-divisible groups over a quasisyntomic ring $S$ and $F\text{-}\mathrm{Gauge}^{\mathrm{vect}}_{\Prism}(S)$ denote the category of vector bundles on $\Spf(S)^\syn.$ The prismatic Dieudonn\'e $F$-gauge functor$$\mathcal{M}:  \mathrm{BT}(S)^\mathrm{op} \to F\text{-}\mathrm{Gauge}^{\mathrm{vect}}_{\Prism}(S),$$ determined by $G \mapsto \mathcal{M}(G)$ is fully faithful. The essential image of $\mathcal{M}$ is the full subcategory of vector bundles with Hodge--Tate weights in $[0,1]$.   \end{theorem} 

The above result gives a different proof of the fully faithfulness result of \cite[Thm.~1.4.4]{alb}. Our approach uses the formalism of quasi-coherent sheaves on $S^{\mathrm{syn}},$ and the dualizability of $\mathcal{M} (G)$ along with the compatibility with Cartier duality (\cref{grading}). In order to prove that $\mathcal{M}(G)$ is a vector bundle, we give a geometric argument by directly computing the cotangent complex $\mathbb{L}_{BG}$ (\cref{cotangent}). Under the equivalence (see \cref{comppp}) of admissible prismatic Dieudonn\'e modules \cite[Thm.~1.4.4]{alb} and vector bundles on $\Spf (S)^\syn$ with Hodge--Tate weights in $[0,1],$ \cref{sna} is equivalent to \cite[Thm.~1.4.4]{alb}.

We extend our approach using prismatic $F$-gauges to a classification of finite locally free commutative group schemes $G$ of $p$-power rank as well. However, in mixed charactertistic, under the presence of $p$-torsion, it turns out that $R^2v_*\mathcal{O}_{BG^\syn}$ as in \cref{pdfg} is not the right invariant. This is a ``pathology" that does not occur when $S$ has characteristic $p$ (see \cref{pathology}). To resolve this, we use the ``higher Eilenberg–MacLane stack" $B^2 G^\syn.$

Below, we let $\mathrm{FFG}(S)$ denote the category of finite locally free commutative group scheme of $p$-power rank over $S.$ Note that in order to describe a prismatic $F$-gauge over $S,$ one can restrict to attention to quasiregular semiperfectoid algberas; more precisely, we have (see \cref{si})
\begin{equation}\label{introtec}
    F\text{-Gauge}_{\Prism}(S) \simeq \lim_{\substack{S \to R; \\R~\text{is qrsp}}} F\text{-Gauge}_{\Prism}(R).
\end{equation}{}

\begin{definition}[Prismatic Dieudonn\'e $F$-gauge of locally free group schemes]\label{ecc222}
Let $S$ be a quasisyntomic ring and let $G \in \mathrm{FFG}(S).$ Let $S \to R$ be a map such that $R$ is qrsp. By functoriality of the syntomification construction, we have a natural map
$$v: B^2G_R^{\mathrm{syn}} \to R^{\mathrm{syn}}.$$
The association $$R \mapsto (\tau_{[-2,-3]} Rv_* \mathcal{O})[3] \in D_{\mathrm{qc}}(R^\syn)$$ determines (via \cref{introtec}) an object of $F\text{-Gauge}_{\Prism}(S)$ that we call prismatic Dieudonn\'e $F$-gauge of $G$ and denote it by $\mathcal{M}(G).$
\end{definition}{}
\begin{remark}
A description similar to \cref{ecc222} using $B^2 G^{\syn}$ in the case when $G$ is a $p$-divisible group is still compatible with \cref{pdfg}. See \cref{ecc22} and \cref{ecc2}. 
\end{remark}
\begin{remark}
 Similar to \cref{introgreentea}, the above construction in \cref{ecc222} should be regarded as heuristic, and the precise construction is given in \cref{cafe4} ({cf}.~\cref{stella1}). In fact, some work is necessary to prove that $\mathcal{M}(G)$ is well-defined as a prismatic $F$-gauge; this is proven in \cref{fffff}. In \cref{dualoo}, we show that $\mathcal{M}(G)$ is dualizable as a prismatic $F$-gauge and satisfies a compatibility formula with Cartier duality:
\begin{equation}\label{fd1}
 \mathcal{M}(G)^* \left \{-1 \right \} [1]\simeq  \mathcal{M}(G^\vee).    
\end{equation}{}

\end{remark}{}
We show the following:

\begin{theorem}[\cref{mainthm22}]\label{mi}
Let $S$ be a quasisyntomic ring. The prismatic Dieudonn\'e $F$-gauge functor $$\mathcal{M}:  \mathrm{FFG}(S)^{\mathrm{op}} \to F\text{-}\mathrm{Gauge}_{\Prism}(S),$$ 
determined by $G \mapsto \mathcal{M}(G)$ is fully faithful.\end{theorem} 

The proof of \cref{mi}, as well as the compatibility with Cartier duality (\cref{dualoo}) uses \cref{earth}, which expresses cohomology with coefficient in a group scheme $G$ as cohomology of $\mathcal{M}(G^\vee)\left \{1 \right \}$ viewed as a quasicoherent sheaf on $\Spf(S)^\syn,$ i.e.,
\begin{equation}\label{fckd}
    R\Gamma_{\mathrm{qsyn}}(S, G) \simeq R\Gamma (\mathrm{Spf}(S)^\syn, \mathcal{M}(G^\vee) \left \{1 \right \}).
\end{equation}{}
A formula as above is of independent interest, e.g., see \cite[1.1.6]{ches} for applications in purity. See \cref{jac} for a similar result for $p$-divisible groups. In \cref{laststatement}, we give a description of the essential image of the embedding of \cref{mi}.

In \cref{finalprop}, we show that the \'etale realization functor $\Te$ for prismatic $F$-gauges (\cite[Cons.~6.3.1]{fg}) can be used to recover Galois representations associated to $p$-divisible groups or finite locally free commutative group schemes of $p$-power rank $G$ from their prismatic Dieudonn\'e $F$-gauge $\mathcal{M}(G).$ More precisely, \begin{equation}\label{fd2}
\Te(\mathcal{M}(G)\left \{1 \right \}) \simeq G^\vee_{\eta}.    
\end{equation}Results in this spirit (but in different framework) appeared in \cite{gal3}, \cite{gal2}, \cite{gal1} (see also \cite{breuilgal}, \cite{zinkgal}, \cite{zinkgal2}, \cite{kisgal}, \cite{kisgal2}, \cite{alb}). Our proof uses the constructions of this paper as well as the almost purity theorem \cite[Thm.~1.10]{pspaces}.

\begin{remark}
Note that since every $G \in \mathrm{FFG}(S)$ can be Zariski locally expressed as a kernel of isogeny of $p$-divisible groups, \emph{a priori} it seems plausible (cf.~\cite{kisin} for the case of $S=\mathcal{O}_K$) that one can deduce that $\mathrm{FFG}(S)$ embeds in the $\infty$-category $F\text{-}\mathrm{Gauge}_{\Prism}(S)$ from \cref{sna} and \cref{comppp} by using descent. However, one runs into proving independence of certain choices of open covers made in this approach. Note that this issue does not arise in Kisin's argument in the case of $\mathcal{O}_K$, because one does not encounter nontrivial open covers; but it is an issue in our generality to even define a functor this way. More crucially, one does not get a direct, explicit description of the functor $\mathcal{M}(G)$ in this approach. 

One of the key aspects of our approach is to provide a direct construction of the prismatic Dieudonn\'e $F$-gauge $\mathcal{M}(G)$ (\cref{ecc222}), and give a direct approach to proving \cref{mi}. Additionally, proof of the other results, such as the Cartier duality formula \cref{fd1}, the expression of cohomology with coefficients in group schemes in terms of $\mathrm{Spf}(S)^\syn$ \cref{fckd}, and the construction of Galois representations from $\mathcal{M}(G)$ \cref{fd2} all directly rely on the construction of $\mathcal{M}(G)$ as in \cref{ecc222}.
\end{remark}{}

\begin{remark}
Since the stack $\mathrm{BT}^h_n$ of $n$-truncated Barsotti--Tate groups of height $h$ is a smooth stack (\cite{luci}, \cite[Thm.~2.1]{lauuu}), it seems plausible to use \cref{mi} to obtain a classification when the base is only assumed to be $p$-complete. In forthcoming work of Gardner--Madapusi \cite{madapusi}, the authors prove that the assignment that sends $X$ to vector bundles on $X^\syn \otimes \mathbb{Z}/p^n$ with Hodge--Tate weights in $[0,1]$ is representable by a \textit{smooth} $p$-adic formal stack over $\mathbb{Z}_p$-- this is used in their paper to obtain a classification of $n$-truncated Barsotti--Tate groups in a more general situation. In a joint work of Madapusi and the author \cite{madapusi2}, we hope to extend these to a more general classification of finite locally free commutative group schemes of $p$-power rank as well.   
\end{remark}{}

\subsection{Notations and conventions}\label{notconven}
\begin{enumerate}
    \item We will use the language of $\infty$-categories as in \cite{Lur09}, more specifically, the language of stable $\infty$-categories \cite{luriehigher}. For an ordinary commutative ring $R,$ we will let $D(R)$ denote the derived $\infty$-category of $R$-modules, so that it is naturally equipped with a $t$-structure and $D_{\ge 0}(R)$ (resp. $D_{\le 0} (R)$) denotes the connective (resp. coconnective) objects, following the homological convention. We will let $\mathcal{S}$ denote the $\infty$-category of spaces, or anima. All tensor products are assumed to be derived.

\item For any presentable, stable, symmetric monoidal $\infty$-category $\mathcal{C},$ we will denote its associated $\infty$-category of $(\mathbb Z, \ge )$-indexed filtered objects (i.e., one has \linebreak $ \cdots \to \Fil^{i+1} \to \Fil^{i} \to \cdots$) by $\Fil (\mathcal{C})$.  Note that $\Fil (\mathcal{C})$ is again a presentable stable $\infty$-category, which can be regarded as a symmetric monoidal $\infty$-category under the Day convolution (see e.g., \cite{fild}). Using \cite[\S~3]{luriehigher}, one can define the category $\mathrm{CAlg}(\mathrm{Fil}(\mathcal{C}))$, which may be called the $\infty$-category of filtered commutative algebra objects of $\mathcal{C}.$ Given any $A \in \mathrm{CAlg}(\mathrm{Fil}(\mathcal{C})),$ one can define the category $\mathrm{Mod}_A (\mathrm{Fil}(\mathcal{C}))$, which maybe called the $\infty$-category of filtered modules over the filtered commutative ring $A.$

\item Let $\mathcal{C}$ be any Grothendieck site and let $\mathcal{D}$ be any presentable $\infty$-category. Then one can define the category of ``sheaves on $\mathcal{C}$ with values in $\mathcal{D}$" denoted by $\mathrm{Shv}_{\mathcal D}(\mathcal C)$ as in \cite[Def.~1.3.1.4]{spectra}. This agrees with the usual notion of sheaves when $\mathcal{D}$ is a $1$-category. Let $\mathrm{PShv}_{\mathcal{D}}(\mathcal C):= \mathrm{Fun}(\mathcal{C}^{\mathrm{op}}, \mathcal{D}).$ There is a natural inclusion functor $\mathrm{Shv}_{\mathcal D}(\mathcal C) \to \mathrm{PShv}_{\mathcal{D}}(\mathcal C),$ that admits a left adjoint, which we will call sheafification and denote by $ \mathrm{PShv}_{\mathcal{D}}(\mathcal C) \ni \mathcal{F} \mapsto \mathcal{F}^\sharp.$ For an ordinary ring $R$, the classical (triangulated) derived category obtained from complexes of sheaves of $R$-modules on $\mathcal{C}$ is the homotopy category of the full subcategory of hypercomplete objects of $\mathrm{Shv}_{D(R)}(\mathcal{C})$ (see \cite[Cor.~2.1.2.3]{spectra}).

\item Let $\mathcal{C}$ be any Grothendieck site and $\mathcal{D}$ be any presentable, stable, symmetric monoidal $\infty$-category. Then $\mathrm{Shv}_{\mathcal{D}}(\mathcal{C})$ is naturally a presentable, stable, symmetric monoidal $\infty$-category. One can define the category $\mathrm{CAlg}(\mathrm{Fil}(\mathrm{Shv}_{\mathcal{D}}(\mathcal{C})))$, which may be called the $\infty$-category of sheaves on $\mathcal{C}$ with values in $\mathrm{CAlg}(\Fil (\mathcal{D})).$ For $A \in \mathrm{CAlg}(\mathrm{Fil}(\mathrm{Shv}_{\mathcal{D}}(\mathcal{C}))),$ one can define the category $\mathrm{Mod}_A(\mathrm{Fil}(\mathrm{Shv}_{\mathcal{D}}(\mathcal{C}))),$ which may be called sheaf of filtered modules over $A$ (on the site $\mathcal{C}$).

\item Let $R$ be any ordinary commutative rings. Then the $\infty$-category of $\mathcal{S}$-valued presheaves on the category of affine schemes over $R$ that satisfies fpqc descent will be called the $\infty$-category of (higher) stacks over $R.$
For a commutative group scheme $G$ over $R$, we let $BG$ denote the classifying stack and $B^n G$ denote the $n$-stack $K(G,n),$ which are all examples of (higher) stacks. The notion of $p$-adic formal stacks is slightly different and appears in \cref{nota1}, which uses the notion of animated rings; the latter being the $\infty$-category obtained from the category of simplicial commutative rings by inverting weak equivalences.

\item Let $R$ be an ordinary commutative ring and $I$ be a finitely generated ideal of $R.$ In this situation, one can define $D(R)^{\wedge}_I$ to be the full subcategory of $D(R)$ spanned by derived $I$-complete objects of $R$ (see \cite[\S~3.4]{proetale}). The embedding $D(R)^{\wedge}_I \to D(R)$ preserves limits, and therefore is a right adjoint. Its left adjoint will be called derived $I$-completion functor. For $M, N \in D(R)^{\wedge}_I,$ we let $M \widehat{\otimes}_R N$ denote the derived $I$-completion of $M \otimes_R N.$

\item We work with a fixed prime $p.$ We will let $R\Gamma_{\w{crys}}(\cdot)$ denote derived crystalline cohomology, which reduces to derived de Rham cohomology $\mathrm{dR}(\cdot)$ modulo $p$ (see \cite{Bha12}). Although the definition of derived crystalline cohomology is apriori quite different, it agrees with the site theoretic definition of crystalline cohomology for a large class of schemes (\cref{comparederordinary}). For this reason, and to simplify notions, we simply use $R\Gamma_{\w{crys}}$ to denote derived crystalline cohomology throughout this paper. We freely use the quasisyntomic descent techniques using quasiregular semiperfectoid algebras introduced in \cite{BMS2}. We also freely use the formalism of (absolute) prismatic cohomology developed by Bhatt--Scholze \cite{BS19}, as well as the stacky approach to prismatic cohomology developed by Drinfeld \cite{prismatization} and Bhatt--Lurie \cite{BhaLu}, \cite{BhaLur2}, \cite{blprep}. At the moment, our main reference for working with prismatic $F$-gauges is based on Bhatt's lecture notes \cite{fg}. We will give a presentation of this theory in \cref{korma1} from a slightly different perspective that is more similar to \cite[\S~1.4]{FJ13} (see also \cite[Rmk.~3.4.4]{fg}).

\item 
Let $S$ be a quasisyntomic ring and let $(S)_\mathrm{qsyn}$ denote the quasisyntomic site of $S$ (see \cite[Variant.~4.33]{BMS2}). For two abelian sheaves $\mathcal{F}$ and $\mathcal G$ on $(S)_\qsyn,$ we let $\mathcal{E}xt^i_{(S)_\qsyn} (\mathcal F, \mathcal G)$ denote the sheafified Ext groups. If $Y$ is a stack over $S,$ we let $\mathcal{H}^i_{(S)_\qsyn}(Y, \mathcal F)$ be sheafification of the abelian group valued functor that sends $$(S)_\qsyn \ni A \mapsto H^i(Y|_{(A)_\qsyn} , \mathcal F|_{(A)_\qsyn}).$$ When $\mathcal{F} = \Prism_{(\cdot)}$ (see \cref{useintronotation}), we use $\mathcal{H}^i_{\Prism} (Y)$ to denote $\mathcal{H}^i_{(S)_\qsyn}(Y, \Prism_{(\cdot)})$.

\item Let $G$ be a $p$-divisible group over a 
quasisyntomic ring $S.$ In this set up, by ($p$-completely) faithfully flat descent (see \cite[Rmk.~4.9]{BMS1} for $i=0$), for $n \ge 0,$ we can regard the group scheme of $p^n$-torsion of $G$, denoted by $G[p^n]$, as a representable abelian sheaf on $(S)_\qsyn.$ In this situation, the collection of $G[p^n]$'s naturally defines an ind-object, as well as a pro-object in $(S)_\qsyn.$ We will again use $G$ to denote $\mathrm{colim}\,G[p^n]$ as an object of $(S)_\qsyn.$ We will use $T_p(G)$ to denote $\lim G[p^n] \in (S)_\qsyn,$ and call it the Tate module of $G.$ In $(S)_\qsyn,$ one has an exact sequence of abelian sheaves $0 \to T_p(G) \to \lim_{p} G \to G \to 0.$ Note that since $G$ is $p$-divisible, the map $p: G \to G$ is a surjection on $(S)_\qsyn$ (e.g., see \cite[Rmk.~4.2.2]{alb}). Therefore, since the quasisyntomic topos is replete (see \cite[Def.~3.1.1]{proetale}, \cite[Ex.~1.5]{MR100}), it follows that $R\lim_p G$ (and consequently $R\lim G[p^n]$) is discrete. This implies that the derived $p$-completion of the abelian sheaf $G$ is isomorphic to $T_p(G)[1]$ in the derived category of abelian sheaves on $(S)_\qsyn.$ We let $\mathbb{Z}_p(1)$ denote the Tate module of of the $p$-divisible group $\mu_{p^\infty}$.

\item 
All group schemes appearing in this paper are commutative unless otherwise mentioned. We use $\mathrm{FFG}(S)$
 to denote the category of finite, locally free, commutative group schemes of $p$-power rank over a ring $S.$
 The category $\mathrm{FFG}(S)$ is self dual to itself, induced by the functor $G \mapsto G^\vee,$
 where $G^\vee$ denotes the Cartier dual of $G.$ We will let $\mathbb G_a$ denote the additive group scheme, $\mathbb G_m$ denote the multiplicative group scheme and $W$ denote the group scheme underlying $p$-typical Witt vectors.

\item Typically, we write $A^*$ to denote a (co)chain complex, $\Fil^\bullet B$ or $B^\bullet$ to denote a filtered object, and $C^{[\bullet]}$ to denote a (co)simplicial object. For a filtered object $B^\bullet,$ we let $\gr^\bullet B$ denote the associated graded object.
 \end{enumerate}{}

\subsection*{Acknowledgements} I would like to thank Johannes Anschütz, Bhargav Bhatt, Vladimir Drinfeld, Luc Illusie, Dmitry Kubrak, Arthur-César Le Bras, Shizhang Li, Jacob Lurie, Keerthi Madapusi, Akhil Mathew and Peter Scholze for helpful conversation related to the content of this paper. In particular, I am grateful to Bhargav Bhatt who suggested the use of classifying stacks in \cite{Mon1}, to Vladimir Drinfeld for an invitation to speak about \cite{Mon1} at Chicago, and to Luc Illusie for many detailed comments and suggestions on a first draft of this paper. Thanks are also due to an anonymous referee for many detailed comments on a previous version of this paper.

During the preparation of this article, I was supported by the University of British Columbia, Vancouver the Institute for Advanced Study, Princeton, and a start up grant from Purdue University.

\newpage

\section{{Crystalline Dieudonn\'e theory and classifying stacks}}\label{section2}
In this section, we prove \cref{die} (see \cref{mainthm1}) and deduce \cite[Thm.~4.2.14]{bbm} due to Berthelot--Breen--Messing (see \cref{yng}). One of the key ingredients in our proof of \cref{die} is \cref{mainconst1} based on the de Rham--Witt complex. \cref{mainconst1} below extends a construction from \cite[\S~6]{Luc1} to stacks, which was originally introduced by Illusie to produce torsion in crystalline cohomology by using the de Rham--Witt complex. 

\subsection{Some constructions involving classifying stacks} In this subsection, we will discuss \cref{mainconst1}, \cref{cons22} and \cref{multicons}, which play a very important role later on. We will also discuss certain calculations involving cohomology of classifying stacks (\cref{compute} and \cref{impmu}). We start by recalling some notations and definitions in the context of crystalline cohomology.

\begin{notation}
 In this section, we fix a prime $p$ and a perfect field $k$ of characteristic $p>0.$ For a ring $R,$ we let $W(R)$ denote the ring of $p$-typical Witt vectors and $W_n(R)$ denote its $n$-truncated variant. We let $\sigma: W(k) \to W(k)$ denote the Frobenius operator on Witt vectors. We let $(k/ W(k))_{\text{Crys}}$ denote the (big) crystalline topos of $k$ (for the Zariski topology) and $(k/ W_n(k))_{\text{Crys}}$ denote its $n$-truncated variant (see \cite[\S 1.1.1]{bbm}). For any scheme $X$ over $k,$ as in {\cite[1.1.4.5]{bbm}}, one can canonically associate an object in $(k/ W(k))_{\text{Crys}}$ or $(k/ W_n(k))_{\text{Crys}}$ that we may also denote by $X$, when no confusion is likely to arise.
\end{notation}

\begin{notation}
 The association $A \mapsto W(A)$ determines a contravariant functor from the category of affine schemes over $k$ to the category of abelian groups, which is represented by an affine group scheme over $k$ that we denote by $W.$ We let $W_n$ denote the $n$-truncated variant of the group scheme $W.$ We refer to \cite[\S III]{dem} for more details on these Witt group schemes.  
\end{notation}{}

\begin{definition}\label{defsheafy}
 In this section, we will work with derived crystalline cohomology (see \cite{Bha12}). In particular, let $(k)_\mathrm{qsyn}$ denote the quasisyntomic site of $k$ (see \cite[Variant.~4.33]{BMS2}). Using \cite[Ex.~5.12]{BMS2} (and arguing derived modulo $p$), it follows that the functor that sends $(k)_\mathrm{qsyn} \ni A \mapsto R\Gamma_{\mathrm{crys}}(A)$ is a \emph{sheaf} with values in the $\infty$-category $D(W(k))^{\wedge}_p.$ Therefore, it extends to a limit preserving functor $$ R\Gamma_{\mathrm{crys}}(\cdot): \mathrm{Shv}_{\mathcal{S}}((k)_\qsyn)^{\mathrm{op}} \to D( W(k))^{\wedge}_p,$$ where $\mathrm{Shv}_{\mathcal{S}} ((k)_\qsyn)$ denotes the category of sheaves on $(k)_\qsyn$ with values in the $\infty$-category $\mathcal{S}.$ This extends the formalism of derived crystalline cohomology to (higher) stacks over $k$, which will be used in this section. For any (higher) stack $\mathcal{Y},$ we will call $ R\Gamma_{\mathrm{crys}}(\mathcal{Y})$ the derived crystalline cohomology of $\mathcal{Y}.$ We let $R\Gamma_{\mathrm{crys}}(\mathcal{Y}/W_n(k))$ denote $R\Gamma_{\mathrm{crys}}(\mathcal{Y}) \otimes_{W(k)} W_n(k).$
\end{definition}{}

\begin{remark}\label{comparederordinary}
By \cite[Thm.~3.27]{Bha12}, for an lci $k$-algebra $A,$ the derived crystalline cohomology of $A$ agrees with crystalline cohomology computed in the topos $(k/ W(k))_{\text{Crys}}$.    
\end{remark}{}

\begin{construction}\label{mainconst1}
We discuss the construction of a map 
\begin{equation}\label{korma}
 T: \varinjlim_{n,V}R\Gamma_{\mathrm{}}(\mathcal{Y}, W_n\mathcal{O}) \to \sigma_*R\Gamma_{\mathrm{crys}} (\mathcal{Y})[1]   
\end{equation}for a stack $\mathcal{Y}$ over a perfect field $k$ of characteristic $p>0.$ 
To this end, suppose that $R$ is a finitely generated polynomial algebra over $k.$ 
Note that we have an exact sequence $$0 \to W(R) \xrightarrow{V^n} W(R)\to  W_n(R) \to 0.$$ Let $W\Omega^*_R$ denote the de Rham Witt complex of $R.$ As in \cite[\S6,~6.1.1]{Luc1} we have another complex $$W\Omega^*_R(n):= W(R) \xrightarrow{F^n d} W\Omega^1_R \xrightarrow{d} W\Omega^2_R \xrightarrow{d}  \ldots.$$ There is a natural map $W\Omega^*_R \to W\Omega^*_R(n)$ induced by $V^n$ (note that $FdV= d$). Therefore, we obtain an exact sequence of complexes
\begin{equation}\label{illusiepre}
   0 \to W\Omega^*_R \to  W\Omega^*_R (n) \to W_n(R) \to 0. 
\end{equation}Taking direct limits produce an exact sequence of complexes \begin{equation}\label{illusie}
   0 \to W\Omega^*_R \to \varinjlim_{n,V} W\Omega^*_R (n) \to \varinjlim_{n,V}W_n(R) \to 0. 
\end{equation}Passing to the derived category, using the comparison of de Rham--Witt and crystalline cohomology and keeping track of the $W(k)$-module structure, we obtain the map $$\varinjlim_{n,V}W_n(R) \to \sigma_*R\Gamma_{\mathrm{crys}} ({R})[1].$$Using \cite[Example~2.13]{sanath} and left Kan extension from the category of finitely generated polynomial algebras over $k$ to all $k$-algebras, we obtain a map 
\begin{equation}\label{triviaa}
\varinjlim_{n,V}W_n(R) \to \sigma_*R\Gamma_{\mathrm{crys}} ({R})[1]
\end{equation}for any $k$-algebra $A.$
Using \cref{triviaa} and universal property of colimits, for a stack $\mathcal{Y}$ over $k$, we obtain a natural map
$$\varinjlim_{n,V} R\Gamma(\mathcal{Y}, W_n\mathcal{O}) \to \sigma_* R\Gamma_{\mathrm{crys}}(\mathcal{Y})[1].$$
This gives the desired map \cref{korma}.
\end{construction}{}

\begin{remark}\label{conj}
Let us give further explanation regarding \cref{mainconst1}, which is based on the construction of the natural map $W_n \mathcal{O}_X \to R\Gamma_{\mathrm{crys}}(X)[1].$ Note that such a map arises from a map $W_n \mathcal{O}_X \to R\Gamma_{\mathrm{crys}}(X)/p^n$ via composition with $R\Gamma_{\mathrm{crys}}(X)/p^n \to R\Gamma_{\mathrm{crys}}(X)[1].$ The desired map $W_n \mathcal{O}_X \to R\Gamma_{\mathrm{crys}}(X)/p^n$ can be thought of as simply arising from the more general conjugate filtration $\mathrm{Fil}^{\bullet}_{\mathrm{conj}}R\Gamma_{\mathrm{crys}}(X)/p^n$ whose graded pieces can be understood by the $\varphi^n$-linear higher Cartier isomorphism 
\begin{equation}\label{illusiee}
 \mathbb{L}W_n \Omega^i_X[-i] \simeq \mathrm{gr}^{i}_{\mathrm{conj}}\left(R\Gamma_{\mathrm{crys}}(X)/p^n \right).   
\end{equation}Left hand side of \cref{illusiee} comes from animation of the de Rham--Witt theory as in \cite[Rmk.~Def.~2.11]{sanath}. The filtration $\mathrm{Fil}^{\bullet}_{\mathrm{conj}}R\Gamma_{\mathrm{crys}}(X)/p^n$ is obtained from animating the canonical filtration when $X$ is smooth and using \cite[Prop.~1.4]{IR}, which gives the desired description of the graded pieces. We thank Illusie for pointing out the isomorphism \cref{illusiee}.
\end{remark}{}

We will apply the map $T$ from \cref{mainconst1}~\cref{korma} in the case $\mathcal{Y}= BG.$ To this end, let us note a few general remarks about cohomology of $BG.$ Let $(k/ W(k))_{\mathrm{Crys}}$ denote the big crystalline topos, and $(k/W_n(k))_{\mathrm{Crys}}$ denote the $n$-truncated variant. Let $\mathcal{F}$ be any object in the derived category of quasisyntomic sheaves. By Cech descent along the effective epimorphism $* \to BG,$ we obtain $$ R \Gamma(BG, \mathcal{F}) \simeq  \varprojlim_{m \in \Delta} R\Gamma (G^{[m]}, \mathcal{F}),$$ where $G^{[\bullet]}$ is the Cech nerve of $* \to BG.$ By the sheafiness property from \cref{defsheafy}, applying this to $\mathcal{F} = R\Gamma_{\mathrm{crys}}(\cdot)$ (resp. $R\Gamma_{\mathrm{crys}}(\cdot)/p^n$), we obtain via Cech descent
$$R \Gamma_{\mathrm{crys}}(BG, \mathcal{F}) \simeq  \varprojlim_{m \in \Delta} R\Gamma_{\mathrm{crys}}(G^{[m]}, \mathcal{F})$$ $$ \simeq R\mathrm{Hom}_{(k/ W(k))_{\mathrm{Crys}}} \left(\varinjlim_{m \in \Delta^{op}}\mathbb{Z}[G^{[m]}], \mathcal{O}^{\mathrm{crys}}\right).$$ Let us define $\mathbb{Z}[BG]:=\varinjlim_{m \in \Delta^{op}}\mathbb{Z}[G^{[m]}].$ This way, we get a convergent spectral sequence with $E_2$-page (see \cite[Tag 07A9]{stacks}) \begin{equation}\label{mainspec}
    E_2^{i,j}= \text{Ext}_{(k/ W(k))_{\mathrm{Crys}}}^i (H^{-j}(\mathbb{Z}[B {G}]), \mathcal{O}^{\mathrm{crys}})\implies H^{i+j}_{\w{crys}}(BG)
\end{equation} (resp. a similar spectral sequence in $(k/W_n(k))_{\mathrm{Crys}}$ converging to $H^*_{\mathrm{crys}}(BG/W_n)$).

\begin{lemma}\label{torsion}
Let $G$ be a group scheme of order $p^m.$ Then for any $i>0,$ the group $H^i_{\mathrm{crys}}(BG)$ is killed by a power of $p.$
\end{lemma}{}
\begin{proof}
 This is a consequence of the above (convergent, first quadrant) $E_2$-spectral sequence and the fact that an $n$-torsion ordinary abelian group $T,$ the group homology $H_i (T, \mathbb Z) = H_i (\mathbb Z [BT])$ is $n$-torsion.
\end{proof}{}
Note that, by definition, for any stack $\mathcal{Y}$, we have an exact sequence
\begin{equation}\label{univcoeff}
    0 \to H^i_{\mathrm{crys}} (\mathcal{Y})/p^n \to H^i_{\mathrm{crys}}(\mathcal{Y}/W_n) \to H^{i+1}_{\mathrm{crys}}(\mathcal{Y})[p^n] \to 0.\end{equation}

\begin{lemma}\label{vanish}
Let $G$ be a group scheme of order $p^m.$ Then $H^1_{\mathrm{crys}}(BG) = 0$    
\end{lemma}{}

\begin{proof}
 We choose a large enough $n$ such that $H^1_{\mathrm{crys}}(BG)[p^n] = H^1_{\mathrm{crys}}(BG)$ and apply \cref{univcoeff} for $i=0.$   
\end{proof}{}

\begin{lemma}\label{derham}
  Let $G$ be a finite group scheme of $p$-power order. Then for all $n \gg 0,$ the map $H^1_{\mathrm{crys}} (BG/W_n) \to H^2_{\mathrm{crys}}(BG)$ is an isomorphism.
\end{lemma}
\begin{proof}
    We can choose $n$ large enough such that $H^2_{\mathrm{crys}}(BG)$ is killed by $p^n$ and apply \cref{univcoeff} for $i=1$ along with the previous lemma.   
\end{proof}{}

\begin{lemma}\label{basechange} Let $k'/k$ be an extension of perfect fields. Then $$R\Gamma_{\mathrm{crys}} (BG_{k'}) \simeq R\Gamma_{\mathrm{crys}}(BG) \otimes_{W(k)} W(k').$$
\end{lemma}
\begin{proof}
 By derived $p$-completeness, one can reduce modulo $p.$ By the de Rham--crystalline comparison, it would be enough to prove that $R\Gamma_{\mathrm{dR}} (BG_{k'}) \simeq R\Gamma_{\mathrm{dR}}(BG) \otimes_{k} k'.$ Since $R\Gamma_{\mathrm{dR}}(G_{k'}) \simeq R\Gamma_{\mathrm{dR}}(G) \otimes_{k} k'$ the result follows from descent along $* \to BG$ and using the fact that totalization of bounded below cochain complexes commute with filtered colimits (see, e.g., \cite[Lem.~2.2.27]{MR4969359}).
\end{proof}{}

\begin{lemma}\label{rightexact}Let $0 \to G' \to G \to G'' \to 0$ be an exact sequence of finite group schemes over $k.$ Then we have an exact sequence $0 \to H^2_{\mathrm{crys}} (BG'') \to H^2_{\mathrm{crys}} (BG) \to H^2_{\mathrm{crys}} (BG')$ of $W(k)$-modules.
    
\end{lemma}{}

\begin{proof}
We pick an $n$ large enough so that $H^2_{\w{crys}}(BG), H^2_{\w{crys}}(BG'), H^2_{\w{crys}}(BG'')$ are all killed by $p^n.$ Then the desired exactness of the maps follow from \cref{derham} and the fact that $H^1_{\mathrm{crys}}(BH/W_n) \simeq \mathrm{Hom}_{(k/W_n(k))_{\mathrm{Crys}}}(H, \mathcal{O}^{\mathrm{crys}}).$ To see the last isomorphism, one may use the analogue of the spectral sequence \cref{mainspec} in $(k/W_n(k))_{\mathrm{Crys}}$ and the isomorphism $H^{-1}(\mathbb{Z}[BH]) \simeq H$. The latter isomorphism is a consequence of the fact that for every ordinary abelian group $A$, the group cohomology $H_1 (A, \mathbb{Z})$ is naturally isomorphic to $A$.
\end{proof}{}

\begin{lemma}
Let $G$ be a finite group scheme of $p$-power rank over a perfect field $k.$ Then $H^2_{\mathrm{crys}}(BG)$ is a finite length $W(k)$-module.    
\end{lemma}{}

\begin{proof}
 By \cref{basechange}, one may assume that $k$ is algebraically closed. In that case, one can argue by induction using \cref{rightexact}, which reduces us to the statement for the simple group schemes $\mathbb{Z}/p, \alpha_p$ and $\mu_p$ (see \cite[Prop.~2, p.~39]{dem}). This follows from \cref{compute} below.
\end{proof}{}

\begin{proposition}[{cf.~\cite{totaro},~\cite{ben}}]\label{compute}
 $H^2_{\mathrm{crys}}(BH)=k$ when $H$ is either $\mathbb{Z}/p, \mu_p$ or $\alpha_p.$   
\end{proposition}{}
\begin{proof}
By base change, it is enough to argue when $k = \mathbb F_p.$ First note that for $n \gg 0,$ we have isomorphisms $\mathrm{Hom}_{(k/W_n(k))_{\mathrm{Crys}}}(H, \mathcal{O}^{\mathrm{crys}}) \simeq H^1_{\w{crys}}(BH/W_n) \simeq H^2_{\mathrm{crys}}(BH).$ Since $H$ is $p$-torsion, it follows by looking at the left term above that $H^2_{\w{crys}}(BH)$ is $p$-torsion. Therefore (by \cref{univcoeff} and \cref{vanish}), we have an isomorphism 
\begin{equation}
 H^1_\mathrm{dR}(BH) \simeq H^2_{\mathrm{crys}}(BH).   
\end{equation}{}
The proof below will freely use the calculation of cotangent complex $\mathbb{L}_{BH/k}$, for which we refer to \cite[\S~VII, 3.1.2]{Ill72} and \cite[Thm.~4.1]{ben}.

\begin{enumerate}
    \item $H= \mathbb{Z}/p:$ Since $\mathbb{Z}/p$ is an \'etale group scheme, it follows that $\mathbb{L}_{B\mathbb{Z}/p/k} =0.$ Therefore, (e.g., by using the conjugate filtration) it follows that $R\Gamma_{\mathrm{dR}}(B\mathbb{Z}/p) \simeq R\Gamma (B\mathbb{Z}/p, \mathcal{O}).$ Using descent along $* \to B \mathbb{Z}/p$, one directly sees that $H^*(B \mathbb{Z}/p, \mathcal{O}) \simeq H^*(\mathbb Z/p, k)$, where the latter denotes group cohomology. Thus, $H^1_{\mathrm{dR}}(B\mathbb{Z}/p) \simeq H^1 (\mathbb{Z}/p, k) \simeq k.$

\item $H= \mu_p:$ Chern class of the $p$-torsion line bundle corresponding to the map $B\mu_p \to B\mathbb{G}_m$ gives a natural map $$c: k \to H^2_{\mathrm{crys}}(B\mu_p).$$ By the proof of \cite[Prop.~11.1]{totaro} (also see the proof of \cite[Thm.~4.6 (5)]{ben}), the map $c$ is an isomorphism.

\item $H= \alpha_p:$ see \cite[Prop.~4.12]{ben}.
\end{enumerate}This ends the proof.\end{proof}{}

\begin{proposition}\label{impmu}
 Let $k$ be a perfect field. Then as $W(k)$-modules, we have a canonical isomorphism $H^2_{\mathrm{crys}}(B\mu_{p^m}) \simeq \sigma^* (W(k)/p^m)$ (where the latter denotes the $W(k)$-module structure via extension of scalars).   
\end{proposition}{}

\begin{proof}


Chern class of the line bundle corresponding to the map $B\mu_{p^m} \to B\mathbb{G}_m$ defines a canonical $p^m$-torsion class on $H^2_{\w{crys}}(B \mu_{p^m}).$ Since $H^1_{\mathrm{crys}}(B \mu_{p^m})=0,$ we obtain a canonical map
$$\sigma^* W(k) \oplus \sigma^* W(k)/p^m[-2] \to R\Gamma_{\w{crys}}(B\mu_{p^m}) $$ in the derived category of $W(k)$-modules. Let $C$ denote the cofiber. It will be enough to prove that $C \in D^{\ge 3}(W(k)).$ Since $C$ is derived $p$-complete it is enough to prove the same for $C/p.$ By base change, we may now assume that $k= \mathbb{F}_p$. It is therefore enough to prove that the induced map 
\begin{equation}\label{inducedmap}
  \mathbb{F}_p \oplus \mathbb{F}_p[-1] \oplus \mathbb{F}_p[-2] \to R\Gamma_{\w{dR}}(B \mu_{p^m}) 
\end{equation} has cofiber in $D^{\ge 3}(\mathbb{F}_p)$.

To this end, we will first use the conjugate filtration to compute the target of the above map. Since $\mu_{p^m}$ lifts to $W_2(k)$ as a group scheme along with a lift of the Frobenius (which is just multiplication by $p$), the conjugate filtration splits (see \cite[Lem.~4.2]{ben} for more details, which applies here since $\mu_{p^m}$ is a diagonalizable group scheme), and one has 
$$R\Gamma_{\mathrm{dR}}(B\mu_{p^m}) \simeq \oplus_{i \ge 0} R\Gamma (B\mu_{p^m}, \wedge^i \mathbb{L}_{B\mu_{p^m}/\mathbb{F}_p}[-i]).$$Now, one computes using the coLie complex (\cite[\S~VII, 3.1.2]{Ill72}) that $\mathbb{L}_{B\mu_{p^m}/\mathbb{F}_p} \simeq \mathcal{O} \oplus \mathcal{O}[-1]$, so that we have $\wedge^i \mathbb{L}_{B\mu_{p^m}/\mathbb{F}_p} \simeq \mathcal{O}[-i+1] \oplus \mathcal{O}[-i].$ Further, we have $H^{>0} (B \mu_{p^m}, \mathcal{O})=0:$ to see the vanishing, note that by Cartier duality, quasi-coherent sheaves on $B
\mu_p$ identify with $\mathbb{Z}/p^m\mathbb{Z}$-graded vector spaces, and thus the global section functor (which corresponds to taking the summand of weight $0$) is exact. Combining these facts, one concludes that $H^i_{\mathrm{dR}}(B\mu_{p^m})$ is $1$-dimensional as an $\mathbb{F}_p$-vector space for all $i$.

Now, to show that \cref{inducedmap} has cofiber in $D^{\ge 3}(\mathbb{F}_p)$, it suffices to show that the induced map on $i$-th cohomology is an isomorphism for $0 \le i \le 2.$ Since both the source and target of these maps are $1$-dimensional, it suffices to show that the maps are nonzero. In order to show the latter, by composing with the natural map $R\Gamma_{\mathrm{dR}}(B\mu_{p^m})\to R\Gamma_{\mathrm{dR}}(B\mu_{p})$ induced by the natural inclusion $\mu_p \to \mu_{p^m}$ and using functoriality, we may assume that $m=1.$ However, in that case, when $i=2$, by construction, we recover the composite map $\mathbb{F}_p \xrightarrow{c} H^2_{\mathrm{crys}}(B\mu_p) \simeq H^2_{\mathrm{dR}}(B\mu_p), $ which is an isomorphism by the proof of \cref{compute} (2). When $i=1$, we recover the composite map $\mathbb{F}_p \xrightarrow{c} H^2_{\mathrm{crys}}(B\mu_p)[p] \simeq H^1_{\mathrm{dR}}(B\mu_p)$, which is again an isomorphism since $c$ is. The case when $i \le 0$ follows directly. This finishes the proof.
\end{proof}{}

\begin{construction}\label{cons22}

Now we can use the map $T$ from \cref{mainconst1} to obtain a map $$C: \varinjlim_{n}H^1 (BG,  W_n) \to \sigma_*H^2_{\mathrm{crys}} (BG).$$ Note that $\varinjlim_{n}H^1(BG, W_n) = \varinjlim_{n}\mathrm{Hom}(G,  W_n).$ Therefore, when $G$ is unipotent, by \cref{defofdieu}, we get a natural map \begin{equation}\label{unimap}
    C^{\mathrm{uni}}: \sigma^*M(G) \to H^2_{\mathrm{crys}}(BG).
\end{equation} Suppose that $G$ is a local group scheme over $k$ of order $p^k$  whose Cartier dual $G^\vee$ is \'etale. In \cref{multicons}, we will produce a natural map
\begin{equation}\label{multimap}
   C^{\mathrm{mult}}: H^2_{\w{crys}}(BG) \to \sigma^* M(G). 
\end{equation}
\end{construction}{}
To this end, in the remarks below, we recall certain relevant constructions.

\begin{remark}[Duality]
    Let $\mathrm{Mod}_{W(k)}^{\mathrm{fl}}$ denote the category of finite length $W(k)$-modules. The functor that sends $M \mapsto M^*:=\mathrm{Hom}_{W(k)}(M, W(k)[\frac{1}{p}]/ W(k))$ induces an antiequivalence of $\mathrm{Mod}_{W(k)}^{\mathrm{fl}}.$ To see this, note that there is a canonical evaluation map $M \to (M^*)^*$, which is an isomorphism, as one directly checks by reducing to cyclic modules. This duality also extends to the set up of Dieudonn\'e modules whose underlying $W(k)$-module is finite length (see \cite[p.~66]{dem}).
\end{remark}{}

\begin{remark}[Galois descent]
    Let $\overline{k}$ be an algebraic closure of $k.$ Let $(\mathrm{Mod}_{W(\overline{k})}^{\mathrm{fl}})^{\mathrm{Gal}(\overline{k}/k)}$ denote the category of finite length $W(\overline{k})$-modules equipped with a (semilinear) action of $\mathrm{Gal}(\overline{k}/k).$ By Galois descent, we obtain an equivalence of categories $(\mathrm{Mod}_{W(\overline{k})}^{\mathrm{fl}})^{\mathrm{Gal}(\overline{k}/k)} \simeq \mathrm{Mod}_{W(k)}^{\mathrm{fl}}$ induced by the functors that send $M \mapsto M \otimes_{W(k)} W(\overline{k}) \in (\mathrm{Mod}_{W(\overline{k})}^{\mathrm{fl}})^{\mathrm{Gal}(\overline{k}/k)}$, and $N \to N^{\mathrm{Gal}(\overline{k}/k)}.$ To check this, one needs to prove that the natural map $N^{\mathrm{Gal}(\overline{k}/k)} \otimes_{W(k)} W(\overline{k}) \to N$ is an isomorphism. Suppose first that $N$ is $p$-torsion. Then $N$ corresponds to a vector bundle in the \'etale site of $\mathrm{Spec}\, k,$ which must be trivial by descent; i.e., $N \simeq M \otimes_k \overline{k}$ for some finite dimensional $k$-vector space $N,$ which implies that the desired map is an isomorphism. The case of general $N$ follows by considering the (finite) $p$-adic filtration on $N$ and using that $H^{>0}(\mathrm{Gal}(\overline{k}/k), \overline{k})\simeq H^{>0}_{\mathrm{\acute{e}t}} (\mathrm{Spec}\,k, \mathcal{O})=0.$

\end{remark}

\begin{construction}\label{multicons}
Let $G$ be such that $G^\vee$ is \'etale. Note that for $n \gg 0,$ we have $$M(G_{\overline{k}}^\vee) \simeq \mathrm{Hom}(G_{\overline{k}}^\vee, \mathbb{Z}/p^n)\otimes_{\mathbb{Z}_p}W(\overline{k}) \simeq \mathrm{Hom}(\mu_{p^n}, G_{\overline{k}})\otimes_{\mathbb{Z}_p} W(\overline{k}),$$ where the last step follows from Cartier duality. By functoriality of crystalline cohomology and \cref{impmu}, we obtain a map $$\mathrm{Hom}(\mu_{p^n}, G_{\overline{k}})\otimes_{\mathbb{Z}_p} W(\overline{k}) \to \mathrm{Hom}_{W(\overline{k})}^{\mathrm{Gal}(\overline{k}/k)} (\sigma_* H^2_{\w{crys}} (BG_{\overline{k}}), W(\overline{k})/p^n);$$ the latter denotes maps taken in $(\mathrm{Mod}_{W(\overline{k})}^{\mathrm{fl}})^{\mathrm{Gal}(\overline{k}/k)}.$ Taking Galois fixed points, we obtain a map $M(G^\vee) \to (\sigma_* H^2_{\w{crys}} (BG))^*.$ Applying duality now produces a map $$\sigma_* H^2_{\w{crys}}(BG) \to M(G^\vee)^* \simeq M(G).$$ This constructs $C^{\mathrm{mult}}$ as desired in \cref{multimap}.   
\end{construction}

\subsection{Recovering the result of Berthelot--Breen--Messing}
Having performed the above constructions, in this subsection, we will now use them to prove one of our main theorems (\cref{mainthm1}), and use it to deduce a result due to Berthelot--Breen--Messing (\cref{yng}).

\begin{lemma}\label{isommup}
 The map $C^{\mathrm{mult}}$ in \cref{multicons} is an isomorphism for $G= \mu_p.$   
\end{lemma}

\begin{proof}
To check isomorphism, we can assume that $k$ is algebraically closed, where it follows from tracing throgh the maps in \cref{multicons} as we explain. Note that the map $\mathrm{Hom}(\mu_p, \mu_p) \to \mathrm{Hom}_{W(k)}(H^2_{\mathrm{crys}}(B\mu_p), H^2_{\mathrm{crys}}(B\mu_p))$ induced from functoriality of crystalline cohomology sends identity to identity, and is nonzero in particular. This gives an isomorphism 
$$\mathbb{Z}/p\mathbb{Z} \otimes_{\mathbb{Z}_p} W(k) \simeq  \mathrm{Hom}(\mu_p, \mu_p) \otimes_{\mathbb{Z}_p} W(k) \to \mathrm{Hom}_{W(k)}(H^2_{\mathrm{crys}}(B\mu_p), H^2_{\mathrm{crys}}(B\mu_p)) $$ where the latter (nonzero) map is seen to be isomorphism by \cref{compute}, since its source and target are both $1$-dimensional over $k$. Since Cartier dual of $\mu_p$ is the group scheme $\mathbb{Z}/p$, by \cref{multicons}, the map
$$M(\mathbb{Z}/p) \to (\sigma_*H^2_{\mathrm{crys}}(B\mu_p))^*$$ is an isomorphism. Since $C^{\mathrm{mult}}$, by construction, is dual to the latter map, it is an isomorphism as desired.
\end{proof}{}

\begin{lemma}\label{wlaf}
 Let $G$ be unipotent. The map $C^{\mathrm{uni}}$ is injective.   
\end{lemma}{}

\begin{proof}
By \cref{mainconst1}, we have the following commutative diagram where the bottom row is exact:
\begin{center}
\begin{tikzcd}
                                       &  & {H^1(BG, \varinjlim W_n)} \arrow[d,"\mathrm{id}"] \arrow[rr, "C^{\mathrm{uni}}"] &  & H^2_{\mathrm{crys}}(BG) \arrow[d] \\
{H^1(BG, \varinjlim_{V} W)} \arrow[rr] &  & {H^1(BG, \varinjlim W_n)} \arrow[rr]           &  & {H^2(BG, W)}                     
\end{tikzcd}
\end{center}{}
Therefore, to prove the injectivity of $C^{\mathrm{uni}},$ it suffices to show that \linebreak $H^1(BG, \varinjlim_{V}W)= \mathrm{Hom}(G, \varinjlim_{V}W)=0.$ But this follows because $W$ is $V$-torsion free, and $G$, being finite and unipotent, is killed by a power of the Verschiebung operator $V$ by \cite[Thm. (v)~p.~38]{dem}.
\end{proof}{}

\begin{proposition}\label{mainthm1}
Let $G$ be a finite commutative $p$-power rank group scheme over $k.$ We have a canonical isomorphism $\sigma^*M(G) \simeq H^2_{\mathrm{crys}}(BG).$
\end{proposition}{}

\begin{proof}
    We use the natural maps $C^{\mathrm{uni}}$ (\cref{mainconst1}) and $C^{\mathrm{mult}}$ (\cref{multicons}) constructed before and argue separately. To check that these natural maps are isomorphisms, we may assume that $k$ is algebraically closed (\cref{basechange}). We have an exact sequence $0 \to G' \to G \to G'' \to 0,$ where one may assume that $G'$ is simple. Since $k$ is algebraically closed, $G'$ must be either $\mathbb{Z}/p, \mu_p$ or $\alpha_p.$
    
    First, we will argue that $C^{\mathrm{mult}}$ is an isomorphism, so in particular we assume that $G^\vee$ is \'etale. In this case, $G' \simeq \mu_p$.
    By \cref{rightexact}, we have the following diagram where the rows are exact:
\begin{center}
    \begin{tikzcd}
0 \arrow[r] & H^2_{\mathrm{crys}}(BG'') \arrow[rr] \arrow[d, "{C}^{\mathrm{mult}}"] &  & H^2_{\mathrm{crys}}(BG) \arrow[rr] \arrow[d, "{C}^{\mathrm{mult}}"] &  & H^2_{\mathrm{crys}}(BG') \arrow[d, "{C}^{\mathrm{mult}}"] \\
        0 \arrow[r]    & \sigma^*M(G'') \arrow[rr]                      &  & \sigma^*M(G) \arrow[rr]                      &  & \sigma^*M(G')               
\end{tikzcd}
\end{center}{}
Since $k$ is assumed to be algebraically closed, and $G$ is diagonalizable (see \cite[p.~36]{dem}), the injection $\mu_p \to G$ induces a surjective map $\mathrm{Hom}(G, \mathbb{G}_m) \to \mathrm{Hom}(\mu_p, \mathbb{G}_m)$ on the group of characters. In particular, the canonical inclusion $\mu_p \to \mathbb{G}_m$ admits a factorization as $\mu_p \to G \to \mathbb{G}_m$. This produces a line bundle on $BG$ whose Chern class in $H^2_{\mathrm{crys}}(BG)$, under the map $H^2_{\mathrm{crys}}(BG) \to H^2_{\mathrm{crys}}(B\mu_p)$ goes to the Chern class of the line bundle corresponding to $B\mu_p \to B\mathbb{G}_m$, which is a nonzero element of $H^2_{\mathrm{crys}}(B\mu_p)$ by the proof of \cref{compute} (2). Therefore, the natural map $H^2_{\mathrm{crys}}(BG) \to H^2_{\mathrm{crys}}(BG')$ is surjective, since its target is $1$-dimensional by \cref{compute}. By \cref{isommup}, the upper and lower horizontal maps in the right of the above diagram are both surjective. Therefore, the claim follows by \cref{isommup}, induction on the length of $G$ and the five lemma. 

Now we will show that $C^\mathrm{uni}$ is an isomorphism, so in particular we assume that $G$ is unipotent, and $G'$ is either $\alpha_p$ or $\mathbb{Z}/p$. Similar to the above paragraph, we have the following diagram where the rows are exact:
\begin{center}
    \begin{tikzcd}
0 \arrow[r] & \sigma^*M(G'') \arrow[rr] \arrow[d, "C^{\mathrm{uni}}"]   &  & \sigma^*M(G) \arrow[rr] \arrow[d, "C^{\mathrm{uni}}"]  &  & \sigma^*M(G') \arrow[d, "C^{\mathrm{uni}}"]  \\
0 \arrow[r] &  H^2_{\mathrm{crys}}(BG'') \arrow[rr] &  & H^2_{\mathrm{crys}}(BG) \arrow[rr] &  & H^2_{\mathrm{crys}}(BG')
\end{tikzcd}
\end{center}By the exactness of $M(\cdot)$ (see e.g., \cite[\S V~4.3~(b)]{MR0302656}), the map $\sigma^* M(G) \to \sigma^* M(G')$ is surjective. By \cref{compute} and \cref{wlaf} the right vertical map is an isomorphism. Therefore, the lower horizontal map in the right is also surjective, and the claim follows from induction on the length of $G$. This finishes the proof.
\end{proof}

\begin{proposition}\label{2-stacks}
   Let $G$ be a finite commutative $p$-power rank group scheme over $k.$ We have a canonical isomorphism 
$H^2_{\mathrm{crys}}(BG) \simeq H^3_{\mathrm{crys}}(B^2G).$
\end{proposition}{}

\begin{proof}
Using descent along $* \to B^2 G,$ we obtain an $E_1$-spectral sequence
$$E_1^{i,j}=H^j_{\w{crys}}((BG)^i) \implies H^{i+j}_{\w{crys}}(B^2G).$$ The claim now follows by using this spectral sequence and \cref{vanish}; the latter lemma implies that $H^1_{\mathrm{crys}}(BG \times BG) \simeq H^1_{\mathrm{crys}}(B (G \times G)) = 0$.   
\end{proof}{}

\begin{proposition}\label{niceday}
 For any group scheme $G$ over $k,$ one has $$H^3_{\mathrm{crys}}(B^2 G) \simeq \mathrm{Ext}^1_{(k/ W(k))_{\mathrm{Crys}}}(G, \mathcal{O}^{\mathrm{crys}}).$$   
\end{proposition}{}
\begin{proof}
By applying descent along $* \to B^2G,$ similar to \cref{mainspec}, we obtain an $E_2$-spectral sequence $$E_2^{i,j}= \text{Ext}_{(k/ W(k))_{\mathrm{Crys}}}^i (H^{-j}(\mathbb{Z}[B^2 {G}]), \mathcal{O}^{\mathrm{crys}})\implies H^{i+j}_{\w{crys}}(B^2G).$$ The claim now follows from some classical functorial calculation of low degree homology groups of Eilenberg--MacLane spaces. Namely, we use the fact that $H^{-1}(\mathbb{Z}[B^2 {G}])=0,$ $H^{-2}(\mathbb{Z}[B^2 {G}])= G$ (by e.g., Hurewicz theorem relating homotopy and homology) and $H^{-3} (\mathbb{Z}[B^2 {G}])= 0;$ the latter vanishing follows from the classical fact that the homology group $H_3(B^2A, \mathbb{Z})=0$ for any ordinary abelian group $A$. In order to see the latter fact, one may use the spectral sequence \cite[Prop.~1.1]{h3breen}, and the connectivity estimate in \cite[Prop.~3.1.3]{Lur2}, which implies that $E^2_{p,q}$ in the relevant spectral sequence in \cite[Prop.~1.1]{h3breen} vanishes for $p+q=3$.
\end{proof}{}


Combining \cref{mainthm1}, \cref{2-stacks} and \cref{niceday}, we obtain
\begin{corollary}[Berthelot--Breen--Messing]\label{yng}
 Let $G$ be a finite commutative group scheme over $k$ of $p$-power rank. Then $\sigma^* M(G) \simeq \mathrm{Ext}^1_{(k/ W(k))_{\mathrm{Crys}}}(G, \mathcal{O}^{\mathrm{crys}}).$ 
\end{corollary}{}

We end this section with some corollaries of \cref{mainthm1} for $p$-divisible groups that will be useful later.

\begin{corollary}\label{washing1}
Let $G$ be a $p$-divisible group over a perfect field $k$. We have a canonical isomorphism $\sigma^* M(G) \simeq H^2_{\mathrm{crys}}(BG)$. 
\end{corollary}
\begin{proof}
 Let $G_n := G[p^n]$. Since $R^1 \varprojlim H^1_{\mathrm{crys}}(BG_n)=0$ (\cref{vanish}), the claim follows from taking inverse limits and using that $M(G) \simeq \varprojlim M(G_n).$ 
\end{proof}{}

\begin{corollary}\label{makelight}
 Let $G$ be a $p$-divisible group over a perfect field $k$. We have a canonical isomorphism $ H^2_{\mathrm{crys}}(BG)/p \simeq H^2_{\mathrm{dR}}(BG).$    
\end{corollary}

\begin{proof}
 Since we have a short exact sequence 
 $$0 \to H^2_{\mathrm{crys}}(BG)/p \to H^2_{\mathrm{dR}}(BG) \to H^3 _{\mathrm{crys}}(BG)[p] \to 0, $$ the claim follows from the vanishing $H^3_{\mathrm{crys}}(BG)=0$ (see \cite[Prop.~3.34]{Mon1}).
\end{proof}

\begin{remark}
The vanishing $H^3_{\mathrm{crys}}(BG)=0$ used above will be generalized and proven in a different way in \cref{locfree}. Note that taking $\gr^0$ of the Hodge--filtration produces a natural map $H^2_{\mathrm{dR}}(BG) \to H^2 (BG, \mathcal{O})$. The following observation will be important later.    
\end{remark}
\begin{lemma}\label{uselatere}
 Let $G$ be a $p$-divisible group over an algebraically closed field $k$ of characteristic $p$. The natural map $H^2_{\mathrm{dR}}(BG) \to H^2 (BG, \mathcal{O})$ is surjective.   
\end{lemma}{}
 \begin{proof}First, we note that $H^1 (BG, \mathcal{O})=0$. This follows from the natural isomorphism $H^1 (BG, \mathcal{O}) \simeq \varprojlim \mathrm{Hom}(G[p^n], \mathbb{G}_a)$ and \cite[Prop.~2.2.1~(c)]{luci}, which implies that $\mathrm{Hom}(G[p^n], \mathbb{G}_a)$ is pro-zero. Applying the same technique as in the proof of \cref{2-stacks} and \cref{niceday} shows that $H^2(BG, \mathcal{O}) \simeq H^3 (B^2G, \mathcal{O}) \simeq \mathrm{Ext}^1(G, \mathbb{G}_a)$, where the latter $\mathrm{Ext}^1$ denotes extensions as fpqc sheaves. This implies that for two $p$-divisible groups $G_1$ and $G_2$, there is a natural isomorphism $H^2 (BG_1, \mathcal{O}) \oplus H^2 (BG_2, \mathcal{O}) \simeq H^2(B(G_1 \times G_2), \mathcal{O}).$ Since by construction, $M(G_1) \oplus M(G_2) \simeq M(G_1 \times G_2)$, \cref{makelight} implies that there is a natural isomorphism $H^2 _{\mathrm{dR}}(BG_1 ) \oplus H^2 _{\mathrm{dR}}(BG_2) \simeq H^2_{\mathrm{dR}}(B(G_1 \times G_2)).$ Now, as $k$ is algebraically closed, for any given $p$-divisible group $G$, there exists a $p$-divisible group $G'$ such that $G \times G' \simeq A[p^\infty]$ for some abelian variety $A$ (see the proof of \cite[Prop.~3.34]{Mon1}). Therefore, in order to check the surjectivity claim of the lemma, one may without loss of generality assume that $G \simeq A[p^\infty]$ for some abelian variety $A.$ Note that the natural map $A[p^n] \to A$ induces natural isomorphisms $H^2_{\mathrm{dR}}(BA)\simeq H^2_{\mathrm{dR}}(BA[p^\infty])$ and $H^2(BA, \mathcal{O})\simeq H^2(BA[p^\infty], \mathcal{O})$. Thus, it suffices to show the surjectivity of the map $H^2_{\mathrm{dR}}(BA) \to H^2(BA, \mathcal{O})$ (see the proof of \cite[Prop.~3.1, Lem.~3.1]{actaar}, cf.~proof of \cref{vectorbundleabsch}). Further, under the natural isomorphisms $H^2_{\mathrm{dR}}(BA) \simeq H^1_{\mathrm{dR}}(A)$ and $H^2(BA, \mathcal{O}) \simeq H^1(A, \mathcal{O})$, it suffices to show that the natural map $H^1_{\mathrm{dR}} (A) \to H^1(A, \mathcal{O})$ is surjective. However, the latter is classical; in fact, one has a natural short exact sequence $0 \to H^0 (A, \Omega^1_A) \to H^1_{\mathrm{dR}}(A) \to H^1 (A, \mathcal{O}) \to 0$ by the degeneration of Hodge--de Rham spectral sequence for abelian varieties (see \cite[Prop.~5.1]{MR241435}).
\end{proof}{}

\newpage

\section{Dieudonn\'e theory and prismatic $F$-gauges}\label{covidy}
In this section, we expand the ideas and constructions involving crystalline cohomology of classifying stacks appearing in \cref{section2} to mixed characteristic situations by using prismatic cohomology. Roughly speaking, we prove that the Dieudonn\'e module functor induced by $G \mapsto H^2_{\Prism}(BG)$ (where the latter denotes absolute prismatic cohomology) give a fully faithful functor. However, one cannot expect this functor to be fully faithful without carefully analyzing what other extra data on  $H^2_{\Prism}(BG)$ one needs to remember. Let us first take a step back to give an intuitive explanation of our perspective on Dieudonn\'e theory taken in this paper.

By Pontryagin duality, for a finite discrete abelian group $G,$ the functor $G \mapsto \mathrm{Hom}(G, \mathbb{S}^1)$ gives an equivalence of categories. Note that in this case, since $\mathbb{S}^1 = K(\mathbb{Z}, 1)$ and $G$ is finite, one has a natural isomorphism $\mathrm{Hom}(G, \mathbb{S}^1) = H^1 (BG, \mathbb{Z}[1]) = H^2(BG, \mathbb{Z}),$ where the latter denotes Betti cohomology. Our main point here is that it is possible to reconstruct a group (or rather its Pontryagin dual) from the cohomology of $BG.$

Now, let $G$ be a finite locally free commutative group scheme over a base scheme $S.$ By Cartier duality, the functor $G \to G^{\vee}:=\mathcal{H}om(G, \mathbb{G}_m)$ gives an antiequivalence of categories. Note that $\mathcal{H}om(G, \mathbb{G}_m)= \mathcal{H}^1(BG, \mathbb{G}_m).$ Further, from a motivic point of view, one may think of $\mathbb{G}_m = \mathbb{Z}(1)[1],$ where $\mathbb{Z}(1)$ is the Tate twist. Thus, $\mathcal{H}^1(BG, \mathbb{G}_m) = \mathcal{H}^2 (BG, \mathbb{Z}(1));$ the latter recovers the group scheme $G$ (or rather, its Cartier dual). The notion of Tate twists and other related twists would play a very important role in our approach.

Since the goal of Dieudonn\'e theory is to classify finite locally free or $p$-divisible group schemes by linear algebraic data, it is natural to look for a more linearized way to recover the Tate twist $\mathbb{Z}(1).$ To this end, we now entirely specialize to the $p$-adic set up. We take $S$ to be a $p$-complete quasisyntomic ring. The $p$-adic Tate module of $\mathbb{G}_m$ denoted by $\mathbb{Z}_p(1)$ can be understood in terms of prismatic cohomology by means of the following formula (cf.~\cite[Sec.~7.4]{BMS2}):
\begin{equation}\label{verynice}
    \mathrm{R\Gamma}(\mathcal{Y}, \mathbb{Z}_p(1)) \simeq \mathrm{Fib}\left(\mathrm{Fil}^1_{\mathrm{Nyg}}R\Gamma_{\Prism}(\mathcal{Y})\left\{1\right \} \xrightarrow{\varphi_1 - \mathrm{can}} R\Gamma_{\Prism}(\mathcal{Y})\left\{1\right \}\right),\end{equation}where the curly brackets denote the Breuil--Kisin twist, and $\mathcal{Y}$ is a stack over $S.$

\begin{definition}
 We refer to \cite{BS19} and \cite{BhaLu} for the theory of absolute prismatic cohomology (equipped with the Nygaard filtration) that will be used in this section. In particular, let $(S)_\mathrm{qsyn}$ denote the quasisyntomic site of $S$ (see \cite[Variant.~4.33]{BMS2}).
By \cite[Cor.~5.5.21]{BhaLu}, the functor that sends $(S)_{\mathrm{qsyn}} \ni A \mapsto \Fil^*_{\Nyg} R\Gamma_{\Prism} (A)$, is a sheaf with values in filtered objects of $D(\mathbb Z _p)^{\wedge}_p.$ Therefore, it extends to a functor $\Fil^\bullet_\Nyg R\Gamma_{\Prism}(\cdot): \mathrm{Shv}_{\mathcal{S}}((S)_\qsyn) \to \Fil(D(\mathbb Z _p)^{\wedge}_p),$ where $\mathrm{Shv}_{\mathcal{S}} ((S)_\qsyn)$ denotes the category of sheaves on $(S)_\qsyn$ with values in the $\infty$-category $\mathcal{S}.$ This extends the formalism of Nygaard filtered prismatic cohomology to (higher) stacks, which will be used in our paper. For any (higher) stack $\mathcal{Y},$ we will call $\Fil^\bullet _\Nyg R\Gamma_{\Prism}(\mathcal{Y})$ the Nygaard filtered prismatic cohomology of $\mathcal{Y}.$
\end{definition}{}

\begin{notation}\label{useintronotation}
  The functor that sends $(S)_{\mathrm{qsyn}} \ni A \mapsto  R\Gamma_{\Prism} (A)$ will be denoted as $\Prism_{(\cdot)}.$ Similarly, the functor that sends $(S)_{\mathrm{qsyn}} \ni A \mapsto \Fil^*_{\Nyg} R\Gamma_{\Prism} (A)$ will be denoted as $\Fil^*_\Nyg \Prism_{(\cdot)}.$
\end{notation}{}

\subsection{Some calculations of syntomic cohomology of classifying stacks}
In this subsection, we record certain calculations by specializing \cref{verynice} to $\mathcal{Y}= BG,$ that will be important later.
\begin{lemma}\label{pde}
    Let $G$ be a $p$-divisible group over a quasisyntomic ring $S$. We have $\mathcal{H}^1_{\Prism} (BG) = 0.$
\end{lemma}{}

\begin{proof}
  Note that any class in prismatic cohomology is killed quasisyntomic locally. Therefore, by applying descent along $* \to BG,$ one obtains that $\mathcal{H}^1_{\Prism} (BG) \simeq \mathcal{H}om_{(S)_{\mathrm{qsyn}}}(G, \Prism_{(\cdot)}).$ Thus it is enough to prove that $\mathcal{H}om_{(S)_{\mathrm{qsyn}}}(G, \Prism_{(\cdot)})=0$. Since $G$ is $p$-divisible, this follows because $\Prism_{(\cdot)}$ is derived $p$-complete.
\end{proof}{}

\begin{remark}\label{pathology}
 Even if we are working over a quasisyntomic base ring $S,$ for a finite locally free group scheme $G$ over $S$, the above vanishing need not hold. One can take $S$ to be a quasiregular semiperfectoid algebra such that $\Prism_S$ has $p$-torsion. Then $G = \mathbb{Z}/p\mathbb{Z}_S$ gives such an example. Let us point out that if $S$ has characteristic $p>0,$ this ``pathology" does not occur since for any quasiregular semiperfect ring $S,$ the ring $\mathbb{A}_{\mathrm{crys}}(S)$ is $p$-torsion free. It is precisely to circumvent this pathology that we will be working with $B^2G$ for finite locally free commutative group schemes, whereas for $p$-divisible groups $G,$ the stack $BG$ suffices. 

\end{remark}{}

\begin{lemma}\label{imp}
    Let $G$ be finite locally free group scheme of $p$-power rank over a quasisyntomic ring $S$. Then we have the following isomorphisms in $(S)_{\mathrm{qsyn}}$
$$\mathcal{H}^2_{(S)_{\mathrm{qsyn}}}(BG, \mathbb{Z}_p(1))\simeq \mathcal{E}xt^1_{(S)_{\mathrm{qsyn}}}(G, \mathbb{Z}_p(1)) \simeq G^{\vee}.$$ 
\end{lemma}{}
\begin{proof}
    Since $G$ is killed by a power of $p,$ it follows that the natural map $\mathbb{G}_m \to \mathbb{Z}_p(1)[1]$ induces an isomorphism $\mathcal{E}xt^1(G, \mathbb{Z}_p(1)) \simeq \mathcal{H}om(G, \mathbb{G}_m) \simeq G^\vee.$ To see that $\mathcal{H}^2(BG, \mathbb{Z}_p(1))= \mathcal{E}xt^1(G, \mathbb{Z}_p(1)),$ we use the $E_2$-spectral sequence 
$$E_2^{i,j}=\mathrm{Ext}^i_{(S)_{\mathrm{qsyn}}}(H^{-j} (\mathbb{Z}[BG]), \mathbb{Z}_p(1) ) \implies H^{i+j}_{\mathrm{qsyn}}(BG, \mathbb{Z}_p(1)).$$ Note that $H^{-1}(\mathbb{Z}[BG]) \simeq G$ and $H^{-2}(\mathbb{Z}[BG])= G\wedge G$, which follows from classical calculations of group homology in low degrees (e.g., see \cite[Thm.~3]{MR49191}). Since locally in $(S)_{\mathrm{qsyn}}$, $\mathbb{Z}_p(1)$ has no higher cohomology, it suffices to show that $\mathcal{H}om(G \wedge G, \mathbb{Z}_p(1) )=0;$ the latter follows from the fact that we have an injection
$$\mathcal{H}om (G \wedge G, \mathbb{Z}_p(1)) \to \mathcal{H}om (G \otimes G, \mathbb{Z}_p(1)) \simeq \mathcal{H}om (G, \mathcal{H}om (G, \mathbb{Z}_p(1)))$$
and $\mathcal{H}om(G, \mathbb{Z}_p(1)) \simeq \mathcal{E}xt^{-1} (G, \mathbb{G}_m) \simeq 0.$
\end{proof}{}

\begin{proposition}\label{ww}
 Let $G$ be a finite locally free group scheme of $p$-power rank over a quasisyntomic ring $S$. Then $ \mathcal{H}^2_{(S)_{\mathrm{qsyn}}}(B^2G, \mathbb{Z}_p(1))=0$ and $\mathcal{H}^3_{(S)_{\mathrm{qsyn}}}(B^2G, \mathbb{Z}_p(1)) = G^\vee.$ 
\end{proposition}{}

\begin{proof}
For any quasisyntomic sheaf of abelian groups $\mathcal{F}$ on $(S)_\qsyn$ such that $\mathcal{H}^i (*, \mathcal{F})=0$ for $i>0,$ we have $\mathcal{H}^2 (B^2G , \mathcal{F}) \simeq \mathcal{H}om(G, \mathcal{F})$ and $\mathcal{H}^3 (B^2 G , \mathcal{F})) \simeq \mathcal{E}xt^1 (G, \mathcal{F}).$

In order to see the above, we use the $E_2$-spectral sequence 
\begin{equation}\label{uuuu}
E_2^{i,j}=\mathrm{Ext}^i_{(S')_{\mathrm{qsyn}}}(H^{-j} (\mathbb{Z}[B^2G_{S'}]), \mathcal{F}) \implies H^{i+j}_{(S')_{\mathrm{qsyn}}}(B^2G_{S'}, \mathcal{F})  
\end{equation} for any quasisyntomic algebra $S'$ over $S.$ The claims now follow from the fact that $H^{-2}(\mathbb Z[B^2 G_{S'}]) = G_{S'}$ and $H^{-3}(\mathbb Z[B^2 G_{S'}])= H^{-1} (\mathbb Z[B^2 G_{S'}]) = 0$ and $H^0(\mathbb Z[B^2 G_{S'}]) \simeq \mathbb Z$ (\textit{cf}.~proof of \cref{niceday}).


 Applying the isomorphisms in the first paragraph to $\mathcal{F} = \mathbb{Z}_p(1),$ the conclusion follows from \cref{imp}.
\end{proof}{}

\begin{lemma}\label{tatemodule}
       Let $G$ be a $p$-divisible group over a quasisyntomic ring $S$. Then $$\mathcal{H}^2_{(S)_{\mathrm{qsyn}}}(BG, \mathbb{Z}_p(1)) \simeq \mathcal{E}xt^1_{(S)_{\mathrm{qsyn}}}(G, \mathbb{Z}_p(1)) \simeq T_p(G^{\vee}).$$
\end{lemma}

\begin{proof}
 This follows from the proof of \cref{imp} by taking inverse limits as we will explain. To obtain the first isomorphism, we again use the the $E_2$-spectral sequence 
$$E_2^{i,j}=\mathrm{Ext}^i_{(S)_{\mathrm{qsyn}}}(H^{-j} (\mathbb{Z}[BG]), \mathbb{Z}_p(1) ) \implies H^{i+j}_{\mathrm{qsyn}}(BG, \mathbb{Z}_p(1)).$$ Using the same steps as in the proof of \cref{imp}, and the vanishing $\mathcal{H}om(G,\mathbb{Z}_p(1))=0$ (which holds because $G$ is $p$-divisible and $\mathbb{Z}_p(1)$ is $p$-complete) we obtain $$\mathcal{H}^2_{(S)_{\mathrm{qsyn}}}(BG, \mathbb{Z}_p(1)) \simeq \mathcal{E}xt^1_{(S)_{\mathrm{qsyn}}}(G, \mathbb{Z}_p(1)).$$ To see the second isomorphism, let $G_n:= G[p^n]$. Since $G \simeq \varinjlim_{n} G_n$, we have $$\displaystyle{R\mathcal{H}om_{(S)_\qsyn}(G, \mathbb{Z}_p(1)) \simeq R\varprojlim_n R\mathcal{H}om_{(S)_\qsyn}(G_n, \mathbb{Z}_p(1))}.$$ Applying \cref{imp}, the vanishing $\mathcal{H}om(G, \mathbb{Z}_p(1)) =0$, and \cite[Prop.~3.22]{MR100} (cf. \cite[Tag~08U5]{stacks}), we obtain $\mathcal{E}xt^1_{(S)_{\mathrm{qsyn}}}(G, \mathbb{Z}_p(1)) \simeq \varprojlim_n \mathcal{E}xt^1_{(S)_{\mathrm{qsyn}}}(G_n, \mathbb{Z}_p(1)) \simeq T_p(G^{\vee})$ as desired. 
\end{proof}{}

\begin{example}\label{tatemoduleabelianscheme}
 Let $A$ be an abelian scheme over a quasisyntomic ring $S$. Let $A[p^\infty]$ denote the $p$-divisible group associated to $A$. Let $A^\vee$ denote the dual abelian scheme. Then there is a canonical isomorphism $A[p^\infty]^\vee \simeq (A^\vee[p^\infty])$. Since the $\mathbb{Z}_p(1)$ is derived $p$-complete, there is a natural isomorphism $\mathcal{E}xt^1_{(S)_{\mathrm{qsyn}}}(A, \mathbb{Z}_p(1)) \simeq \mathcal{E}xt^1_{(S)_{\mathrm{qsyn}}}(A[p^\infty], \mathbb{Z}_p(1))$ induced via the natural map $A[p^\infty] \to A$ regarded as sheaves on $(S)_\qsyn$ (see the proof of \cite[Prop.~4.39]{Mon1}). Therefore, by \cref{tatemodule}, we obtain a natural isomorphism 
$$\mathcal{E}xt^1_{(S)_{\mathrm{qsyn}}}(A, \mathbb{Z}_p(1)) \simeq \mathcal{E}xt^1_{(S)_{\mathrm{qsyn}}}(A[p^\infty], \mathbb{Z}_p(1)) \simeq T_p(A^\vee).$$
 
\end{example}{}

\begin{lemma}\label{w}
    Let $G$ be a finite locally free group scheme of $p$-power rank or a $p$-divisible group. Then $\mathcal{E}xt^{2}_{(S)_{\mathrm{qsyn}}}(G, \mathbb{Z}_p(1))=0.$
\end{lemma}{}

\begin{proof}

If $G$ is a $p$-divisible group, in the notations of the proof of \cref{tatemodule}, $G \simeq \varinjlim_n G_n$. The natural inclusion maps $G_n \to G_{n+1}$ induce surjections $\mathcal{E}xt^1_{(S)_\qsyn} (G_{n+1}, \mathbb{Z}_p(1)) \to \mathcal{E}xt^1_{(S)_\qsyn} (G_{n}, \mathbb{Z}_p(1))$ by the second isomorphism in \cref{imp} and Cartier duality. Therefore, by repleteness of $(S)_\qsyn$ (see \cref{notconven} (9)), ${R^1 \varprojlim_n\mathcal{E}xt^1_{(S)_\qsyn} (G_{n+1}, \mathbb{Z}_p(1)) \simeq 0.} $
By \cite[Prop.~3.22]{MR100}, it follows that $\mathcal{E}xt^2_{(S)_{\mathrm{qsyn}}}(G, \mathbb{Z}_p(1)) \simeq \varprojlim_n \mathcal{E}xt^2_{(S)_{\mathrm{qsyn}}}(G_n, \mathbb{Z}_p(1)).$ Thus, it is enough to prove the lemma when $G$ is a finite locally free group scheme.

The fiber of the natural map $\mathbb{G}_m \to \mathbb{Z}_p(1)[1]$ is uniquely $p$-divisible. But since $G$ is killed by a power of $p,$ we must have $\mathcal{E}xt^2(G, \mathbb{Z}_p(1)) \simeq \mathcal{E}xt^1(G, \mathbb{G}_m).$ To show vanishing of the latter, let $n$ be an integer that kills $G.$ Using the exact sequence $0 \to \mu_n \to \mathbb{G}_m  \xrightarrow{n} \mathbb{G}_m \to 0,$ one sees that any class $u \in \mathrm{Ext}^1(G, \mathbb{G}_m)$ arises from a class $v \in \mathrm{Ext}^1(G, \mu_n).$ Let us represent $v$ by an exact sequence $0 \to \mu_n \to H \to G \to 0.$ The class $u$ can be described via pushout of $\mu_n \to H$ along $\mu_n \to \mathbb{G}_m$. The class $u$ can be killed if there exists a map $H \to \mathbb{G}_m$ such that the composition $\mu_n \to H \to \mathbb{G}_m$ is the natural inclusion. However, by Cartier duality, we have a surjection of group schemes $H^\vee \to \mu_n^\vee;$ i.e., we have a surjection $\mathcal{H}om (H, \mathbb{G}_m) \to \mathcal{H}om(\mu_n, \mathbb{G}_m)$ of quasisyntomic sheaves. Therefore, the natural inclusion $\mu_n \to \mathbb{G}_m$ can be lifted quasisyntomic locally and thus the class $u$ can also be killed quasisyntomic locally.
\end{proof}{}

\begin{proposition}\label{finiteflat0rev}
  For a finite locally free group scheme $G$ over a quasisyntomic ring $S$ for which $\mathcal{H}^1_{\Prism}(BG)=0$, we have the following exact sequence of abelian sheaves in $(S)_{\mathrm{qsyn}}$:
\begin{equation}\label{finiteflat0}
  0 \to  G^\vee \to \mathcal{H}^2_{(S)_{\mathrm{qsyn}}} (BG, \mathrm{Fil}_{\mathrm{Nyg}}^1 \Prism_{(\cdot)}) \left \{1\right \} \xrightarrow{\varphi_1 - \mathrm{can}} \mathcal{H}^2_{\Prism} (BG)\left \{ 1 \right \} \to 0.
\end{equation}   
\end{proposition}{}

\begin{proof}
Applying \cref{verynice} for $\mathcal{Y}:= BG$ and using \cref{imp}, we obtain the desired exactness in the left and in the middle. For the surjectivity on the right, we first note that by applying descent along $* \to BG,$ we have $$\mathcal{H}^2_{\Prism}(BG) \simeq \mathcal{E}xt^1_{(S)_{\mathrm{qsyn}}}(G, \Prism_{(\cdot)}).$$ In order to see the above, we use the $E_2$-spectral sequence 
\begin{equation}\label{u}
E_2^{i,j}=\mathrm{Ext}^i_{(S')_{\mathrm{qsyn}}}(H^{-j} (\mathbb{Z}[BG_{S'}]), \Prism_{(\cdot)} ) \implies H^{i+j}_{(S')_{\mathrm{qsyn}}}(BG_{S'}, \Prism_{(\cdot)})  
\end{equation} for any quasisyntomic algebra $S'$ over $S.$ This implies that $\mathcal{H}^1_{\Prism}(BG)= \mathcal{H}om_{(S)_{\mathrm{qsyn}}}(G, \Prism_{(\cdot)}).$ By our hypothesis, it follows that $\mathcal{H}om_{(S)_{\mathrm{qsyn}}}(G, \Prism_{(\cdot)})=0.$ To prove that $\mathcal{H}^2_{\Prism} (BG) \simeq \mathcal{E}xt^1(G, \Prism_{(\cdot)}),$ by the above spectral sequence and the isomorphism $H^{-2}(\mathbb{Z}[BG]) \simeq G \wedge G$, it would be enough to prove that $\mathcal{H}om_{(S)_{\mathrm{qsyn}}}(G \wedge G, \Prism_{(\cdot)}) =0.$ This follows from the vanishing $\mathcal{H}om_{(S)_{\mathrm{qsyn}}}(G, \Prism_{(\cdot)})=0$ in a way similar to the proof of \cref{imp}. Now we claim that $$\mathcal{H}^2_{(S)_{\mathrm{qsyn}}} (BG, \mathrm{Fil}_{\mathrm{Nyg}}^1 \Prism_{(\cdot)}) \simeq \mathcal{E}xt^1_{(S)_{\mathrm{qsyn}}}(G, \mathrm{Fil}_{\mathrm{Nyg}}^1 \Prism_{(\cdot)}).$$ In order to see this, we may again use an $E_2$-spectral sequence 
\begin{equation}\label{u}
E_2^{i,j}=\mathrm{Ext}^i_{(S')_{\mathrm{qsyn}}}(H^{-j} (\mathbb{Z}[BG_{S'}]), \mathrm{Fil}^1_{\mathrm{Nyg}}\Prism_{(\cdot)} ) \implies H^{i+j}_{(S')_{\mathrm{qsyn}}}(BG_{S'}, \mathrm{Fil}^1_{\mathrm{Nyg}}\Prism_{(\cdot)})  
\end{equation} for any quasisyntomic algebra $S'$ over $S.$ Note that for a qrsp algebra $S_0 \in (S)_{\mathrm{qsyn}},$ we have an injection of discrete abelian groups  $0 \to \mathrm{Fil}^1_{\mathrm{Nyg}} \Prism_{S_0} \to \Prism_{S_0}.$ Since qrps algebras form a basis for $(S)_{\mathrm{qsyn}},$ it follows that we have an injection $\mathcal{H}om_{(S)_{\mathrm{qsyn}}}(G, \mathrm{Fil}_{\mathrm{Nyg}}^1 \Prism_{(\cdot)}) \to \mathcal{H}om_{(S)_{\mathrm{qsyn}}}(G, \Prism_{(\cdot)})=0.$ This implies that $\mathcal{H}om_{(S)_{\mathrm{qsyn}}}(G, \mathrm{Fil}_{\mathrm{Nyg}}^1 \Prism_{(\cdot)})=0.$ To prove that $\mathcal{H}^2_{(S)_\mathrm{qsyn}} (BG, \mathrm{Fil}_{\mathrm{Nyg}}^1 \Prism_{(\cdot)}) \simeq \mathcal{E}xt^1_{(S)_{\mathrm{qsyn}}}(G, \mathrm{Fil}^1_{\mathrm{Nyg}}\Prism_{(\cdot)}),$ using the spectral sequence, it is enough to prove that $\mathcal{H}om_{(S)_{\mathrm{qsyn}}}(G \wedge G, \mathrm{Fil}^1_{\mathrm{Nyg}}\Prism_{(\cdot)}) =0,$ for which we argue the same way as before. Now the surjectivity in the right of \cref{finiteflat0} is equivalent to the surjectivity of $\mathcal{E}xt^1_{(S)_{\mathrm{qsyn}}}(G, \mathrm{Fil}_{\mathrm{Nyg}}^1 \Prism_{(\cdot)})\left \{1\right \}\xrightarrow{\varphi_1 - \mathrm{can}}\mathcal{E}xt^1_{(S)_{\mathrm{qsyn}}}(G, \Prism_{(\cdot)})\left \{1\right \}.$ However, this surjectivity follows from $\mathcal{E}xt^{2}_{(S)_{\mathrm{qsyn}}}(G, \mathbb{Z}_p(1))=0$ (see \cref{w}).
\end{proof}{}
 

\begin{proposition}
 For a $p$-divisible group $G$ over a quasisyntomic ring $S,$ we have the following exact sequence of abelian sheaves in $(S)_{\mathrm{qsyn}}:$
\begin{equation}\label{p-div}
  0 \to  T_p(G^\vee) \to \mathcal \mathcal{H}^2_{(S)_{\mathrm{qsyn}}} (BG, \mathrm{Fil}_{\mathrm{Nyg}}^1 \Prism_{(\cdot)}) \left \{ 1 \right \} \xrightarrow{\varphi_1 - \mathrm{can}} \mathcal{H}^2_{\Prism} (BG)\left \{ 1 \right \} \to 0.
\end{equation}   
\end{proposition}{}
\begin{proof}
By \cref{pde}, in this case we have $\mathcal{H}^1_{\Prism} (BG)=0.$ The proof now follows exactly in the same way as \cref{finiteflat0rev}.\end{proof}{}

\begin{notation}
Let $\tau_{[-3,-2]} R\Gamma(B^2 G_{(\cdot)}, \mathbb{Z}_p(1))^\sharp$ denotes the $D(\mathbb{Z}_p)^{\wedge}_p$ valued sheaf obtained by sheafifying the presheaf that sends $$(S)_{\mathrm{qsyn}} \ni T \mapsto \tau_{[-3,-2]} R\Gamma(B^2 G_T, \mathbb{Z}_p(1));$$ similarly for $\tau_{[-3,-2]}\mathrm{Fil}^1_{\mathrm{Nyg}}R\Gamma_{\Prism}(B^2 G_{(\cdot)})\left \{1\right\}^\sharp$ and 
$\tau_{[-3,-2]} R\Gamma_{\Prism}(B^2 G_{(\cdot)})\left \{1\right\}^\sharp.$ We have the following:
\end{notation}{}

\begin{proposition}\label{finiteflatrev}Let $G$ be a finite locally free group scheme over a quasisyntomic ring $S$. We have the following fiber sequence of $D(\mathbb Z_p)^{\wedge}_p$-valued sheaves on $(S)_{\mathrm{qsyn}}$:
\begin{equation}\label{finiteflat2}
\begin{split}  \tau_{[-3,-2]} R\Gamma(B^2 G_{(\cdot)}, \mathbb{Z}_p(1))^\sharp \to \tau_{[-3,-2]}\mathrm{Fil}^1_{\mathrm{Nyg}}R\Gamma_{\Prism}(B^2 G_{(\cdot)})\left \{1\right\}^\sharp \\ \xrightarrow{\varphi_1 - \mathrm{can}} \tau_{[-3,-2]} R\Gamma_{\Prism}(B^2 G_{(\cdot)})\left \{1\right\}^\sharp.\end{split}
\end{equation}   
\end{proposition}{}
\vspace{1mm}
\begin{proof}
Using \cref{verynice} for $\mathcal{Y}:= B^2 G,$ it would suffice to show that $\mathcal{H}^1_{\Prism}(B^2G)=0$ and the map $\mathcal{H}^3_{(S)_{\mathrm{qsyn}}} (B^2 G, \mathrm{Fil}^1_{\mathrm{Nyg}} \Prism_{(\cdot)}) \xrightarrow{\varphi_i - \mathrm{can}} \mathcal{H}^3_{(S)_{\mathrm{qsyn}}} (B^2 G, \Prism_{(\cdot)})$ is surjective.

To this end, we use the $E_2$-spectral sequence 
\begin{equation}\label{u}
E_2^{i,j}=\mathrm{Ext}^i_{(S')_{\mathrm{qsyn}}}(H^{-j} (\mathbb{Z}[B^2G_{S'}]), \Prism_{(\cdot)} ) \implies H^{i+j}_{(S')_{\mathrm{qsyn}}}(B^2G_{S'}, \Prism_{(\cdot)})  
\end{equation} for any quasisyntomic algebra $S'$ over $S.$ Since $H^{-1} (\mathbb{Z}[B^2G])=0,$ the above spectral sequence gives $\mathcal{H}^1_{\Prism} (B^2 G)=0.$ Using the same spectral sequence and the fact that $H^{-3}(\mathbb{Z}[B^2G])=0, H^{-2} (\mathbb{Z}[B^2G]) \simeq G$, and $H^{-1}(\mathbb{Z}[B^2 G])=0,$ it follows that $$\mathcal{H}^3_{(S)_{\mathrm{qsyn}}}(B^2 G, \Prism_{(\cdot)}) \simeq \mathcal{E}xt^1_{(S)_{\mathrm{qsyn}}}(G, \Prism_{(\cdot)}).$$ By an entirely similar argument, we have $$\mathcal{H}^3_{(S)_{\mathrm{qsyn}}}(B^2 G, \mathrm{Fil}^1_{\mathrm{Nyg}}\Prism_{(\cdot)}) \simeq \mathcal{E}xt^1_{(S)_{\mathrm{qsyn}}}(G, \mathrm{Fil}^1_{\mathrm{Nyg}}\Prism_{(\cdot)}).$$ The desired surjectivity is equivalent to the surjectivity of $$\mathcal{E}xt^1_{(S)_{\mathrm{qsyn}}}(G, \mathrm{Fil}_{\mathrm{Nyg}}^1 \Prism_{(\cdot)})\xrightarrow{\varphi_1 - \mathrm{can}}\mathcal{E}xt^1_{(S)_{\mathrm{qsyn}}}(G, \Prism_{(\cdot)}).$$ However, this surjectivity follows from the vanishing $\mathcal{E}xt^{2}_{(S)_{\mathrm{qsyn}}}(G, \mathbb{Z}_p(1))=0$ (\cref{w}).
\end{proof}{}

\begin{proposition}
Let $G$ be a finite locally free group scheme over a quasisyntomic ring $S.$ We have the following fiber sequence of $D(\mathbb{Z}_p)^{\wedge}_p$-valued sheaves on $(S)_{\mathrm{qsyn}}$:   
\begin{equation}\label{finiteflat}
\begin{split}  R\Gamma_{\mathrm{qsyn}}((\cdot), G^\vee) [-3] \to \tau_{[-3,-2]}\mathrm{Fil}^1_{\mathrm{Nyg}}R\Gamma_{\Prism}(B^2 G_{(\cdot)})\left \{1\right\}^\sharp \\ \xrightarrow{\varphi_1 - \mathrm{can}} \tau_{[-3,-2]} R\Gamma_{\Prism}(B^2 G_{(\cdot)})\left \{1\right\}^\sharp. \end{split}
\end{equation}
\end{proposition}{}
\begin{proof}

This follows from \cref{ww} and \cref{finiteflatrev}. Indeed, by \cref{ww}, $\tau_{[-3,-2]} R\Gamma(B^2 G_{(\cdot)}, \mathbb{Z}_p(1))^\sharp \simeq (G^\vee(\cdot)[-3])^\sharp \simeq R\Gamma_\qsyn ((\cdot), G^\vee)[-3].$
\end{proof}{}

\begin{remark}\label{tricky}
The above exact sequences make it clear that it is possible to reconstruct $G$ from prismatic cohomology, and what other \textit{extra data} is essential for the purpose of reconstruction. Namely, we use $\mathcal{H}^2_{\Prism}(BG)$ or $\tau_{[-2,-3]} R\Gamma_{\Prism}(B^2G),$ viewed as a prismatic crystal, as well as the Nygaard filtration, and the \textit{divided} Frobenius. If one wanted to fully faithfully embed $p$-divisible groups and finite locally free group schemes, one naturally looks for a category where all of these data make sense. This is given by the derived category of the stack $S^{\mathrm{syn}},$ introduced by Drinfeld and Bhatt--Lurie, which gives the category of coefficients for prismatic cohomology along with the other natural data such as the Nygaard filtration, divided Frobenius, etc. We will see that the formalism of quasi-coherent sheaves on these stacks (which automatically keeps track of the relevant data) gives a very convenient framework for Dieudonn\'e theory.
\end{remark}{}

\subsection{Preliminaries on prismatic $F$-gauges}\label{korma1}
 For the rest of this section, the base ring $S$ should be assumed to be a $p$-complete quasisyntomic ring unless otherwise mentioned. In this set up, one can consider the prismatization stacks $\mathrm{Spf}(S)^\Prism, \mathrm{Spf}(S)^{\mathrm{Nyg}}$ and $\mathrm{Spf}(S)^{\mathrm{syn}}$ as defined in \cite{drinew}, \cite{BhaLur2}; they are also discussed in \cite{fg}, which will be our main reference. The category of prismatic $F$-gauges, by definition (see \cref{prismfga}), is the derived $\infty$-category of quasicoherent sheaves on the stack $\mathrm{Spf}(S)^\syn.$ In \cref{fgaugeff}, we give an equivalent description of the category of prismatic $F$-gauges by working in the quasisyntomic site of $S,$ which is similar to the notion of $F$-gauges in \cite{FJ13} in spirit. In \cref{comppp}, we explain how to view admissible prismatic Dieudonn\'e modules introduced in \cite{alb} in terms of prismatic $F$-gauges.

\begin{definition}[$p$-adic formal stack]\label{nota1}
 A functor $X$ from the category of $p$-nilpotent animated rings to the $\infty$-category $\mathcal{S}$ will be simply called a $p$-adic formal prestack. A $p$-adic formal prestack $X$ that satisfies descent (in the sense of \cite[\S~6.1.3]{Lur09}) for the fpqc topology will be called a $p$-adic formal stack.   
\end{definition}{}

\begin{example}
 When $R$ is qrsp, see \cite[Rmk.~5.5.3]{fg} for the description of $\mathrm{Spf}(R)^{\mathrm{Nyg}}$ as a functor defined on $p$-nilpotent rings.   
\end{example}{}

\begin{remark}\label{presen}
By \cite[Rmk.~5.5.18]{fg} (cf.~\cite[Cor.~6.12.5]{madapusi}), using quasisyntomic descent, for a quasisyntomic ring $S,$ one has a natural isomorphism of $p$-adic formal stacks 
\begin{equation}\label{descentnyg1}
   \mathrm{colim}_{\substack{S \to R; \\ R~\text{is qrsp} }} \mathrm{Spf}(R)^{\mathrm{Nyg}} \xrightarrow{ \sim } \mathrm{Spf}(S)^{\mathrm{Nyg}}. 
\end{equation}{}
In particular, let $S \to R_0$ be a chosen quasisyntomic cover, where $R_0$ is qrsp. Then each term in the Cech conerve of $S \to R_0$ is a qrsp algebra (\cite[Lem.~4.30]{BMS1}). Since $\mathrm{Spf} (R_0)^{\mathrm{Nyg}} \to \mathrm{Spf} (S)^{\mathrm{Nyg}}$ is a surjection of $p$-adic formal stacks (\cite[Cor.~6.12.5]{madapusi}), by applying descent, and the Tor-independent finite product preservation property of the functor $X \mapsto X^{\mathrm{Nyg}}$ (see \cite[Rmk.~5.5.18]{fg}), it follows that we have natural isomorphisms
\begin{equation}\label{simpny}
 \mathrm{colim}_{[n] \in \Delta^{\mathrm{op}}}\mathrm{Spf}(R_0^{[n]})^{\mathrm{Nyg}} \simeq  \mathrm{colim}_{[n] \in \Delta^{\mathrm{op}}}\left(\mathrm{Spf}(R_0)^{\mathrm{Nyg}}\right)^{[n]} \simeq  \mathrm{Spf}(S)^{\mathrm{Nyg}},   
\end{equation}where $R_0^{[n]}:= R_0^{\otimes_S n}$ is a qrsp algebra. Note that one defines (cf.~\cref{sc}) 
\begin{equation}\label{sick3}
 \mathrm{Spf}(S)^{\mathrm{syn}}:=  \mathrm{coeq} \left(\xymatrix{
  \mathrm{Spf}(S)^\Prism\ar@<1ex>[r]^{\mathrm{can}}\ar@<-1ex>[r]_{\varphi} & \mathrm{Spf}(S)^{\mathrm{Nyg}}
 }\right) .  \end{equation}
 Therefore, by \cref{descentnyg1}, we have a natural isomorphism of $p$-adic formal stacks \begin{equation}\label{descentsyn1}
   \mathrm{colim}_{\substack{S \to R; \\ R~\text{is qrsp} }} \mathrm{Spf}(R)^{\mathrm{syn}} \xrightarrow{ \sim } \mathrm{Spf}(S)^{\mathrm{syn}}. 
\end{equation}Furthermore, using the cover $S \to R_0,$ from \cref{simpny} we have a natural isomorphism 
\begin{equation}\label{simpsyn}
 \mathrm{colim}_{[n] \in \Delta^{\mathrm{op}}}\mathrm{Spf}(R_0^{[n]})^{\mathrm{syn}} \simeq \mathrm{Spf}(S)^{\mathrm{syn}}.   
\end{equation}These covers by syntomification stacks associated to qrsp algebras will be used later in certain arguments.
\end{remark}{}

\begin{remark}\label{introdd}
 The stacks $(S)^\Prism,$ $(S)^{\mathrm{Nyg}}$ and $(S)^\syn$ are related to the theory of (absolute) prismatic cohomology in the following way: there is a canonical line bundle on each of these stacks, denoted as $\mathcal{O}\left \{i \right \}$, and called the Breuil--Kisin twist (see \cref{bktwis}) such that $$R\Gamma(S^\Prism, \mathcal{O}\left \{ i \right \}) \simeq \Prism_S \left \{i \right \},$$ $$R\Gamma(S^\Nyg, \mathcal{O}\left \{ i \right \}) \simeq \Fil^i_\Nyg\Prism_S \left \{i \right \},$$ and $$R\Gamma(S^\syn, \mathcal{O}\left \{ i \right \}) \simeq R\Gamma_\syn (S, \mathbb{Z}_p(i)),$$ where the latter denotes syntomic cohomology of weight $i.$ See \cref{lion} for certain generalizations.
\end{remark}{}




\begin{construction}[Quasi-coherent sheaves on $p$-adic formal stacks]\label{conspadicqcoh}
Let $X$ be a $p$-adic formal prestack. One defines the quasi-coherent derived $\infty$-category of $X$ denoted as $$D_{\mathrm{qc}}(X):= \varprojlim_{(\mathrm{Spec}\, T \to X) \in \mathcal{C}_X} D(T),$$ where $D(T)$ is the derived category of the $p$-nilpotent animated ring $T.$ By faithfully flat descent, the association $X \mapsto D_{\mathrm{qc}}(X)$, viewed as a functor from the opposite $\infty$-category of $p$-adic formal stacks to the $\infty$-category of (not necessarily small) $\infty$-categories (denoted as $\widehat{\mathcal{C}\mathrm{at}}_\infty$) is a sheaf in the $\infty$-categorical sense, i.e., it preserves small limits (see e.g., \cite[Ex.~4.2.4, Ex. 4.2.5]{Lur2}, cf.~\cite[Prop.~6.2.3.1]{spectra}). 
\end{construction}{}

\begin{definition}[Prismatic $F$-gauges]\label{prismfga}
 Let $S$ be a quasi-syntomic ring. One defines the stable $\infty$-category of prismatic $F$-gauges over $S$ to be
$$F\text{-Gauge}_{\Prism}(S) := D_{\mathrm{qc}}(\mathrm{Spf}(S)^{\mathrm{syn}}).$$   
\end{definition}{}

\begin{remark}\label{si}
By \cref{descentsyn1} and the sheaf property discussed in \cref{conspadicqcoh}, it follows that
\begin{equation}\label{willprobneedit}
F\text{-Gauge}_{\Prism}(S) \simeq \lim_{\substack{S \to R; \\R~\text{is qrsp}}} F\text{-Gauge}_{\Prism}(R).   
\end{equation}
In other words, specifying a quasicoherent sheaf on $\mathrm{Spf}(S)^{\mathrm{syn}}$ is equivalent to producing a quasicoherent sheaf on $\mathrm{Spf}(R)^{\mathrm{syn}}$ for every map $S \to R,$ where $R$ is qrsp, in a manner that is compatible with base change.
\end{remark}{}

\begin{remark}[Explicit description of $F$-gauges]\label{cherry}
When $R$ is qrsp, it is possible to describe the category $F\text{-Gauge}_{\Prism}(R)$ completely algebraically. Let $(\Prism_R, I)$ be the prism associated to $R.$ The Nygaard filtered prismatic cohomology of $R$ denoted as $\mathrm{Fil}^\bullet_{\mathrm{Nyg}}\Prism_R$ is naturally a filtered ring. Note that the Frobenius on prismatic cohomology naturally defines a map $\mathrm{Fil}^\bullet_{\mathrm{Nyg}} \Prism_R \to I^\bullet \Prism_R$ of filtered rings. 

Now, $F\text{-Gauge}_{\Prism}(R)$ is the $\infty$-category of $(p,I)$-complete $\mathbb{Z}$-indexed filtered modules $\mathrm{Fil}^\bullet M$ over $\mathrm{Fil}^\bullet_{\mathrm{Nyg}}\Prism_R$ equipped with a $\mathrm{Fil}^\bullet_{\mathrm{Nyg}}\Prism_R$-linear map $$\varphi: \mathrm{Fil}^\bullet M \to I^\bullet \Prism_R \otimes_{\Prism_R} M=: I^\bullet M$$ (where the $\mathrm{Fil}^\bullet_{\mathrm{Nyg}} \Prism_R$-module structure on the right hand side is obtained by restriction of scalar and $M$ is the $(p,I)$-completed underlying object of $\mathrm{Fil}^\bullet M$) with the property that the map 
\begin{equation}\label{who1}    
\mathrm{Fil}^\bullet M \widehat{\otimes}_{\mathrm{Fil}^\bullet_{\mathrm{Nyg}}\Prism_R} I^\bullet \Prism_R \to I^\bullet M
\end{equation}{}
associated via adjunction is an isomorphism of filtered $I^\bullet \Prism_R$-modules.

Combined with \cref{si}, in principle, we obtain a purely algebraic description of $F\text{-Gauge}_{\Prism}(S)$ for any quasisyntomic ring $S.$ 
\end{remark}{}

\begin{example}Let $k$ be a perfect field of characteristic $p$. Then the category $F\text{-Gauge}_{\Prism}(k)$ is equivalent to the category of ``$\varphi$-gauges" introduced by Fontaine--Jannsen \cite[\S~1.2]{FJ13}. See the discussion in \cite[Rmk.~3.4.6]{fg}.
\end{example}{}

\begin{example}[Breuil--Kisin twists]\label{bktwis} For any quasisyntomic ring $S,$ the stack $\mathrm{Spf}(S)^{\mathrm{syn}}$ is equipped with a line bundle $\mathcal{O}\left \{1 \right \}$ arising from the Breuil--Kisin twist in prismatic cohomology. Let $\mathcal{O}\left\{n\right\} := \mathcal{O}\left\{1\right\}^{\otimes n}$ for $n \in \mathbb{Z}.$ We explain an algebraic description of this in the case when $S$ is qrsp, which glues to define the desired line bundle on $\mathrm{Spf}(S)^{\mathrm{syn}}$. For $n \in \mathbb{Z},$ the filtered $\mathrm{Fil}^\bullet_{\mathrm{Nyg}} \Prism_S$-module defined by setting
$$\mathrm{Fil}^i M:= \mathrm{Fil}^{i-n}_{\mathrm{Nyg}} \Prism_S \otimes_{\Prism_S} \Prism_S \left \{ -n \right \}$$ equipped with the natural Frobenius map (obtained as tensor product of the Frobenius  $\Fil^{i-n}_{\mathrm{Nyg}} \Prism_S \to I^{i-n} \Prism_S$ and the Frobenius $\Prism_S \left \{-n\right \} \to I^n \Prism_S \left \{-n\right \}$, where $I$ is the ideal of $\Prism_S$; see \cite[Rmk.~2.5.9]{BhaLu}) defines an object of $F\text{-Gauge}_{\Prism}(S)$ that gives the Breuil--Kisin twist $\mathcal{O}\left \{-n \right \}.$ In order to see the isomorphism in \cref{who1}, one uses transitivity property of tensor products and the fact that the Frobenius map $\Prism_S \left \{-n \right \} \to I^n \Prism_S \left \{-n \right \}$ induces a canonical $\Prism_S$-linear  isomorphism $\varphi^*_{\Prism_S} (\Prism_S \left \{-n \right \}) \simeq I^n \Prism_S \left \{-n \right \}$ (see \cite[Rmk.~2.5.9]{BhaLu}).
    
\end{example}{}

\begin{remark}[Effective $F$-gauges]
 When $R$ is qrsp, an object of $F\text{-Gauge}_{\Prism}(R)$ is called effective if the underlying $\mathbb{Z}$-filtered object $\mathrm{Fil}^{\bullet} M$ has the property that the natural maps $\mathrm{Fil}^i M \to \mathrm{Fil}^{i-1}M$ are all isomorphisms for $i \le 0.$ 
\end{remark}{}

\begin{example}
 Let $X$ be a quasisyntomic $p$-adic formal over the qrsp algebra $R.$ Then by quasisyntomic descent, the filtered $\mathrm{Fil}^\bullet_{\mathrm{Nyg}} \Prism_R$-module $\mathrm{Fil}^\bullet_{\mathrm{Nyg}}R\Gamma_{\Prism}(X) $ equipped with the natural Frobenius map has the structure of an effective $F$-gauge.
\end{example}{}

\begin{definition}[Weak $F$-gauges]\label{earthq1}
 Let $R$ be a qrsp ring. The category of prismatic weak $F$-gauges over $R$ denoted as $F\text{-Gauge}^{\mathrm{w}}_{\Prism}(R)$ is the $\infty$-category of $(p,I)$-complete $\mathbb{Z}$-indexed filtered modules $\mathrm{Fil}^\bullet M$ over $\mathrm{Fil}^\bullet_{\mathrm{Nyg}}\Prism_R$ equipped with a $\mathrm{Fil}^\bullet_{\mathrm{Nyg}}\Prism_R$-linear map $$\varphi: \mathrm{Fil}^\bullet M \to I^\bullet \Prism_R \otimes_{\Prism_R} M=: I^\bullet M;$$ as before, the $\mathrm{Fil}^\bullet_{\mathrm{Nyg}} \Prism_R$-module structure on the right hand side is obtained by restriction of scalar and $M$ is the $(p,I)$-completed underlying object of $\mathrm{Fil}^\bullet M$. Note that the natural functor $F\text{-Gauge}^{\mathrm{}}_{\Prism}(R) \to F\text{-Gauge}^{\mathrm{w}}_{\Prism}(R)$ is fully faithful.
\end{definition}{}

\begin{remark}\label{uset}
 If $\Fil^\bullet M$ is a weak prismatic $F$-gauge over a qrsp ring $R$ such that $\Fil^n M$ stabilizes for $n \ll 0,$ then by passing to filtered colimits on both sides of the map $$\varphi: \mathrm{Fil}^\bullet M \to I^\bullet \Prism_R \otimes_{\Prism_R} M=: I^\bullet M,$$ we obtain a map $$\varphi_M: (\varphi_{\Prism_R}^* M)[1/I] \to M[1/I].$$ If $\Fil^\bullet M$ is further assumed to be a dualizable prismatic $F$-gauge, then $\varphi_M$ is an isomorphism; the latter claim is obtained by passing to underlying objects of the filtered isomorphism in \cref{who1}, where the tensor product in the left can be computed in the non-completed sense by dualizability of $\Fil^\bullet M$.
\end{remark}{}

\begin{construction}[Canonical $t$-structures]\label{fl}
We call a prismatic weak $F$-gauge as above effective if the underlying $\mathbb{Z}$-filtered object $\mathrm{Fil}^\bullet M$ has the property that the maps $\mathrm{Fil}^iM \to \mathrm{Fil}^{i-1} M$ are isomorphisms for $i \le 0.$ We note that the category of effective prismatic weak $F$-gauges (which is a stable $\infty$-category with all small limits and colimits) has the pleasant feature that it is canonically equipped with a $t$-structure, where an object is connective (resp. coconnective) if the underlying filtered module $\mathrm{Fil}^\bullet M$ has the property that $\mathrm{Fil}^i M$ is connective (resp. coconnective) for all $i \in \mathbb{Z}.$ \end{construction}

\begin{remark}[Gauges as filtered crystals on the quasisyntomic site]\label{vanc} Suppose that $S$ is a quasisyntomic ring. In \cite{fg}, one defines the category of prismatic gauges over $S$ to be $D_{\mathrm{qc}}(\mathrm{Spf}(S)^{\mathrm{Nyg}}).$ We will explain how to think of them in the quasisyntomic site of $S.$ Let $(S)_{\mathrm{qrsp}}$ denote category of qrsp algebras over $S$ equipped with the quasisyntomic topology, which, by \cite[Lem.~4.27]{BMS2} forms a site (also see \cite[Variant.~4.33]{BMS2}).

 Note that the association $$(S)_{\mathrm{qrsp}} \ni R \mapsto \mathrm{Fil}^{\bullet}_{\mathrm{Nyg}}\Prism_{R}$$ defines a $D(\mathbb {Z})$-valued sheaf of $\mathbb{Z}$-indexed filtered rings on $(S)_{\mathrm{qrsp}}$, which we denote by $\mathrm{Fil}^{\bullet}\Prism_{(\cdot)}$. Also, the association $(S)_{\mathrm{qrsp}} \ni R \mapsto I$,
 where $I$ is the invertible ideal of $\Prism_R$ defines an invertible sheaf of $\Prism_{(\cdot)}$-modules which we will denote by $\mathscr I.$ Let $\mathrm{Gauge}_{\Prism}(S)$ be the $\infty$-category of derived $(p,\mathscr{I})$-complete sheaf of filtered modules $\mathrm{Fil}^{\bullet}\mathscr{M}$ over the sheaf of filtered ring $\mathrm{Fil}^{\bullet} \Prism_{(\cdot)}$ that satisfies the following ``crystal" property: for every map $R \to T$ in $(S)_{\mathrm{qrsp}},$ the natural map $\mathrm{Fil}^{\bullet}\mathscr{M} (R) \to \mathrm{Fil}^{\bullet} \mathscr{M} (T)$ of filtered $\mathrm{Fil}^{\bullet} \Prism_{R}$-modules induces an isomorphism
\begin{equation}\label{crystalprop}
\mathrm{Fil}^{\bullet} \mathscr{M}(R) \widehat{\otimes}_{\mathrm{Fil}^{\bullet} \Prism_R} \mathrm{Fil}^{\bullet} \Prism_{T} \xrightarrow{\sim} \mathrm{Fil}^{\bullet} \mathscr{M}(T)    
\end{equation}{}
of derived $(p,\mathscr{I}(T))$-complete $\mathrm{Fil}^{\bullet} \Prism_T$-modules. 

\end{remark}{}
\begin{proposition}\label{equivalence}
 Let $S$ be a quasisyntomic ring. Then $D_{\mathrm{qc}}(\Spf(S)^\Nyg) \simeq \mathrm{Gauge}_{\Prism}(S).$   
\end{proposition}{}

\begin{proof}
   Indeed, by \cite[Rmk.~5.5.4]{fg}, if $R$ is qrsp, an object of $D_{\mathrm{qc}}(\Spf(R)^\mathrm{Nyg})$ identifies with a derived $(p, I)$-complete filtered module over the filtered ring $\mathrm{Fil}^\bullet_{\mathrm{Nyg}}\Prism_R.$ Now let $\mathscr F \in D_{\mathrm{qc}}(\Spf(S)^\Nyg)$ and let $R \in (S)_\mathrm{qrsp}.$ The pullback of $\mathscr{F}$ to $\Spf(R)^\Nyg$ can be identified with a derived $(p,I)$-complete filtered module over the filtered ring $\mathrm{Fil}^\bullet_{\mathrm{Nyg}}\Prism_R,$ that we denote as $\Fil^\bullet \mathscr{M}_R.$ Using \cref{presen}, \cref{descentnyg1} and \cref{conspadicqcoh}, it follows that the association $(S)_{\mathrm{qrsp}} \ni R \mapsto \Fil^\bullet\mathscr{M}_R$ defines a derived $(p , \mathscr I)$-complete $D(\mathbb {Z})$-valued sheaf. Further, as pullback along any map $\mathrm{Spf}\,T \to \mathrm{Spf}\, R$ for qrsp rings $R$ and $T$ identifies with (completed) filtered tensor products along the induced map $\Fil^\bullet \Prism_R \to \Fil^\bullet \Prism_T$ in our set up, this sheaf satisfies the crystal property \cref{crystalprop}. This determines a functor $D_{\mathrm{qc}}(\Spf(S)^\Nyg) \to \mathrm{Gauge}_{\Prism} (S).$

Conversely, given $\Fil^\bullet \mathscr M \in \Ga(S)$ and $R \in (S)_\qrsp$, the filtered module $\Fil^\bullet \mathscr M (R)$ over the filtered ring $\Fil^\bullet_\Nyg \Prism_R$ determines an object $\mathscr F_R$ of $D_{\mathrm{qc}}(\Spf(R)^\Nyg).$
The crystal property of $\Fil^\bullet \mathscr M$ ensures that the construction $R \mapsto \mathscr F_R$ is compatible with (completed) base change for maps $R \to R'$ in $(S)_{\qrsp}.$ Since $D_{\mathrm{qc}}(\Spf(S)^\Nyg) \simeq \lim _{R \in (S)_{\qrsp}} D_{\mathrm{qc}}(\Spf(R)^\Nyg)$ (see \cite[Rmk.~5.5.18]{fg}), the association $R \mapsto \mathscr F_R$ determines an object of $D_{\mathrm{qc}} (\Spf(S)^\Nyg).$ This determines a functor $\mathrm{Gauge}_{\Prism} (S) \to D_{\mathrm{qc}}(\Spf(S)^\Nyg)$ that is inverse to the functor described in previous paragraph, which proves the proposition.
\end{proof}{}


\begin{remark}[$F$-gauges as filtered Frobenius crystals on the quasisyntomic site]\label{fgaugeff}Let $S$ be a quasisyntomic ring. We work on the site $(S)_\qrsp$ as before. Note that the category of prismatic $F$-gauges over $S$ can be equivalently described as the $\infty$-category of derived $(p,\mathscr{I})$-complete sheaf of $\mathbb Z$-indexed filtered modules $\mathrm{Fil}^{\bullet} \mathscr{N}$ over the sheaf of filtered ring $\mathrm{Fil}^{\bullet} \Prism_{(\cdot)}$ on $(S)_{\mathrm{qrsp}}$ equipped with a $\mathrm{Fil}^{\bullet} \Prism_{(\cdot)}$-linear Frobenius map 
$$\varphi:  \mathrm{Fil}^\bullet \mathscr N \to \mathscr{I}^\bullet \Prism_{(\cdot)}\otimes_{\Prism_{(\cdot)}} \mathscr N =: \mathscr{I}^\bullet \mathscr N$$(where the $\mathrm{Fil}^\bullet \Prism_{(\cdot)}$-module structure on the right hand side is obtained by restriction of scalar along the Frobenius map $\mathrm{Fil^\bullet} \Prism_{(\cdot)} \to \mathscr I^\bullet \Prism_{(\cdot)}$ and $\mathscr{N}$ denotes $(p, \mathscr I)$-completed underlying object of $\mathscr{N}$) such that the following two conditions hold:
\begin{enumerate}
    \item $\mathrm{Fil}^\bullet \mathscr{N}$ satisfies the ``crystal" property as in \cref{vanc}, \cref{crystalprop}.

\item The natural map $\mathrm{Fil}^\bullet \mathscr N \widehat{\otimes}_{\mathrm{Fil}^\bullet \Prism_{(\cdot)}} \mathscr{I}^\bullet \Prism_{(\cdot)} \to \mathscr{I}^\bullet \mathscr{N}$ associated to $\varphi$ via adjunction is an isomorphism of $\mathscr{I}^\bullet \Prism_{(\cdot)}$-modules.
\end{enumerate}{}
The equivalence of the above description with $D_{\mathrm{qc}}(S^{\mathrm{syn}})$ follows from \cref{presen} and \cref{cherry}. A prismatic $F$-gauge $\Fil^\bullet \mathscr N$ is called effective if the natural maps $\Fil^i \mathscr N \to \Fil^{i-1}\mathscr N$ are isomorphisms for all $i \le 0$ and the full subcategory spanned by such objects will be denoted by $F\text{-Gauge}^{\mathrm{}}_{\Prism}(S)^\mathrm{eff}.$
\end{remark}{}


\begin{remark}\label{smallworagain1}
 We record a variant of the category of prismatic $F$-gauges. Let $S$ be a quasisyntomic ring. We work on the site $(S)_\qrsp$ as before. We define the $\infty$-category $F\text{-}\mathrm{Mod}_{\Fil^\bullet \Prism}(S)$ to be the $\infty$-category of derived $(p, \mathscr I)$-complete sheaf of $\mathbb Z$-indexed filtered modules $\mathrm{Fil}^{\bullet} \mathscr{N}$ over the sheaf of filtered ring $\mathrm{Fil}^{\bullet} \Prism_{(\cdot)}$ on $(S)_{\mathrm{qrsp}}$ equipped with a $\mathrm{Fil}^{\bullet} \Prism_{(\cdot)}$-linear Frobenius map 
$$\varphi:  \mathrm{Fil}^\bullet \mathscr N \to \mathscr{I}^\bullet \Prism_{(\cdot)}\otimes_{\Prism_{(\cdot)}} \mathscr N =: \mathscr{I}^\bullet \mathscr N.$$ By definition, there is a fully faithful functor
$F\text{-Gauge}^{\mathrm{}}_{\Prism}(S) \to F\text{-}\mathrm{Mod}_{\Fil^\bullet \Prism}(S).$ Let us denote by $F\text{-}\mathrm{Mod}_{\Fil^\bullet \Prism}(S)^{\mathrm{eff}}$ the full subcategory spanned by $\Fil^\bullet \mathscr N \in F\text{-}\mathrm{Mod}_{\Fil^\bullet \Prism}(S)$ with the property that the natural maps $\Fil^i \mathscr N \to \Fil^{i-1}\mathscr N$ are isomorphisms for all $i \le 0.$ Note that $F\text{-}\mathrm{Mod}_{\Fil{^\bullet}\Prism}(S)^{\mathrm{eff}}$ is a stable $\infty$-category with all small limits and colimits.
\end{remark}{}

\begin{remark}\label{adjoinnn}
Let $S$ be a quasisyntomic ring. Let $\mathscr F \in \mathrm{Shv}_{D(\mathbb Z)}(S_\qrsp).$ Then the association $$\mathscr{F} \mapsto R\mathcal{H}om_{(S)_\qrsp}(\mathscr F, \mathrm{Fil}^\bullet \Prism_{(\cdot)})$$ determines a functor $$\mathrm{Shv}_{D(\mathbb{Z})}(S_\qrsp) \to (F\text{-}\mathrm{Mod}_{\Fil^\bullet \Prism}(S)^{\mathrm{eff}})^{\mathrm{op}}$$that preserves colimits. By the adjoint functor theorem, it must have a right adjoint. In this case, the right adjoint is explicitly determined by the functor that sends $$(F\text{-}\mathrm{Mod}_{\Fil^\bullet \Prism}(S)^{\mathrm{eff}})^{\mathrm{op}} \ni \Fil^\bullet \mathscr N \mapsto R\mathcal{H}om_{\syn}(\Fil^\bullet \mathscr N, \Fil^\bullet \Prism_{(\cdot)}),$$ where $R\mathcal{H}om_{\syn}(\Fil^\bullet \mathscr N, \Fil^\bullet \Prism_{(\cdot)})$ is the $D(\mathbb Z)$-valued sheaf on $(S)_\qrsp$ determined by $$(S)_\qrsp \ni R \mapsto R\Hom_{F\text{-}\mathrm{Mod}_{\Fil^\bullet \Prism}(R)}(\Fil^\bullet \mathscr N|_{(R)_\qrsp}, \Fil^\bullet \Prism_{(\cdot)}|_{(R)_\qrsp}). $$ To see the aforementioned description of the right adjoint, one may apply the adjunction to the case where $\mathscr{F}$ is given by the sheafification of the constant presheaf $\mathbb{Z}$ on $S_\qrsp$ and any $\Fil^\bullet \mathscr N \in (F\text{-}\mathrm{Mod}_{\Fil^\bullet \Prism}(S)^{\mathrm{eff}})^{\mathrm{op}} .$ 
\end{remark}{}

\begin{definition}[Hodge--Tate weights]\label{hodg}
Let us suppose that $S$ is a quasiregular semiperfectoid algebra. We have a natural map of graded rings 
$$\bigoplus_{i \in \mathbb Z} \mathrm{Fil}^i_{\mathrm{Nyg}} \Prism_S  \to S,$$ where $S$ is viewed as a graded ring sitting in degree zero, which is obtained by quotienting the left hand side by the graded ideal $$I:=\left(\bigoplus_{i \ne 0}\mathrm{Fil}^i_{\mathrm{Nyg}} \Prism_S \right) \oplus \mathrm{Fil}^1 _{\w{Nyg}} \Prism_S,$$ where the summand $\Fil^1_{\Nyg} \Prism_S$ on the right has weight zero.
This defines a map 
$$\mathrm{Spf}(S) \times B\mathbb{G}_m \to \mathrm{Spf}(S)^{\mathrm{Nyg}}.$$By quasisyntomic descent, this defines a map \begin{equation}\label{pullbackHT}
j: \mathrm{Spf}(S) \times B\mathbb{G}_m \to \mathrm{Spf}(S)^{\mathrm{Nyg} }  
\end{equation}for any quasisyntomic ring $S$ (cf.~\cite[Remark~5.3.14]{fg}). For any $M \in D_{\w{qc}}(\mathrm{Spf}(S)^{\w{Nyg}}),$ the pullback $j^* M$ can be identified with $\oplus_{i \in \mathbb{Z}} M_i \in D_{\w{qc}}(\mathrm{Spf}(S)^{\w{}}).$ The set of integers $i$ such that $M_i \ne 0$ is called the set of Hodge--Tate weights of $M.$ For $M \in D_{\w{qc}}(\mathrm{Spf}(S)^{\w{syn}}),$ the set of Hodge--Tate weights of $M$ is defined as the set of Hodge--Tate weights of the pullback of $M$ along $\mathrm{Spf}(S)^{\w{Nyg}} \to \mathrm{Spf}(S)^{\w{syn}}.$
\end{definition}{}

\begin{remark}\label{guolii}
 Let $S$ be a qrsp algebra. We will describe an explicit way to understand pullback of an object in $\mathscr{F} \in D_{\mathrm{qc}}(\Spf(S)^\Nyg)$ along the map $j$ from \cref{pullbackHT} (see \cite[Cons.~2.41]{guoli}). By \cref{equivalence}, we may identify $\mathscr F$ with a filtered module $\Fil^\bullet M$ over the filtered ring $\Fil^\bullet_{\Nyg} \Prism_S.$ Passing to graded pieces, we obtain a graded module $\mathrm{gr}^\bullet M$ over the graded ring $\mathrm{gr}^\bullet_{\Nyg} \Prism_S. $ Using the isomorphism $\mathrm{gr}^0 _{\mathrm{Nyg}} \Prism_S \simeq S,$ we obtain a natural map $\mathrm{gr}^\bullet_{\Nyg} \Prism_S \to S$ of graded rings. In this situation, by construction, it follows that $j^* \mathscr F$ corresponds to the graded $S$-module 
\begin{equation}\label{foh}
\mathrm{gr}^\bullet M \otimes_{\mathrm{gr}^\bullet_{\Nyg} \Prism_S} S.    
\end{equation}{}
 
\end{remark}{}

\begin{remark}
By \cref{hodg}, the Breuil--Kisin twist $\mathcal{O}\left \{-n \right \}$ has Hodge--Tate weight $n.$   
\end{remark}{}

\begin{remark}\label{sick}
 Suppose that $S$ is a quasiregular semiperfectoid algebra. Let $\mathrm{Fil}^\bullet M$ (eqiupped with a Frobenius) denote a prismatic $F$-gauge over $S$ in the sense of \cref{cherry}. Then the filtered module underlying the prismatic $F$-gauge $\Fil^\bullet M \left \{n \right \} :=  \Fil^\bullet M \otimes \mathcal{O}\left \{n \right \}$ is given by $$\Fil^i M \left \{n \right \} \simeq \Prism_S \left \{n \right \} \otimes_{\Prism_S} \Fil^{i+n} M.$$ 
\end{remark}{}

\begin{construction}[Syntomic cohomology of coefficients]\label{lion}
 Suppose that $S$ is a quasiregular semiperfectoid algebra. Let $\mathrm{Fil}^\bullet M$ (eqiupped with a Frobenius) denote a prismatic $F$-gauge over $S$ in the sense of \cref{cherry} and let $\mathscr{M}$ be the object of $D_\mathrm{qc}(\Spf(S)^\syn)$ corresponding to it. Let $\mathscr M \left \{ n \right \} := \mathscr M \otimes \mathcal{O} \left \{ n \right \}.$

For simplicity, let us assume that $\mathrm{Fil}^\bullet M$ is effective. Let $\mathscr{M}'\left \{ n \right \}$ denote the pullback of $\mathscr{M} \left \{ n \right \}$ to $\mathrm{Spf}(S)^\Nyg.$ Then, by \cref{sick}, \begin{equation}\label{sick1}
R\Gamma (\mathrm{Spf}(S)^\Nyg, \mathscr{M}'\left \{ n \right \}) \simeq \Prism_S\left \{ n \right \} \otimes_{\Prism_S} \Fil^{n} M =: \Fil^{n} M \left \{ n \right \}.
\end{equation}
Also, we have $R\Gamma (\mathrm{Spf}(S)^\Prism, \mathrm{can}^* \mathscr{M}' \left \{n \right \}) \simeq \Fil^0 M \left \{ n \right \}.$ Moreover, we have an induced map $i: \Fil^{n} M \left \{ n \right \} \to \Fil^0 M \left \{ n \right \},$ which is simply the canonical map arising from the filtration.
Now, we also have the map $\varphi: \Spf(S)^\Prism \to \Spf(S)^\Nyg,$ which induces a map 
\begin{equation}\label{sick2}
 R\Gamma (\mathrm{Spf}(S)^\Nyg, \mathscr{M}'\left \{ n \right \}) \to R\Gamma (\mathrm{Spf}(S)^\Prism, \varphi^* \mathscr{M}' \left \{n \right \}).   
\end{equation}{}
However, since $\mathscr{M}'\left \{n \right \}$ by construction descends to $\Spf (S)^\syn,$ we have an isomorphism $\varphi^* \mathscr{M}' \left \{n \right \} \simeq \mathrm{can}^* \mathscr{M}' \left \{n \right \}$. Combining this with \cref{sick1} and \cref{sick2}, we obtain a map
\begin{equation}\label{dividedfrob}
\nu_n:\Fil ^n M \left \{n \right\} \to \Fil^0 M \left \{n \right\}.
\end{equation}{}
 One may think of the above map as a certain ``divided Frobenius" for coefficients. In this set up, by \cref{sick3}, we have
\begin{equation}\label{sick5}
    R\Gamma (\Spf(S)^\syn, \mathscr{M}\left \{n \right \}) \simeq \mathrm{fib} \left(\Fil^n M \left \{n \right\} \xrightarrow{\nu_n - i} \Fil^0 M \left \{n \right\} \right).
\end{equation}{}

\end{construction}{}

We now move on to discussing some examples and a concrete description of prismatic $F$-gauges with Hodge--Tate weights on $[0,1]$ (see \cref{comppp},~cf.~\cite[Thm.~2.54]{guoli}, \cite[\S~5.27]{madapusi}). We start with the example of cohomology of $B\mathbb{G}_m$ which will be used later.

\begin{proposition}[Cohomology of $B\mathbb{G}_m$]\label{bgm}
Let $S$ be a quasiregular semiperfectoid ring. The effective prismatic weak $F$-gauge defined by the filtered object $H^2 (\mathrm{Fil}^\bullet_{\mathrm{Nyg}} R\Gamma_{\Prism}(B \mathbb{G}_m ))$ is isomorphic to the prismatic $F$-gauge $\mathcal{O}\left \{ -1 \right \}.$
\end{proposition}
\begin{proof}
Let us choose a map $R \to S,$ where $R$ is a perfectoid ring mapping onto $S.$ It would suffice to prove that the natural map $$\bigoplus_{n \in \mathbb{Z}} \mathrm{Fil}_{\mathrm{Nyg}}^{n-1}\Prism_{S} \left\{n-1\right \}
\to \bigoplus_{n \in \mathbb{Z}} H^2 (\mathrm{Fil}_{\mathrm{Nyg}}^nR\Gamma_{\Prism}(B\mathbb{G}_m) \left \{n \right \} )$$ of graded of modules over the graded ring $\oplus_{n \in \mathbb{Z}} \mathrm{Fil}_{\mathrm{Nyg}}^{n}\Prism_{S} \left\{n\right \}$ induced by the tautological Chern class $c_1 \in H^2 (\mathrm{Fil}_{\w{Nyg}}^1R\Gamma_{\Prism}(B\mathbb{G}_m)\left \{1 \right \})$ is an isomorphism. To this end, we fix an $n \in \mathbb{Z}$. It will suffice to prove that the natural map induced by the classes $c_1^i$ for $i \in \mathbb N$  \begin{equation}\label{inducedmapchernpower}
 \bigoplus_{i \ge 0} \mathrm{Fil}^{m-i}_{\mathrm{Nyg}}\Prism_S\left \{n-i \right \}[-2i] \xrightarrow[]{} \mathrm{Fil}^m_{\mathrm{Nyg}}R\Gamma_{\Prism}(B \mathbb G_m)\left \{n \right \}.   
\end{equation}
 is an isomorphism. By passing to the graded pieces of the filtrations appearing on both sides, \cref{inducedmapchernpower} induces a natural map
 \begin{equation}\label{nygaardchernpower}
     \bigoplus_{i \ge 0} \mathrm{gr}^{m-i}_{\mathrm{Nyg}}\Prism_S\left \{n-i \right \}[-2i] \xrightarrow[]{} \mathrm{gr}^m_{\mathrm{Nyg}}R\Gamma_{\Prism}(B \mathbb G_m)\left \{n \right \}.
 \end{equation}
 By choosing a generator for the principal ideal of the prism $\Prism_R$, one can trivialize the Breuil--Kisin twists. Further, we also obtain a commutative diagram 
 \begin{center}
   
  \begin{tikzcd}

{ \bigoplus_{i \ge 0} \mathrm{gr}^{m-i}_{\mathrm{Nyg}}\Prism_S[-2i]} \arrow[d, "\simeq"] \arrow[rrr] &  &  & \mathrm{gr}^m_{\mathrm{Nyg}}R\Gamma_{\Prism}(B \mathbb G_m). \arrow[d, "\simeq"] \\
{\bigoplus_{i \ge 0} \mathrm{Fil}^{m-i}_{\mathrm{conj}}\overline{\Prism}_S[-2i]} \arrow[rrr]         &  &  & \mathrm{Fil}^m_{\mathrm{conj}}R\Gamma_{\mathrm{HT}}(B \mathbb G_m)    \end{tikzcd}   
 \end{center}{}
where the lower horizontal map
\begin{equation}\label{conjugatechern}
  \bigoplus_{i \ge 0} \mathrm{Fil}^{m-i}_{\mathrm{conj}}\overline{\Prism}_S[-2i] \xrightarrow[]{} \mathrm{Fil}^m_{\mathrm{conj}}R\Gamma_{\mathrm{HT}}(B \mathbb G_m) 
\end{equation} is induced by powers of $c_1^{\mathrm{HT}} \in H^2 (\mathrm{Fil}^1_{\mathrm{conj}}R\Gamma_{\mathrm{HT}}(B \mathbb G_m)).$

Now note that in order to prove that \cref{inducedmapchernpower} is an isomorphism for all $m$, it suffices to prove it for $m=0$, and inductively, by passing to graded pieces, to prove that \cref{nygaardchernpower}, or equivalently, \cref{conjugatechern} is an isomorphism. Further, to prove \cref{inducedmapchernpower} is an isomorphism for $m=0$, by going derived modulo the ideal of $\Prism_S$, it suffices to prove that the induced map 
$\bigoplus_{i \ge 0} \overline{\Prism}_S\left \{n-i \right \}[-2i] \xrightarrow[]{} R\Gamma_{\mathrm{HT}}(B \mathbb G_m)\left \{n\right \}$ is an isomorphism. However, once we prove that \cref{conjugatechern} is an isomorphism, the latter map is automatically an isomorphism, by passing to the underlying object of the filtration.

Finally, to prove \cref{conjugatechern} is an isomorphism, since both sides of the map are zero for $m<0$, it suffices to show that \cref{conjugatechern} induces isomorphism on the graded pieces. Since $\mathbb{L}_{B\mathbb{G}_m/S} = \mathcal{O}[-1],$ from the transitivity fiber sequence for cotangent complex, $\mathbb L_{B\mathbb G_m/R}^{\wedge p}\simeq \mathcal{O}[-1] \oplus (\mathbb L_{S/R}^{\wedge p}\otimes \mathcal{O}).$ Therefore, $R\Gamma (B\mathbb{G}_m, \wedge^m\mathbb L_{B\mathbb G_m/R}[-m]^{\wedge p})\simeq \bigoplus_{i \ge 0} \wedge^{m-i} \mathbb L_{S/R}^{\wedge p}[-i-m]$. Thus the graded pieces of \cref{conjugatechern} identifies with the natural map 
\begin{equation}\label{hodgecherntoomuch}
\bigoplus_{i \ge 0}\mathrm{gr}_{\mathrm{conj}}^{m-i} \overline{\Prism}_S [-2i] \simeq \bigoplus_{i \ge 0} \wedge^{m-i} \mathbb L_{S/R}^{\wedge p}[-i-m]   \to \bigoplus_{i \ge 0} \wedge^{m-i} \mathbb L_{S/R}^{\wedge p}[-i-m].
\end{equation}which is induced by multiplication by $t^i$ on the $i$-th summand where $t \in H^2 (B\mathbb{G}_m, \mathbb{L}_{B\mathbb{G}_m/R}[-1])$ $ \simeq H^1(B\mathbb{G}_m, \mathbb{L}_{B\mathbb{G}_m/R}) \simeq  H^1(B\mathbb{G}_m, \mathbb{L}_{B\mathbb{G}_m/S}) \simeq S.$
It suffices to show that $t$ is a unit in $S$. By \cite[Thm.~7.6.2]{BhaLu}, the class $t$ is equal to the tautological Chern class, classically defined using $d\log.$ Thus, the fact that $t$ is a unit can be seen via functoriality. For example, let $\mathbb{P}^1 \to B\mathbb{G}_m$ be the map classifying $\mathcal{O}(1) \in \mathrm{Pic}(\mathbb{P}^1).$ Then map of $S$-modules $S \simeq H^1 (B\mathbb{G}_m, \mathbb{L}_{B\mathbb{G}_m/S}) \to H^1 (\mathbb{P}^1, \Omega^1_{\mathbb{P}^1}) \simeq S$ sends $t$ to a unit (see \cite[Tag~0FMI]{stacks}). Therefore, $t$ must be a unit. This shows that \cref{hodgecherntoomuch} is an isomorphism, which finishes the proof.
\end{proof}{}


\begin{example}[Breuil--Kisin modules]
 Let $R$ be a qrsp ring and let $(\Prism_R, \varphi_, I)$ be the associated prism. Let $M$ be a finite rank projective module over $\Prism_R$ equipped with an isomorphism $$\varphi_M: (\varphi^* M)[1/I] \xrightarrow{\sim} M[1/I]$$ or $\Prism_R$-modules. This data is also called a Breuil--Kisin module over $\Prism_R.$ We will denote the category of Breuil--Kisin modules by $\mathrm{BK}(\Prism_R).$

 For $i \in \mathbb Z,$ let us define
$$\Fil^i M:= \left \{m \in M \mid  \varphi_M(m) \in I^i M \right \}.$$ Then $\Fil^\bullet M$ is a filtered module over $\Fil^\bullet_{\mathrm{Nyg}}\Prism_R.$ Since $M$ is projective of finite rank, it follows that there exists some $n \in \mathbb Z$ such that $\varphi_M (M) \subseteq I^n M.$ In particular, $\Fil^i M$ stabilizes for $i \ll 0$ and the underlying object of $M^\bullet$ is naturally isomorphic to $M.$ Further, we have a $\Fil^\bullet_{\mathrm{Nyg}}\Prism_R$-linear map $$\Tilde{\varphi}_M: \Fil^\bullet M \to I^\bullet M.$$ By the above discussion, this gives a fully faithful functor (see \cref{uset}) 
\begin{equation}\label{funcc}
 \mathrm{BK}(\Prism_R) \to F\text{-Gauge}^{\mathrm{w}}_{\Prism}(R)   
\end{equation}In general, this functor does not factor through the inclusion $F\text{-Gauge}_{\Prism}(R) \to F\text{-Gauge}^{\mathrm{w}}_{\Prism}(R).$
\end{example}{}

\begin{example}
Let $R$ be a qrsp ring. Let us denote by $ F\text{-Gauge}^{\mathrm{vect}}_{\Prism}(R)$ the category $\mathrm{Vect}(\Spf(R)^\syn)$ of vector bundles on $\Spf(R)^\syn.$ By \cref{uset}, one obtains a functor 
\begin{equation}\label{funcc2}
\mathrm{BK}: F\text{-Gauge}^{\mathrm{vect}}_{\Prism}(R) \to \mathrm{BK}(\Prism_R).   
\end{equation}By \cite[Cor.~2.53]{guoli}, this functor is fully faithful.
\end{example}

\begin{remark}
 Let $R$ be a qrsp ring and let $\mathrm{DM}^{\mathrm{adm}}(R) \subset \mathrm{BK}^{\mathrm{}}(\Prism_R)$ be the full subcategory of admissible prismatic Dieudonn\'e modules defined in \cite[Def.~1.3.5]{alb} which we recall. An object $(M, \varphi_M) \in \mathrm{BK}^{\mathrm{}}(\Prism_R)$ is an admissible prismatic Dieudonn\'e module if the following two conditions are satisfied.
\begin{enumerate}
    \item The map $\varphi_M|_{\varphi^* M}$ has its image contained in $M$ and cokernel of the induced map $\varphi^* M \to M$ is killed by $I.$

\item The image of the composition $M \xrightarrow{\varphi_M|_M} M \to M/IM$ is a finite locally free $\Prism_R/ \Fil^1_{\mathrm{Nyg}}\Prism_R \simeq R$-module $F_M$ such that the induced map $\overline{\Prism}_R \otimes_R F_M \to M/IM$ is a monomorphism.  
\end{enumerate}

 Let $\Fil^\bullet M$ be the prismatic weak $F$-gauge associated to $(M, \varphi_M) \in \mathrm{DM}^\mathrm{adm}(R)$ by the functor \cref{funcc}. In this situation, it follows that one can write $M \simeq W \oplus T$ for two projective $\Prism_R$-modules $W$ and $T$ such that we have $\Fil^1 M \simeq (\Fil^1_\Nyg\Prism_R \otimes W) \oplus 
 T$ (see proof of \cite[Lem.~4.1.23]{alb}). One may define a filtration 
\begin{equation}\label{e}
F^i M:= (\Fil^i _\Nyg\Prism_R \otimes W) \oplus 
(\Fil^{i-1}_\Nyg \Prism_R \otimes T),    
\end{equation}which can be naturally equipped with the structure of a prismatic weak $F$-gauge denoted as $F^\bullet M.$

 We will now check that the prismatic weak $F$-gauge $F^\bullet M$ is actually a prismatic $F$-gauge. We need to check that the Frobenius map $$ {F}^\bullet M \widehat{\otimes}_{\mathrm{Fil}^\bullet_{\mathrm{Nyg}}\Prism_R} I^\bullet \Prism_R \to I^\bullet M$$ is an isomorphism. In our case, one can directly compute the filtered tensor product appearing on the left hand side of the above map and see that this amounts to checking that the map
$$\theta_n: (W \otimes_{\Prism_R, \varphi} I^n \Prism_R) \oplus (T\otimes_{\Prism_R, \varphi} I^{n-1} \Prism_R) \to I^n M$$ is an isomorphism for all $n \in \mathbb Z.$
Note that, since $I$ is a principal ideal generated by a nonzerodivisor, the source and the target of these maps are projective $\Prism_R$-modules of the same rank. Therefore, it suffices to check that they are surjective. Further, by scaling, it suffices to show the surjectivity of $\theta_1.$ Note that the natural map $u: \Fil^1 M \otimes_{\Prism_R, \varphi} \Prism_R \to I M $ factors canonically through $\theta_1.$ Thus, it suffices to show that $u$ is surjective. But that follows from \cite[Rmk.~4.7]{alb} since $(M, \varphi_M)$ was assumed to be admissible.

Note that there are natural injective maps $F^iM \to \Fil^i M,$ which are isomorphisms. Indeed, to see this isomorphism, it suffices to show that the induced map $F^iM/F^{i+1} M \to I^i M/ I^{i+1}M$ is injective, which follows from the fact that $W$ and $T$ are projective modules (\textit{cf}.~\cite[Cor.~2.53]{guoli}). 

 In summary, we see that \cref{uset} refines to a fully faithful functor
\begin{equation}\label{foodishere}
{G}_{\Prism}': \mathrm{DM}^{\mathrm{adm}}(R) \to F\text{-Gauge}^{\mathrm{vect}}_{\Prism}(R).  
\end{equation}{}

\end{remark}{}

\begin{remark}\label{anslebras1}
 Let $M \in \mathrm{DM}^{\mathrm{adm}}(R).$ Using the explicit description in \cref{e} of $G_\Prism' (M)$ as a filtered module over $\Fil^\bullet \Prism_S$, and  \cref{guolii}~\cref{foh}, one sees that the pullback $j^* G_\Prism'(M)$ identifies with the graded module $W_0 \oplus T_0,$ where $W_0:= W \otimes_{\Prism_R} R$ and $T_0:= T \otimes_{\Prism_R} R,$ and $W_0$ (resp.~ $T_0$) is in degree $0$ (resp.~ degree $1$). Therefore, $G_\Prism' (M)$ has Hodge--Tate weights in $[0,1]$ and it follows that $G_\Prism'$ refines to a functor 
\begin{equation}\label{aa5}
G_\Prism: \mathrm{DM}^{\mathrm{adm}}(R) \to F\text{-Gauge}^{\mathrm{vect}}_{[0,1]}(R),   
\end{equation}where the latter denotes the full subcategory of $F\text{-Gauge}^{\mathrm{vect}}_{\Prism}(R)$ spanned by objects of Hodge--Tate weights in $[0,1].$
\end{remark}{}

\begin{proposition}\label{comppp}
 The functor $G_\Prism: \mathrm{DM}^{\mathrm{adm}}(R) \to F\mathrm{\text{-}{Gauge}}^{\mathrm{vect}}_{[0,1]}(R)$ from \cref{aa5} is an equivalence. \footnote{The proof of this proposition in the preprint ~\cite[Thm.~2.54]{guoli} has errors; namely, the isomorphisms in last two lines of page 24, loc. cit. do not hold.}
\end{proposition}{}

\begin{proof}
It suffices to show that for $M \in F\text{-Gauge}^{\mathrm{vect}}_{[0,1]}(R),$ the Breuil--Kisin module $Q:=\BK(M)$ \cref{funcc2} is an object of $\mathrm{DM}^{\mathrm{adm}}(R).$ Using \cref{guolii}~\cref{foh}, it follows that since $M$ has Hodge--Tate weights in $\ge 0,$ the $F$-gauge $M$ must be effective. This implies that the Frobenius $\varphi_Q$ on $Q$ refines to a map $\varphi^* Q \to Q.$ Further, since $M$ has Hodge--Tate weights in $[0,1],$ the pullback $j^*M$ identifies with a graded module $W_0 \oplus T_0,$ where $W_0, T_0$ are projective $R$-modules in degree $0$ and $1$, respectively. Let us use $\Fil^\bullet M$ to denote the filtered $\Fil^\bullet_\Nyg \Prism_R$ module underlying $M$. We set $N^\bullet:= \mathrm{gr}^\bullet M,$ which is a (locally free) graded module over $\mathrm{gr}^\bullet_{\Nyg} \Prism_R \simeq \Fil^\bullet_{\mathrm{conj}} \overline{\Prism}_R\left \{ \bullet \right\}.$ By \cref{guolii}, 
\begin{equation}\label{usee}
 \mathrm{gr}^\bullet M \otimes_{\mathrm{gr}^\bullet_{\Nyg} \Prism_R} R \simeq W_0 \oplus T_0.   
\end{equation}
One can view $P^\bullet:= N^\bullet \left \{-\bullet \right \}$ as a filtered object (as it is a graded module over $\Fil^\bullet_{\mathrm{conj}} \overline{\Prism}_R$), and because $N^\bullet$ is locally free as a graded module over $\mathrm{gr}^\bullet_\Nyg \Prism_R,$ it follows that the maps $P^i \to P^{i+1}$ are injective. Further, the underlying object of $P^\bullet,$ denoted by $P^\infty: = \mathrm{colim}\,P_i$ can be viewed as a projective $\overline{\Prism}_R$-module. By computing the graded tensor product \cref{usee}, we have canonical isomorphisms $N^0 \simeq W_0$ and $N^1 /\gr^1 \Prism_R N^0 \simeq T_0. $ By composition, we obtain a canonical map \begin{equation}\label{maptosplit}
    N^1 \to N^1 /\gr^1 \Prism_R N^0 \simeq T_0.
\end{equation} Suppose that $\tau: T_0 \to N^1  $ is any choice of a splitting (which exists since $T_0$ is projective) of the latter composite map. Using $\tau$, we obtain the maps
\begin{equation}\label{hurry}
C_\tau: ( \mathrm{gr}^i_\Nyg \Prism_R \otimes_R W_0) \oplus (\mathrm{gr}^{i-1}_\Nyg \Prism_R \otimes_R T_0) \to N^i. 
\end{equation} By \cref{usee}, it also follows that these maps are surjective. Indeed, the surjectivity for $i \in \left \{0,1 \right\}$ follows from construction, and for $i \ge 2$, it follows inductively from unwrapping the tensor product in \cref{usee} and using the fact that the summand in degree $i$ in the right hand side of \cref{usee} is zero. $C_\tau$ leads to a map of filtered objects
$$C'_\tau: ( P^i \otimes_R W_0) \oplus (P^{i-1} \left \{-1 \right \}\otimes_R T_0) \to P^i $$

Passing to colimit, we obtain a surjective map 
\begin{equation}\label{saporo}
S_\tau: (  \overline{\Prism}_R \otimes_R W_0) \oplus (\overline{\Prism}_R \left \{-1\right \} \otimes_R T_0) \to P^\infty.   
\end{equation}Since both sides are projective modules of the same rank, the map \cref{saporo} must be an isomorphism. This implies that the maps \cref{hurry} are injective, and therefore also isomorphisms. In particular, by construction, we have a canonical map (which does not depend on the choice of splitting $\tau$) \begin{equation}\label{trashday}
    \overline{\Prism}_R \otimes_R W_0 \to P^\infty
\end{equation} which is injective.

Let $M^\w{u}$ be the underlying object of the $F$-gauge $M.$ Note that we have $M^\w{u} = \Fil^0 M,$ since we have shown that $M$ is effective. There is a natural surjective map $$M^\w{u}/ \Fil^1 \Prism_R M^\w{u} \simeq M^{\w{u}} \otimes_{\Prism_R} R \to \gr^0 M \simeq W_0,$$ whose kernel $K$ is a projective module surjecting onto $T_0$ via the natural diagram 
\begin{center}
    \begin{tikzcd}
\mathrm{Fil}^1 M \arrow[rr] \arrow[d]               &  & N^1 \arrow[d]                      \\
K \simeq \mathrm{Fil}^1M/\Fil^1 \Prism_R M^\w{u} \arrow[rr] &  & N^1 /\gr^1 \Prism_R N^0 \simeq T_0
\end{tikzcd}
\end{center}{}
Since $K$ and $T_0$ must have the same rank as projective modules, the natural map $K \to T_0$ is an isomorphism. Therefore, we obtain an exact sequence $$0 \to T_0 \to M^\w{u}\otimes_{\Prism_R} R \to W_0 \to 0.$$ By projectivity, we may split this sequence to get an isomorphism $\rho: W_0 \oplus T_0 \simeq M^u \otimes_{\Prism_R} R .$ Since $\Prism_R \to R$ is henselian (see \cite[Lem.~4.1.28]{alb}), we can lift it to an isomorphism $\rho': W \oplus T \simeq M^u,$ where $W$ and $T$ are projective $\Prism_R$-modules lifting $W_0$ and $T_0$. This gives an identification $\Fil^1 M \simeq (\Fil^1_{\Nyg}\Prism_R \otimes _{\Prism_R} W) \oplus T.$ In particular, we may choose a section $\rho'': T \to \Fil^1 M$ to the surjection $\Fil^1 M \to T$. This gives a composite map $T \xrightarrow{\rho''} \Fil^1 M \to N^1$ which factors as $\rho''':T_0 \simeq  T/\Fil^1 \Prism_R T \to N^1.$ Note that we have a commutative diagram

\begin{center}
\begin{tikzcd}
T \arrow[r, "\rho''"] \arrow[d]     & \Fil^1 M \arrow[r, two heads] \arrow[d]                        & T \arrow[d] \\
T_0 \arrow[rd] \arrow[r, "\rho'''"] & N^1 \arrow[d] \arrow[r, "\cref{maptosplit}"]                   & T_0         \\
                                    & {\Fil^1M/\Fil^1 \Prism_R M^u \simeq N^1/\gr^1 N^0,} \arrow[ru] &            
\end{tikzcd}
\end{center} which shows that $\tau:= \rho'''$ gives a splitting to \cref{maptosplit}.
Furthermore, using $\rho''$ we can define filtered maps 
\begin{equation}\label{who}
( \mathrm{Fil}^i_\Nyg \Prism_R \otimes_{\Prism_R} W) \oplus (\mathrm{Fil}^{i-1}_\Nyg \Prism_R \otimes_{\Prism_R} T) \to \Fil^i M.   
\end{equation}{}
Since the above maps induce isomorphism on the underlying objects and the maps in \cref{hurry} are isomorphisms (for $\tau:= \rho'''$), it follows that the maps \cref{who} are isomorphisms too.
The condition \cref{who1} of $M$ being an $F$-gauge now translates to the map
$$\theta_n: (W \otimes_{\Prism_R, \varphi} I^n \Prism_R) \oplus (T\otimes_{\Prism_R, \varphi} I^{n-1} \Prism_R) \to I^n M$$ being isomorphism for all $n \in \mathbb Z.$ In particular, $\theta_1$ is surjective. This implies that cokernel of the map $\varphi^* Q \to Q$ is killed by $I,$ where $Q = \mathrm{BK}(M).$ Note that $\theta_0, \theta_1$ being isomorphisms imply that $\Fil^1 M \simeq \left\{x \in Q\mid \varphi(x) \in IQ\right\}$. To check that $Q$ is an admissible prismatic Dieudonn\'e module, one needs to further verify that the image of the composition $Q \xrightarrow{\varphi_Q} Q \to Q/IQ$ is a finite locally free $R \simeq \mathrm{gr}^0_\Nyg \Prism_R$-module $F_Q$ and the induced map $\overline{\Prism}_R \otimes_R F_Q \to Q/IQ$ is an injection. Both of these follow from the fact that $F_Q$ naturally identifies with $Q/ \Fil^1 M \simeq \gr^0 M = N^0 \simeq W_0$ as an $R$-module, and the map $\overline{\Prism}_R \otimes_R F_Q \to Q/IQ$ identifies with the injective map \cref{trashday}, under the further identification of $P^\infty$ with $M^{\w{u}}\otimes_{\Prism_R} \overline{\Prism}_R \simeq M^\w{u}/ I M^{\w{u}} \simeq Q/IQ.$ 
\end{proof}{}

\subsection{Prismatic Dieudonn\'e $F$-gauges associated to $p$-divisible groups}\label{prismaticdie}
In this section, we begin by explaining how to reinterpret the work of \cite{alb} on $p$-divisible groups in terms of prismatic $F$-gauges. More specifically, we define the prismatic Dieudonn\'e $F$-gauge of a $p$-divisible group over a quasisyntomic ring. After that, we will define the prismatic Dieudonn\'e $F$-gauge of a finite locally free group scheme of $p$-power rank and prove fully faithfulness. Let us recall the following definition:

\begin{definition}
For a $p$-divisible group $G= \left \{ G_n\right \},$ the classifying stack of $G$
 denoted as $BG$ is defined to be $\varinjlim BG_n$ taken in the category of stacks.    
\end{definition}{}

Our main construction in this section can be stated in terms of the following definition:

\begin{definition}\label{prelimdef} Let $S$ be a quasiregular semiperfectoid ring. Let $G$ be a $p$-divisible group over $S.$ We define $\mathcal{M}(G)$ to be the effective prismatic weak $F$-gauge over $S$ obtained from $H^2 (\mathrm{Fil}^\bullet_{\mathrm{Nyg}}R\Gamma_{\Prism}(BG))$ (see \cref{fl}).
\end{definition}

We will prove that $\mathcal{M}(G)$ is in fact a prismatic $F$-gauge and moreover, a vector bundle when viewed as a quasi-coherent sheaf on $\mathrm{Spf}(S)^{\mathrm{syn}}.$ To this end, we will begin by an understanding of the cotangent complex of the classifying stack of a $p$-divisible group (cf.~\cite{luci}). At first, we assume that $p^N S =0,$ and we consider an $n$-truncated Barsotti--Tate group $G$ for $n \ge N$ over $S.$ Let $\ell_G$ denote the co-Lie complex of $G,$ which is a perfect complex of $S$-modules with Tor amplitude in $[0,1].$ Its dual $\ell^{\vee}_G$ is called the lie complex. We let $\omega_G:= H^0 (\ell_G),$ which is a finite locally free $S$-module, whose rank is called the dimension of $G.$

 By a result of Grothendieck \cite[App.]{Mess2}, one has $$\ell^{\vee}_G \simeq \tau_{\ge -1}R\mathcal{H}om_{\mathbb Z} (G^\vee, \mathbb{G}_a) \simeq  \tau^{\le 1}R\mathcal{H}om_{\mathbb Z} (G^\vee, \mathbb{G}_a).$$ Note that there is a natural map $\phi_n: \mathcal{E}xt^1 (G^\vee, \mathbb{G}_a) \to \mathcal{H}om (G^\vee, \mathbb{G}_a)$ obtained as follows (cf.~\cite[2.2.3,~2.2.4]{luci}): an element of $\mathcal{E}xt^1 (G^\vee, \mathbb{G}_a)$ determines a map $u: G^\vee \to \mathbb{G}_a[1].$ Applying $(\cdot)\otimes_{\mathbb Z}^L \mathbb Z/p^n$ and noting that $G^\vee$ is killed by $p^n,$ we obtain a map $G^\vee [1] \oplus G^\vee \to \mathbb{G}_a[2]\oplus \mathbb{G}_a[1];$ applying $\pi_1$ gives the desired map $\phi_n$. If $G$ is an $n$-truncated Barsotti--Tate group, then $\phi_n$ is an isomorphism (see \cite[Prop.~2.2.4]{luci}). Suppose now that $G=\left \{ G_n \right \}$ is a $p$-divisible group over $S.$ Let $f: G_{n+1} \to G_n$ be any map of group schemes. By construction, we have a commutative diagram

\begin{center}
    \begin{tikzcd}
{\mathcal{E}xt^1(G_n, \mathbb{G}_a)} \arrow[rr, "\phi_n"] \arrow[d, "f_1"] &  & {\mathcal{H}om(G_n, \mathbb{G}_a)} \arrow[d, "p f_0"] \\
{\mathcal{E}xt^1(G_{n+1}, \mathbb{G}_a)} \arrow[rr, "\phi_{n+1}"]        &  & {\mathcal{H}om(G_{n+1}, \mathbb{G}_a).}             
\end{tikzcd}
\end{center}{}

\begin{proposition}\label{cotangent}
    Let $S$ be a bounded $p^\infty$-torsion, $p$-complete ring and $G = \left \{ G_n \right \}$ be a $p$-divisible group over $S.$ Then, we have a natural isomorphism $$\mathbb L_{BG/S} := \varprojlim  (\mathbb{L}_{BG_n/S})^{\wedge}_p \simeq  \omega_G[-1],$$ where ${\omega_G}$ is a locally free $S$-module of finite rank.
\end{proposition}{}

\begin{proof}
First, we argue over $S/p^NS.$ The above commutative diagram applied to the maps induced by $``p": G_{n+1}^\vee \to G_n^\vee$ shows that the induced ind-object $\mathcal{E}xt^1(G_n^\vee, \mathbb{G}_a )$ is equivalent to zero (cf.~\cite[2.2.1]{luci}). By duality, this implies that the pro-object $H^{-1} (\ell_G)$ is zero. Since $\mathbb{L}_{BG_n} \simeq \ell_{G_n}[-1]$ (cf.~proof of \cite[Thm.~3.1]{totaro}),  we see that $\mathbb{L}_{BG_n}$ is naturally pro-isomorphic to $\omega_{G_n}[-1];$ since $\omega_G := \varprojlim \omega_{G_n}$ is locally free (see \cite[Prop.~2.2.1,~(b),~(c)]{luci}), this gives the claim in the case over $S/p^NS$. Since $S$ has bounded $p$-power torsion, the pro-objects (in the category of animated rings) $S \otimes_{\mathbb{Z}}^L \mathbb{Z}/p^k $ and $S/p^kS$ are pro-isomorphic, and our claim in the $p$-complete case follows from base change properties of cotangent complex, using \cite[Tag~0D4B]{stacks}, and taking limits.
\end{proof}{}

\begin{remark}
 Note that for any map $g: G_n \to G_{n+1},$ one has a similar diagram
\begin{center}
    \begin{tikzcd}
{\mathcal{E}xt^1(G_{n+1}, \mathbb{G}_a)} \arrow[rr, "\phi_{n+1}"] \arrow[d, "pg_1"] &  & {\mathcal{H}om(G_{n+1}, \mathbb{G}_a)} \arrow[d, "g_0"] \\
{\mathcal{E}xt^1(G_{n}, \mathbb{G}_a)} \arrow[rr, "\phi_{n}"]        &  & {\mathcal{H}om(G_{n}, \mathbb{G}_a).}            
\end{tikzcd}
\end{center}{}
Let $S$ be such that $p^NS = 0$. For a $p$-divisible group $G= \left \{ G_n \right \}$ over $S,$ the above diagram implies that the structure maps $i: G_n \to G_{n+1}$ induces a pro-system $\mathcal{H}om(G_n, \mathbb{G}_a)$ that is pro-zero. We have already seen in the above proof that the maps $``p": G_{n+1} \to G_n$ induces a direct system $\mathcal{E}xt^1(G_n, \mathbb{G}_a)$ that is equivalent to zero. However, the pro-system $\mathcal{E}xt^1(G_n, \mathbb{G}_a)$ induced by the maps $i: G_n \to G_{n+1}$ is nonzero; in fact, we have $\mathcal{E}xt^1(G, \mathbb{G}_a) \simeq \omega_{G^\vee}^{\vee},$ the latter will be denoted by $t_{G^\vee}.$
\end{remark}{}

\begin{proposition}\label{cotangentcalc}
    Let $R \to S$ be a perfectoid ring surjecting onto a quasiregular semiperfectoid algebra $S.$ Let $G$ be a $p$-divisible group over $S.$ Then we have a natural fiber sequence $$\mathbb{L}_{S/R} \otimes_S \mathcal{O}_{BG} \to \mathbb{L}_{BG/R}  \to \omega_G[-1].$$
\end{proposition}{}

\begin{proof}
Follows from the transitivity fiber sequence of cotangent complex associated to the maps $BG \to \mathrm{Spf}\, S \to \mathrm{Spf}\, R$ and \cref{cotangent}.
\end{proof}{}

Using the above proposition, one can fully describe the conjugate filtration on absolute Hodge--Tate cohomology $H^2_{\overline{\Prism}}(BG).$ Below, let $t_{G^\vee} := \mathcal{E}xt^1(G, \mathbb{G}_a) \simeq H^2 (BG, \mathcal{O})$, which is a locally free $S$-module since $\omega_{G^\vee}$ is locally free. First, we record a lemma.

\begin{lemma}\label{harddlemma}
    Let $G$ be a $p$-divisible group over a bounded $p^\infty$-torsion, $p$-complete ring $S$. Then $H^1(BG, \mathcal{O})$ and $H^3(BG, \mathcal{O})$ vanish.
\end{lemma}{}

\begin{proof}
 For $H^1 (BG, \mathcal{O})=0$, similar to \cref{pde}, it suffices to note that $\mathcal{H}om _{(S)_\qsyn}(G, \mathcal{O})=0$, which holds since $G$ is $p$-divisible and $\mathcal{O}$ is derived $p$-complete. For the vanishing of $H^3(BG, \mathcal{O})$, since $R^1\lim H^2(BG,\mathcal{O})/p^k =0$, we can reduce to the case where $S$ is killed by a power of $p$. Note that we have a convergent spectral sequence 
$$E_2^{i,j}= \text{Ext}_{(S)_\mathrm{qsyn}}^i (H^{-j}(\mathbb{Z}[B {G}]), \mathbb{G}_a)\implies H^{i+j}(BG, \mathcal{O}).$$ In order to prove that $H^3(BG, \mathcal{O})=0$, it suffices to show that $E^{i,j}_2=0$ for $i+j=3$ and $i,j \ge 0$. Note that $E^{3,0}_2 \simeq \mathrm{Ext}^3 _{(S)_\qsyn}(\mathbb{Z}, \mathbb{G}_a) \simeq H^3 (\mathrm{Spec}\, S, \mathcal{O})=0.$ Below, we give arguments for the other necessary vanishings.

\begin{enumerate}
    \item For $i=2$, $j=1$, we need to show that $\mathrm{Ext}^2_{(S)_\qsyn}(G, \mathbb{G}_a)=0,$ which follows from \cite[Cor.~2.2.7,~(ii),~(iii)]{luci}.

    \item For $i=1$, $j=2$, we need to prove that $\mathrm{Ext}^1_{(S)_\qsyn} (G \wedge G, \mathbb{G}_a)=0.$ To this end, we consider the short exact sequence $0 \to Q \to G \otimes_\mathbb{Z} G \to \wedge^2_\mathbb{Z} G \to 0.$ By construction, it follows that $p^2: Q \to Q$ is surjective since $G$ is $p$-divisible. In particular, we must have $\mathrm{Hom}(Q, \mathbb{G}_a)=0$ since $S$ is killed by a power of $p$. Thus, to prove $\mathrm{Ext}^1_{(S)_\qsyn} (G \wedge G, \mathbb{G}_a)=0$, by an exact sequence chase, it suffices to prove that $\mathrm{Ext}^1_{(S)_\qsyn} (G \otimes_\mathbb{Z} G, \mathbb{G}_a)=0$. Using the spectral sequence $$\text{Ext}_{(S)_\mathrm{qsyn}}^i (H^{-j}(G \otimes^L_\mathbb{Z} G), \mathbb{G}_a)\implies \mathrm{Ext}^{i+j}_{(S)_\qsyn}(G \otimes^L_\mathbb{Z}G, \mathbb{G}_a),$$ one sees that $\mathrm{Ext}^1_{(S)_\qsyn} (G \otimes_\mathbb{Z} G, \mathbb{G}_a)$ is a subquotient of $\mathrm{Ext}^1_{(S)_\qsyn} (G \otimes^L_\mathbb{Z} G, \mathbb{G}_a)$. Now, $\mathrm{Ext}^1_{(S)_\qsyn} (G \otimes^L_\mathbb{Z} G, \mathbb{G}_a)\simeq \mathrm{Ext}^1_{(S)_\qsyn} (G ,R\mathcal{H}om (G, \mathbb{G}_a)) \simeq \mathrm{Hom}_{(S)_\qsyn}(G ,R\mathcal{H}om (G, \mathbb{G}_a)[1]).$ Since $\mathcal{H}om(G, \mathbb{G}_a)=0$ and $\mathcal{E}xt^1(G, \mathbb{G}_a) \simeq \omega_{G^\vee}$, $\mathrm{Hom}_{(S)_\qsyn}(G ,R\mathcal{H}om (G, \mathbb{G}_a)[1]) \simeq \mathrm{Hom}_{(S)_\qsyn}(G, \omega_{G^\vee})$, which is zero since $G$ is $p$-divisible and $\omega_{G^\vee}$ is killed by a power of $p$. This gives the desired vanishing.

    \item For $i=0, j=3$, we need to show that $\mathrm{Hom}_{(S)_\qsyn} (H^{-3}(\mathbb{Z}[BG]), \mathbb{G}_a)=0.$ By \cref{socurious}, multiplication by $p$ map induces a certain surjective map $[p]_3$ on $H^{-3}(\mathbb{Z}[BG])$. By \cite[\S~6.2]{h3breen}, we have an exact sequence 
    \begin{equation}\label{breenseq}
       0 \to \wedge^3_{\mathbb{Z}}G \to H^{-3}(\mathbb{Z}[BG]) \to L_1 \wedge^2G \to 0,  
    \end{equation} where the term in the right denotes a derived functor. Since multiplication by $p^3$ is surjective on $\wedge^3_\mathbb{Z} G$, we have $\mathrm{Hom}_{(S)_\qsyn}(\wedge^3_\mathbb{Z} G, \mathbb{G}_a)=0$. Further, multiplication by $p$ on $G$ induces the map $p^2: L_1 \wedge^2G \to L_1 \wedge^2G $, which is surjective by the above exact sequence since $[p]_3$ is surjective. Therefore, $\mathrm{Hom}_{(S)_\qsyn}( L_1 \wedge^2G, \mathbb{G}_a)=0$. This gives the desired vanishing.
    
\end{enumerate}
This finishes the proof.\end{proof}{}
The following lemma was used in the above proof.

\begin{lemma}\label{socurious}
 Let $G$ be a $p$-divisible group over a bounded $p^\infty$-torsion, $p$-complete ring $S$. Then the map $[p]_3: H^{-3}(\mathbb{Z}[BG]) \to  H^{-3}(\mathbb{Z}[BG])$ induced by $p: G \to G$ is surjective.   
\end{lemma}{}

\begin{proof}
 Since the topos $\mathcal{X}$ associated to $(S)_\qsyn$ can be generated by affine schemes, it is (locally) coherent, and by Deligne's theorem (\cite[Exposé VI, p. 336]{AGV}), it has enough points. Therefore, it suffices to check surjectivity of $[p]_3$ after taking stalks at an arbitrary (geometric) point $x: \mathrm{Sets} \to \mathcal{X}$. We may view $G$ as an abelian group object of $\mathcal{X}$ and let $A:= G_x$. Since $G \simeq \varinjlim G[p^n]$, it follows that every element of the abelian group $A$ is killed by a power of $p$. Further, $p: A \to A$ is also surjective.

 We claim that there is an isomorphism $A \simeq (\mathbb{Q}_p/\mathbb{Z}_p)^{\oplus I}$ for some index set $I.$ To this end, note that if $\ell \ne p$ is a prime, then $\ell: A \to A$ is surjective. In order to see this, let $a \in A$. Then there exists $N$ such that $p^Na  =0.$ Since $\mathrm{gcd}(\ell, p^N)=1$, we have $u \ell + p^Nv=1$ for $u,v \in \mathbb{Z}$. Therefore, $a=  (u\ell + p^Nv )a= \ell u a$. Setting $b:= ua$ gives an element such that $\ell b = a$, as desired. Since $p: A \to A$ is also surjective, it follows that $A$ is a divisible group. By the classification of divisible groups (e.g., see \cite[\S~24]{MR233887}), $A$ is a direct sum of possibly infinitely many copies of $\mathbb{Q}$ and the groups $\mathbb{Q}_\nu/\mathbb{Z}_\nu$, where $\nu$ is some prime number. Since every element of $A$ is $p$-power torsion, $A$ must be a direct sum of copies of $\mathbb{Q}_p/\mathbb{Z}_p$, as claimed.

 Since taking stalks is an exact functor, by the first paragraph, it suffices to show that the map $[p]_3: H_3 (A, \mathbb{Z}) \to H_3 (A, \mathbb{Z})$ on group homology induced by $p: A \to A$ is surjective. Since $A \simeq (\mathbb{Q}_p/\mathbb{Z}_p)^{\oplus I}$, by considering filtered colimits, we may reduce to the case when the index set $I$ is finite. First, we show that the map $[p]_3: H_3 (\mathbb{Q}_p/\mathbb{Z}_p, \mathbb{Z}) \to H_3(\mathbb{Q}_p/\mathbb{Z}_p, \mathbb{Z})$ induced by $p: \mathbb{Q}_p/\mathbb{Z}_p \to  \mathbb{Q}_p/\mathbb{Z}_p$ is surjective. By writing $\mathbb{Q}_p/\mathbb{Z}_p \simeq \varinjlim _k \mathbb{Z}/p^k\mathbb{Z}$ and using the classical calculation of group homology of cyclic groups via the Tate complex, it follows that $H_3 (\mathbb{Q}_p/\mathbb{Z}_p, \mathbb{Z}) \simeq \mathbb{Q}_p/\mathbb{Z}_p.$ Further, one sees (e.g., by using \cref{breenseq} in the case of cyclic groups) that $[p]_3$ identifies with $p^2: \mathbb{Q}_p/\mathbb{Z}_p \to \mathbb{Q}_p/\mathbb{Z}_p$, which is indeed surjective. In fact, by a similar argument, we have that $H_2 (\mathbb{Q}_p/\mathbb{Z}_p, \mathbb{Z})=0, H_1 (\mathbb{Q}_p/\mathbb{Z}_p, \mathbb{Z}) \simeq \mathbb{Q}_p/\mathbb{Z}_p,$ and the induced map $[p]_i: H_i (\mathbb{Q}_p/\mathbb{Z}_p, \mathbb{Z}) \to H_i(\mathbb{Q}_p/\mathbb{Z}_p, \mathbb{Z})$ is surjective for all $i \le 3.$
 
We will argue by induction on $m$ that $[p]_3: H_3 (A, \mathbb{Z}) \to H_3 (A, \mathbb{Z})$ is surjective for $A:= (\mathbb{Q}_p/\mathbb{Z}_p)^m$. The previous paragraph treats the case $m=1$, so we assume that $m \ge 2$ and let $B:= (\mathbb{Q}_p/\mathbb{Z}_p)^{m-1}.$ By the K\"unneth formula, we have a short exact sequence
$$0 \to \bigoplus_{i+j=3}H_i(B, \mathbb{Z}) \otimes_\mathbb{Z} H_j(\mathbb{Q}_p/\mathbb{Z}_p, \mathbb{Z}) \to H_3 (A, \mathbb{Z}) \to \bigoplus_{i+j=2} \mathrm{Tor}_1^{\mathbb{Z}} (H_i(B,\mathbb{Z}), H_j(\mathbb{Q}_p/\mathbb{Z}_p, \mathbb{Z})) \to 0.$$
Since $p: B \to B$ induces a surjective map on $H_2(B, \mathbb{Z}) \simeq \wedge^2 B$, by induction we conclude that $p: A \to A$ induces a surjection on $\bigoplus_{i+j=3}H_i(B, \mathbb{Z}) \otimes_\mathbb{Z} H_j(\mathbb{Q}_p/\mathbb{Z}_p, \mathbb{Z})$. Therefore, it would suffice to show that the induced map on $\bigoplus_{i+j=2} \mathrm{Tor}_1^{\mathbb{Z}} (H_i(B,\mathbb{Z}), H_j(\mathbb{Q}_p/\mathbb{Z}_p, \mathbb{Z})$ is surjective. To this end, we reduce to showing that the induced map on $\mathrm{Tor}_1^{\mathbb{Z}} (B, \mathbb{Q}_p/\mathbb{Z}_p) \simeq \varinjlim B[p^n] \simeq B = (\mathbb{Q}_p/\mathbb{Z}_p)^{m-1}$ is surjective. However, this induced map identifies with $p^2: (\mathbb{Q}_p/\mathbb{Z}_p)^{m-1} \to (\mathbb{Q}_p/\mathbb{Z}_p)^{m-1}$ which is surjective. This finishes the proof.
\end{proof}{}

\begin{proposition}[Hodge--Tate sequence]\label{hodgetate}
    Let $G$ be a $p$-divisible group over a quasiregular semiperfectoid ring $S.$ Then there is a canonical exact sequence 

$$0 \to \overline{\Prism}_S \otimes_S t_{G^\vee} \to H^2_{\overline{\Prism}}(BG) \to (\overline{\Prism}_S \otimes_S \omega_G) \left \{-1 \right \} \to 0. $$
\end{proposition}{}

\begin{proof}
 Note that $H^2_{\Prism}(BG)$ can be computed as relative prismatic cohomology of $BG_{\overline{\Prism}_S}$ with respect to the canonical prism $(\Prism_S, I)$ associated to $S.$ The claim now follows from the conjugate filtration on relative Hodge--Tate cohomology (which is defined canonically) and \cref{cotangent}. Indeed (working over $\overline{\Prism}_S$), we have $H^2 (BG, (\wedge^i \omega_G[-1])[-i])=0$ for $i \ge 2$ and $H^1 (BG,  \omega_G[-2])=0$. Thus the conjugate filtration (cf.~\cref{cotangent}) gives an exact sequence $$0 \to \overline{\Prism}_S \otimes_S t_{G^\vee} \to H^2_{\overline{\Prism}}(BG) \to (\overline{\Prism}_S \otimes_S \omega_G) \left \{-1 \right \} \to 0,$$ where the surjectivity on the right follows from $H^3 (BG, \mathcal{O})=0$ (\cref{harddlemma}).
\end{proof}{}

\begin{remark}\label{badwriter}
 The Hodge--Tate sequence above sequence can be split non-canonically. It also implies that $H^2_{\overline{\Prism}} (BG)$ is a locally free $\overline{\Prism}_S$-module of rank $\w{height}(G),$ as $\mathrm{height}(G) = \dim (G) + \dim(G^\vee)$. Using the fact that $H^3(BG, \mathcal{O})=0$ (resp. $H^1(BG, \mathcal{O})=0$) and \cref{cotangent}, one can obtain that $H^3_{\overline{\Prism}}(BG)=0$ (resp. $H^1_{\overline{\Prism}}(BG)=0$), as explained below.
\end{remark}{}

\begin{lemma}\label{badwriter11}
  Let $G$ be a $p$-divisible group over a quasiregular semiperfectoid ring $S.$ Then $H^1_{\overline{\Prism}}(BG)$ and $H^3_{\overline{\Prism}}(BG)$ vanish.  
\end{lemma}

\begin{proof}
  Similar to \cref{hodgetate}, we may compute the desired Hodge--Tate cohomology groups as relative Hodge--Tate cohomology of $BG_{\overline{\Prism}_S}$ with respect to the canonical prism $(\Prism_S, I)$ associated to $S$. By using the conjugate filtration, arguing in a manner similar to \cref{hodgetate}, the vanishing $H^1_{\overline{\Prism}}(BG)=0$ follows as a consequence of $H^1 (BG, \mathcal{O})=0$ (\cref{harddlemma}) and $H^1 (BG,  \omega_G[-2])=0$. The assertion $H^3_{\overline{\Prism}}(BG)=0$ is also proven using a similar strategy. Indeed (working over $\overline{\Prism}_S$), since $H^2 (BG, (\wedge^i \omega_G[-1])[-i])$ and $H^3 (BG, (\wedge^i \omega_G[-1])[-i])$ both vanish for $i \ge 2$, the relative conjugate filtration gives an exact sequence
  $$ H^3 (BG, \mathcal{O}) \to H^3_{\overline{\Prism}}(BG) \to H^3(BG, \omega_G[-2]) .$$ Now, $ H^3(BG, \omega_G[-2]) \simeq H^1 (BG, \omega_G) \simeq \mathrm{Hom}(G, \omega_G)=0$ since $G$ is $p$-divisible and $\omega_G$ is derived $p$-complete. As $H^3(BG, \mathcal{O})=0$ (\cref{harddlemma}), we obtain the desired vanishing.
\end{proof}{}

\begin{proposition}\label{locfree}

     Let $G$ be a $p$-divisible group over a quasiregular semiperfectoid ring $S.$ Then $H^3_\Prism(BG)=0$ and  $H^2_{\Prism}(BG)$ is a locally free $\Prism_S$-module of rank equal to $\mathrm{height}(G).$
\end{proposition}

\begin{proof}
Let $(\Prism_S, I)$ be the prism associated to $I.$ In this situation, $I = (d)$ for some element $d.$ Using the fact that $R\Gamma_{\Prism}(BG)$ is derived $d$-complete, by a limit argument, one deduces that $H^3_{\Prism} (BG)=0.$ Indeed, since $H^3_{\overline{\Prism}}(BG)=0$ (\cref{badwriter}), by induction, it follows that $H^3 (R\Gamma_{\Prism}(BG) \otimes _{\Prism_S} (\Prism_S/d^k))=0.$ By derived $d$-completeness, $R\Gamma_{\Prism} (BG) \simeq \varprojlim R\Gamma_{\Prism} (BG) \otimes_{\Prism_S} (\Prism_S/d^k).$ Using Milnor sequences, in order to show that $H^3_{\Prism}(BG) \simeq \varprojlim_k H^3 (R\Gamma_{\Prism}(BG) \otimes _{\Prism_S} (\Prism_S/d^k))$, it suffices to prove that $$R^1 \varprojlim_k H^2 (R\Gamma_{\Prism}(BG) \otimes _{\Prism_S} (\Prism_S/d^k))=0.$$ Note that we have a fiber sequence
$$R\Gamma_{\Prism}(BG)\otimes_{\Prism_S} (\Prism_S/d) \to R\Gamma_{\Prism}(BG)\otimes_{\Prism_S} (\Prism_S/d^k) \to R\Gamma_{\Prism}(BG)\otimes_{\Prism_S} (\Prism_S/d^{k-1}).$$ This implies that the map $H^2 (R\Gamma_{\Prism}(BG) \otimes _{\Prism_S} (\Prism_S/d^k)) \to H^2 (R\Gamma_{\Prism}(BG) \otimes _{\Prism_S} (\Prism_S/d^{k-1}))$ is a surjection, which gives the desired $R^1\lim$-vanishing. This shows that $H^3_{\Prism}(BG) =0$.

Now, using the fact that $H^1_{\overline{\Prism}} (BG)=0$ (\cref{badwriter11}), $H^3_{\Prism}(BG)=0$, and the universal coefficient theorem, one sees that $H^2_{\Prism}(BG) \otimes_{\Prism_S}^L ({\Prism}_S/d) = H^2_{\overline{\Prism}}(BG)$ and $H^2_{\Prism}(BG) \simeq \varprojlim_k H^2_{\Prism}(BG)/d^k$. By \cref{hodgetate}, $H^2_{\overline{\Prism}}(BG)$ is a locally free $\overline{\Prism}_S$-module of rank $\mathrm{height}(G).$ Therefore, invoking \cite[Tag~0D4B]{stacks} finishes the proof.
\end{proof}{}

\begin{corollary}[Base change]\label{basechangedm}
Let $S \to S'$ be a map of quasiregular semiperfectoid rings and let $G$ be a $p$-divisible group over $S$. Then we have a natural isomorphism $H^2_\Prism(BG) \otimes_{\Prism_S} \Prism_{S'} \simeq H^2_{\Prism}(BG_{S'}).$   
\end{corollary}

\begin{proof}
By the last paragraph in the proof of \cref{locfree}, $H^2_\Prism (BG) \otimes_{\Prism_S} \overline{\Prism}_{S} \simeq H^2_{\overline{\Prism}} (BG)$ (and similarly over $S'$). Thus, in order to prove that the natural map  $H^2_\Prism(BG) \otimes_{\Prism_S} \Prism_{S'} \to H^2_{\Prism}(BG_{S'})$ is an isomorphism, by derived completenss of a prism with respect to its ideal, it suffices to prove that the natural map $H^2_{\overline{\Prism}}(BG) \otimes_{\overline{\Prism}_S} \overline{\Prism}_{S'} \to H^2_{\overline{\Prism}}(BG_{S'})$ is an isomorphism. However, that follows from the Hodge--Tate sequence (\cref{hodgetate}) and five lemma.
\end{proof}{}


Before we proceed further, we note the following notations which will be used later on.

\begin{notation}\label{notation11}
 Let $G$ be a $p$-divisible group over a quasiregular semiperfectoid ring $S$, with a fixed choice of a perfectoid ring $R$ mapping surjectively onto $S$. Let $R\Gamma_{\Prism}(BG)$ denote the absolute prismatic cohomology. Below, we use the following notations: $F^k_{\w{N}}:= H^2 (\mathrm{Fil}^k_{\mathrm{Nyg}}R\Gamma_{\Prism}(BG))$ and
$F^k_{\w{conj}}:= H^2 (\mathrm{Fil}^k_{\mathrm{conj}}R\Gamma_{\overline{\Prism}}(BG)).$ We point out that while $R\Gamma_{{\Prism}}(BG),$ $ \Fil^k_{\mathrm{Nyg}}R\Gamma_{{\Prism}}(BG), R\Gamma_{\overline{\Prism}}(BG)$ are defined independently of $R$, the conjugate filtration $F^k_{\mathrm{conj}}$ is constructed as the conjugate filtration relative to the choice of $R$ (see \cite[Def.~12.1]{BS19}).
\end{notation}{}
\begin{proposition}\label{laaa}Let $R \to S$ be a perfectoid ring $R$ mapping surjectively onto the quasiregular semiperfectoid algebra $S$. In the above notations,
the natural maps $F_{\mathrm{conj}}^{k-1} \to F_{\mathrm{conj}}^{k}$ are injective and there is a natural exact sequence
$$0\to \wedge^k \mathbb{L}_{S/R}[-k] \otimes_S t_{G^\vee} \to (F_{\mathrm{conj}}^{k}/F_{\mathrm{conj}}^{k-1} ) \left \{k \right \}   \to \wedge^{k-1} \mathbb{L}_{S/R}[-k+1] \otimes_S \omega_{G} \to 0.$$
\end{proposition}{}

\begin{proof}
First, we claim that $H^2 (BG, \wedge^k \mathbb{L}_{S/R}[-k]\otimes_S \mathcal{O}_{BG}) \simeq t_{G^\vee} \otimes_S \wedge^k \mathbb{L}_{S/R}[-k] .$ To see this, we observe that since $G$ is $p$-divisible, $$H^1 (BG_{S/p^m}, M\otimes_S S/p^m \otimes_{S/p^m} \mathcal{O} ) \simeq \mathrm{Hom}_{(S/p^m)_\qsyn}(G,  M \otimes_S S/p^m \otimes_{S/p^m} \mathbb{G}_a )=0,$$ where $M:=\wedge^k \mathbb{L}_{S/R}[-k]$. Therefore, by $p$-completeness (and $R^1\lim$ vanishing of the above term), we have 
$$H^2(BG, M \otimes_S \mathcal{O}) \simeq \varprojlim_m H^2 (BG_{S/p^m}, M \otimes_S S/p^m \otimes_{S/p^m} \mathcal{O}).$$ Now $M \otimes_S S/p^m$ is a flat $S/p^m$-module (cf.~\cite[Rmk~4.21]{BMS2}) and we have a natural isomorphism $ H^2 (BG_{S/p^m}, M \otimes_S S/p^m \otimes_{S/p^m} \mathcal{O}) \simeq \mathrm{Ext}^1_{(S/p^m)_\qsyn}(G, M \otimes_S S/p^m \otimes_{S/p^m} \mathbb{G}_a)$. Let $G_j:= G[p^j]$. Since $\mathrm{Hom}(G_j,M \otimes_S S/p^m \otimes_{S/p^m} \mathbb{G}_a )$ is pro-zero (as the target is killed by $p^m$), it follows that 
$$\mathrm{Ext}^1(G, M \otimes_S S/p^m \otimes_{S/p^m} \mathbb{G}_a) \simeq \varprojlim_j \mathrm{Ext}^1(G_j, M \otimes_S S/p^m \otimes_{S/p^m} \mathbb{G}_a).$$ By \cite[Cor.~2.2.7,~(iii)]{luci}, the right term above is naturally isomorphic to $\mathrm{Ext}^1(G_m, M \otimes_S S/p^m \otimes_{S/p^m} \mathbb{G}_a)$. By \cref{uselemmafr}, it follows that for any direct system $N_r$ of $S/p^m$-modules, we have $\varinjlim_r \mathrm{Ext}^1(G_m, N_r \otimes_{S/p^m} \mathbb{G}_a) \simeq \mathrm{Ext}^1(G_m, \varinjlim_r N_r \otimes_{S/p^m} \mathbb{G}_a).$ Since $M \otimes_S S/p^m$ is a flat $S/p^m$-module, by writing it as a direct limit of finite free modules, we conclude that $\mathrm{Ext}^1(G_m, M \otimes_S S/p^m \otimes_{S/p^m} \mathbb{G}_a) \simeq \mathrm{Ext}^1(G_m,  \mathbb{G}_a) \otimes_{S/p^m} (M \otimes_S S/p^m) \simeq t_{G^\vee} \otimes_S (M \otimes_S S/p^m).$ Therefore, by taking limits, we obtain a natural isomorphism $H^2 (BG, M \otimes_S \mathcal{O}) \simeq  t_{G^\vee}\otimes_S \wedge^k\mathbb{L}_{S/R} [-k].$

Now, we have a fiber sequence 

\begin{equation}\label{tallash}
    \mathrm{Fil}^{k-1}_{\mathrm{conj}}R\Gamma_{\overline{\Prism}}(BG) \to \mathrm{Fil}^k_{\mathrm{conj}}R\Gamma_{\overline{\Prism}}(BG) \to \mathrm{gr}^k_{\mathrm{conj}}R\Gamma_{\overline{\Prism}}(BG).
\end{equation}

Note that $\mathrm{gr}^k_{\mathrm{conj}}R\Gamma_{\overline{\Prism}}(BG) \left \{ k \right \} \simeq \wedge^k \mathbb{L}_{BG/R}[-k]$, and by \cref{cotangentcalc}, the latter has a canonical, finite (increasing) filtration where the $v$-th graded piece is given by  
$\wedge^{k-v} \mathbb{L}_{S/R} [-k+v] \otimes_S \mathrm{Sym}^v \omega_G[-2v].$

From this filtration, it follows that we have a natural short exact sequence $$0 \to \wedge^k \mathbb{L}_{S/R}[-k] \otimes_S t_{G^\vee} \to H^2(\mathrm{gr}^k_{\mathrm{conj}}R\Gamma_{\overline{\Prism}}(BG)) \left \{k \right \}\to \wedge^{k-1} \mathbb{L}_{S/R}[-k+1] \otimes_S \omega_{G}\to 0.$$ From the above filtration, we also note that $H^1(\mathrm{gr}^k_{\mathrm{conj}}R\Gamma_{\overline{\Prism}}(BG))=0$; moreover, we also have $H^3(\mathrm{gr}^k_{\mathrm{conj}}R\Gamma_{\overline{\Prism}}(BG))=0$. In order to see this, note that using the finite filtration on $\wedge^k \mathbb{L}_{BG/R}[-k]$ considered above, one reduces to showing that $H^3 (BG, \wedge^k \mathbb{L}_{S/R}[-k] \otimes_S \mathcal{O}_{BG})=0$. By $p$-completeness, since $R^1 \varprojlim H^2(BG_{S/p^m}, M \otimes_S S/p^m \otimes_{S/p^m} \mathcal{O} ) \simeq R^1 \varprojlim (t_{G^\vee} \otimes_S M  \otimes_S S/p^m) = 0$, we have 
$$H^3 (BG, \wedge^k \mathbb{L}_{S/R}[-k] \otimes_S \mathcal{O}_{BG}) \simeq \varprojlim H^3 (BG_{S/p^m}, M \otimes_S S/p^m \otimes_{S/p^m} \mathcal{O}).$$ Vanishing of the individual terms in the limit above follows the same way as in the proof of \cref{harddlemma}, since the vanishing $\mathrm{Ext}^2 (G, M \otimes_S S/p^m \otimes \mathbb{G}_a)=0$ again follows from \cite[Cor.~2.2.7,~(ii),~(iii)]{luci}.

Using induction on $k$, \cref{harddlemma}, and \cref{tallash}, we obtain that $H^3 (\mathrm{Fil}^k_{\mathrm{conj}}R\Gamma_{\overline{\Prism}}(BG))=0$ for all $k.$
 Applying $H^2(\cdot)$ to \cref{tallash} now gives the desired claim.
\end{proof}{}
The following lemma was used in the proof above.

\begin{lemma}\label{uselemmafr}
Let $G$ be a flat, affine, commutative group scheme over a quasisyntomic ring $T$ killed by a power of $p$. Let $Q_r$ be a direct system of $T$-modules. Then
the natural map
$$\varinjlim_rR\mathrm{Hom}(G, Q_r \otimes_T \mathbb{G}_a) \to R\mathrm{Hom}(G, \varinjlim_r Q_r \otimes_T \mathbb{G}_a)$$
is an isomorphism, where the $R\mathrm{Hom}$ is taken in the derived category of abelian sheaves on $T$
\end{lemma}

\begin{proof}
  By using the Breen--Deligne resolution (see e.g., \cite[\S~2]{bbm} and \cite[Appx. to Lecture IV]{condensed}) along with the Dold--Kan correspondence, one can write $G$ as a colimit of a simplicial object where each term is a finite direct sum of $\mathbb{Z}[G^{\times m}]$. Since totalizations of coconnective objects commute with filtered colimits (see, e.g., \cite[Lem.~2.2.27]{MR4969359}), it suffices to show that the natural map 
  $$\varinjlim_r R\mathrm{Hom}(\mathbb{Z}[G^{\times m}], Q_r \otimes_T \mathbb{G}_a) \to R\mathrm{Hom}(\mathbb{Z}[G^{\times m}], \varinjlim_r Q_r \otimes_T \mathbb{G}_a)$$ is an isomorphism. Equivalently, we need to show that the natural map 
  $$\varinjlim_r R\Gamma_{\qsyn}(G^{\times m}, u^*Q_r) \to R\Gamma_{\qsyn}(G^{\times m}, u^* \varinjlim_r Q_r)$$ is an isomorphism, where $u: G \to \mathrm{Spf}\,T$ is the structure map of the group scheme. However, since $u$ is flat and $G$ is affine, the above holds true since the natural map $\varinjlim _r \mathcal{O}(G^{\times m}) \otimes_T Q_r \to \mathcal{O}(G^{\times m})\otimes _T \varinjlim_r Q_r$ is an isomorphism.
\end{proof}{}

\begin{remark}\label{useinprooflter}
By the last paragraph in the proof of \cref{laaa}, we have \linebreak  $H^3 (\mathrm{Fil}^k_{\mathrm{conj}}R\Gamma_{\overline{\Prism}}(BG))=0$. Similarly, one also obtains $H^1 (\mathrm{Fil}^k_{\mathrm{conj}}R\Gamma_{\overline{\Prism}}(BG))=0$. Using the fiber sequence $$ \Fil^{k+1}_{\w{Nyg}} R\Gamma_\Prism (BG) \to \Fil^k_{\w{Nyg}} R\Gamma_\Prism (BG) \to \Fil^k_{\w{conj}} R\Gamma_{\overline{\Prism}}(BG) \left \{k \right \},$$ we obtain that the natural map $F^{k+1}_{\w{N}} \to F^{k}_{\w{N}}$ is injective (see \cref{notation11}). Further, since $H^3_\Prism (BG)=0$ (\cref{locfree}), using the vanishing $H^3 (\mathrm{Fil}^k_{\mathrm{conj}}R\Gamma_{\overline{\Prism}}(BG))=0$, the above fiber sequence inductively implies that $H^3( \Fil^k_{\w{Nyg}} R\Gamma_\Prism (BG))=0$ for all $k$. Consequently, we have a short exact sequence $0 \to F^{k+1}_{\w{N}} \to F^{k}_{\w{N}} \to F^k_{\w{conj}} \left \{k \right \} \to 0. $
\end{remark}{}

\begin{lemma}[Breuil--Kisin property]\label{bkproperty}
 Let $G$ be a $p$-divisible group over a quasiregular semiperfectoid algebra $S$. Then the natural Frobenius map $\varphi^* H^2_{\Prism}(BG) \to H^2_{\Prism}(BG)$ is an isomorphism after inverting the ideal $I$ of $\Prism_S$.   
\end{lemma}{}
\begin{proof}
First, let us additionally assume that $S$ is perfectoid.
 Let $G_n:= G[p^n]$. Since $H^2_\Prism (BG_n)$ can be computed as relative prismatic cohomology of $(B{G_n})_{\overline{\Prism}_S}$ with respect to $(\Prism_S, I)$, and $BG_n$ is a smooth Artin stack, by \cite[Prop.~2.2.10]{MR4850408}, $$H^2 (\varphi^* R\Gamma_{\Prism}(BG_n)) \to H^2_\Prism (BG_n)$$ is an isomorphism after inverting $I$. Since $S$ is assumed to be perfectoid, $\varphi: \Prism_S \to \Prism_S$ is an isomorphism. Therefore, we obtain that $\varphi^*H^2_\Prism(BG_n) \to H^2_\Prism (BG_n)$ is an isomorphism after inverting $I$ (see \cite[Ex.~2.2.12]{MR4850408}). Moreover, by writing $I=(d)$, and choosing a smooth hypercover similar to the proof of \cite[Prop.~2.2.10]{MR4850408} combined with \cite[Cor.~15.5]{BS19} (for $i=2$), we obtain a natural operator $V$ on $H^2_\Prism (BG_n)$ such that $\varphi V = d^2.$ Now $H^1_{\Prism}(BG_n) \simeq \mathrm{Hom}_{(S)_\qsyn} (G_n, \Prism_{(\cdot)}) \subseteq H^0 _{\Prism}(G_n)$, and the latter is $p$-torsion free by \cite[Lem.~5.6]{alb} since $G_n$ is a syntomic affine scheme over $\Prism_S/d \simeq S$ and $\Prism_S \simeq A_{\mathrm{inf}}(S)$ is $p$-torsion free. As $G_n$ is killed by $p^n$, this implies that $H^1_{\Prism}(BG_n)=0.$ Therefore, $ H^2_{\Prism} (BG)  \simeq \varprojlim_n H^2_{\Prism} (BG_n)$ is also equipped with an operator $V$ such that $\varphi V = d^2.$ In particular, $\varphi^* H^2_{\Prism} (BG) \to H^2_{\Prism} (BG)$ is an isomorphism.

Now for a general quasiregular semiperfectoid algebra $S$, we choose a perfectoid ring $R$ mapping sujrectively onto $S$. As a consequence of the almost purity theorem, the $p$-completed henselization of $R$ at the ideal $\mathrm{Ker}(R \to S)$, denoted by $R'$, is perfectoid (see \cite[Cor.~2.10]{alb}). Thus, there exists a perfectoid ring $R'$ mapping surjectively onto $S$ such that $R' \to S$ is a henselian surjection. By \cite[Lem.~4.70]{alb} (which uses the results of (\cite[\S~2]{lauuu} and \cite{luci}), there exists a $p$-divisible group $G'$ over $R'$ such that $G'_S \simeq G.$ Thus, the lemma follows from the perfectoid case addressed in the above paragraph by applying base change (\cref{basechange}).
\end{proof}{}

\begin{construction}[Prismatic--de Rham comparison]\label{washing}
 Let $G$ be a $p$-divisible group over a quasiregular semiperfectoid ring $S$. By \cite[Cons.~5.4.9]{BhaLu} (and quasisyntomic descent), we obtain a natural composite map
 $$R\Gamma_{\Prism} (BG) \to R\Gamma_{\mathrm{dR}}(BG) \to \widehat{R\Gamma}_{\mathrm{dR}}(BG) ,$$ where the latter map comes from taking completion with respect to the Hodge filtration on ($p$-completed) derived de Rham cohomology. This defines a natural map of $\Prism_S$-modules
 \begin{equation}
  H^2_{\Prism} (BG) \to  \widehat{H}^2 _{\mathrm{dR}}(BG):= H^2 (\widehat{R\Gamma}_{\mathrm{dR}}(BG)), 
 \end{equation}where the right hand side is viewed as a $\Prism_S$-module via restriction of scalars along the natural map $\Prism_S \to S.$ This defines a natural (surjective) map
 \begin{equation}\label{derhamcomparisonmap}
     H^2_{\Prism} (BG) \otimes_{\Prism_S} S \to  \widehat{H}^2 _{\mathrm{dR}}(BG). 
 \end{equation}{} We will show that the above map is an isomorphism. Before that we need two lemmas.
\end{construction}

\begin{lemma}[Hodge filtration sequence]\label{hodgefilseqs}
In the situation of \cref{washing}, there is a canonical short exact sequence of $S$-modules
$$0 \to \omega_G \to \widehat{H}^2 _{\mathrm{dR}}(BG) \to t_{G^\vee} \to 0. $$
\end{lemma}{}

\begin{proof}
 We have a fiber sequence $$\Fil^1_{\mathrm{Hodge}}\widehat{R\Gamma}_{\mathrm{dR}}(BG) \to \widehat{R\Gamma}_{\mathrm{dR}}(BG) \to R\Gamma (BG, \mathcal{O}).$$ Let us denote $\Fil^i:=  \Fil^i_{\mathrm{Hodge}}\widehat{R\Gamma}_{\mathrm{dR}}(BG))$ and $\gr^i:= \gr^i_{\mathrm{Hodge}}\widehat{R\Gamma}_{\mathrm{dR}}(BG))$. Since $H^1(BG, \mathcal{O})=0$ (\cref{harddlemma}), it will suffice to show that $$H^2 (\Fil^1_{\mathrm{Hodge}}\widehat{R\Gamma}_{\mathrm{dR}}(BG)) \simeq \omega_G\,\,\, \mathrm{and}\,\,\,H^3 (\Fil^1_{\mathrm{Hodge}}\widehat{R\Gamma}_{\mathrm{dR}}(BG)) =0.$$ Since $\gr^1 \simeq R\Gamma (BG, \omega_G[-2])$ (cf.~\cref{cotangent}), similar to the proof of \cref{hodgetate} and \cref{badwriter11}, it will be sufficient to show that 
 \begin{equation}
  H^2 (\Fil^2) =0\,\,\, \mathrm{and}\,\,\, H^3 (\Fil^2)=0.   
 \end{equation} However, since $\Fil^i$ is a complete descending filtration by construction, the above vanishings follow from the vanishings 
$$ 
   H^2 (\gr^i) \simeq H^2 (BG, (\wedge^i \omega_G[-1])[-i])=0\,\,\, \mathrm{and} \,\,\,   H^3 (\gr^i) \simeq H^3 (BG, (\wedge^i \omega_G[-1])[-i])=0 
$$ for $i \ge 2$.
\end{proof}

\begin{lemma}\label{hodgecomsame}
  Let $G$ be a $p$-divisible group over an algebraically closed field $k$ of characteristic $p$. Then the natural map $H^2_{\mathrm{dR}}(BG) \to  \widehat{H}^2_{\mathrm{dR}}(BG)$ is an isomorphism.  
\end{lemma}

\begin{proof}
By \cref{uselatere} and using the Hodge filtration on $R\Gamma_{\mathrm{dR}}(BG)$, we have a natural diagram of short exact sequences
\begin{center}
\begin{tikzcd}
0 \arrow[r] & H^2(\Fil^1 R\Gamma_{\mathrm{dR}}(BG)) \arrow[rr] \arrow[d] &  & H^2_{\mathrm{dR}}(BG) \arrow[d] \arrow[rr] &  & {H^2(BG, \mathcal{O})} \arrow[r] \arrow[d, "\simeq"] & 0 \\
0 \arrow[r] & \omega_G \arrow[rr]                                        &  & \widehat{H}^2_{\mathrm{dR}}(BG) \arrow[rr] &  & t_{G^\vee} \arrow[r]                                 & 0
\end{tikzcd}
\end{center}
Thus, it suffices to prove that the left vertical map is an isomorphism. Let $\Fil^1_{\mathrm{H}}{\mathcal{O}}^{\mathrm{dR}}$ denote the sheaf on $(k)_\qsyn$ determined by the assignment $(k)_{\qrsp} \ni T \to \Fil^1_{\mathrm{Hodge}}{R\Gamma}_{\mathrm{dR}}(T).$ Then we have $\mathrm{Ext}^1_{(k)_\qsyn}(G, \Fil^1_{\mathrm{H}}{\mathcal{O}}^{\mathrm{dR}}) \simeq {H}^2(\Fil^1 R\Gamma_{\mathrm{dR}}(BG)).$ Therefore, for any two $p$-divisible groups $G_1$ and $G_2$ over $k$, the natural map 
${H}^2(\Fil^1 R\Gamma_{\mathrm{dR}}(BG_1)) \oplus{H}^2(\Fil^1 R\Gamma_{\mathrm{dR}}(BG_2)) \to {H}^2(\Fil^1 R\Gamma_{\mathrm{dR}}(B(G_1 \times G_2)))$ is an isomorphism. Therefore, using the same technique as in the proof of \cref{uselatere}, it suffices to check that the left vertical map is an isomorphism when $G := A[p^\infty]$ for some abelian variety $A$ over $k$. By the proof of \cref{hodgefilseqs}, it suffices to show that the natural map $ \Fil^1 R\Gamma_{\mathrm{dR}}(B A[p^\infty]) \to  \Fil^1 \widehat{R\Gamma}_{\mathrm{dR}}(B A[p^\infty])$ on Hodge filtration is an isomorphism. To this end (since the $\gr^0$ on both sides are naturally isomorphic), it is enough to show that the natural map $  R\Gamma_{\mathrm{dR}}(B A[p^\infty]) \to   \widehat{R\Gamma}_{\mathrm{dR}}(B A[p^\infty])$ is an isomorphism.  Since the natural maps $R\Gamma_{\mathrm{dR}}(B A)  \to R\Gamma_{\mathrm{dR}}(B A[p^\infty])$ and $\widehat{R\Gamma}_{\mathrm{dR}}(B A)  \to \widehat{R\Gamma}_{\mathrm{dR}}(B A[p^\infty])$ are isomorphisms (see the proof of \cite[Prop.~3.1, Lem.~3.1]{actaar}, cf.~proof of \cref{vectorbundleabsch}), it suffices to show that the natural map $  R\Gamma_{\mathrm{dR}}(B A) \to   \widehat{R\Gamma}_{\mathrm{dR}}(B A)$ is an isomorphism. By descent along $* \to BA$, it suffices to verify that the natural map $  R\Gamma_{\mathrm{dR}}( A^{\times k}) \to   \widehat{R\Gamma}_{\mathrm{dR}}(A^{\times k})$ is an isomorphism for $k \ge 1$. However, that follows because the Hodge filtration $\Fil^i$ on de Rham cohomology of an abelian variety is nonzero for only finitely many integers $i$. This finishes the proof.
\end{proof}{}

\begin{proposition}\label{usin}
 The map in \cref{derhamcomparisonmap} is an isomorphism.   
\end{proposition}
\begin{proof}
By \cref{locfree} and \cref{hodgefilseqs}, the source and target of the map \cref{derhamcomparisonmap} are both projective $S$-modules of rank equal to $\mathrm{height}(G).$ Therefore, it suffices to check isomorphism after base changing along an arbitrary map $S \to k$, where $k$ is an algebraically closed field of characteristic $p$. Note that \cref{hodgefilseqs} implies that the natural map $\widehat{H}^2_{\mathrm{dR}}(BG) \otimes_S k \simeq \widehat{H}^2_{\mathrm{dR}}(BG_k) $ is an isomorphism. Further, by \cref{basechangedm}, $$ H^2_{\Prism} (BG) \otimes_{\Prism_S} S \otimes_S k \simeq H^2_{\Prism} (BG) \otimes_{\Prism_S} \Prism_k \otimes_{\Prism_k} k \simeq H^2_{\Prism}(BG_k) \otimes_{\Prism_k} k.$$ Under the isomorphism $\Prism_k \simeq W(k)$, the map $\Prism_k \to k$ identifies with the composition $W(k) \xrightarrow{\varphi} W(k) \to W(k)/p \simeq k$ (see \cite[Thm.~12.2]{BS19} for $i=0$). Therefore, to prove the proposition, it suffices to show that the composition
$$H^2_{\mathrm{crys}}(BG_k)/p \to H^2_{\mathrm{dR}}(BG) \to \widehat{H}^2_{\mathrm{dR}}(BG)$$ is an isomorphism. This follows from \cref{makelight} and \cref{hodgecomsame}.
\end{proof}{}

\begin{corollary}\label{talktoday}
 Let $G$ be a $p$-divisible group over a quasiregular semiperfectoid ring $S$. Then there is a canonical short exact sequence of $S$-modules
 $$0 \to \omega_G \to H^2_{\Prism}(BG)\otimes_{\Prism_S}S \to t_{G^\vee} \to 0.$$
\end{corollary}{}

\begin{proof}
Follows from \cref{hodgefilseqs} and \cref{usin}.    
\end{proof}{}

\begin{proposition}[Dualizability]\label{dualo1}
Let $S$ be a quasiregular semiperfectoid algebra. Let $G$ be a $p$-divisible group over $S$ of height $h$. Then $\mathcal{M}(G)$ (see \cref{prelimdef}) is a prismatic $F$-gauge on $S$ and is a vector bundle of rank $h$ when viewed as a quasicoherent sheaf on $\mathrm{Spf}(S)^{\mathrm{syn}}$.
\end{proposition}

\begin{proof}
The proof will proceed via an explicit understanding of $F^\bullet_{\w{N}}:=H^2 (\mathrm{Fil}^\bullet_{\mathrm{Nyg}}R\Gamma_{\Prism}(BG)).$ To this end, since $H^1 (BG, \mathcal{O})=0$, we have a short exact sequence $0 \to F^1_{\w{N}} \to H^2_{\Prism}(BG) \to  t_{G^\vee} \to 0,$ where the surjectivity follows from the vanishing $H^3( \Fil^1_{\w{Nyg}} R\Gamma_\Prism (BG))=0$. Further, there is a natural $\Prism_S$-linear map $$\gamma^1: F_\w{N}^1 \to H^2 (BG, \mathbb{L}_{BG/S}[-1]) \simeq \omega_G$$ arising from the composition $$F^1_\w{N} \to F^1_{\w{conj}}\left \{1\right \} \to H^2 (\gr^1_{\mathrm{conj}}R\Gamma_{\overline{\Prism}}(BG))\left \{1\right \} \to H^2 (BG, \mathbb{L}_{BG/S}[-1]),$$ where the source is viewed as a $\Prism_S$-module via the (surjective) map $\Prism_S \to S$ (see \cref{laaa} and \cref{useinprooflter}).

By \cref{talktoday}, the induced map $H^2_{\Prism}(BG) \otimes_{\Prism_S}S \to t_{G^\vee}$ arising from above map $H^2_{\Prism}(BG) \to t_{G^\vee}$ fits into a natural short exact sequence $$0 \to \omega_G \to H^2_{\Prism}(BG)\otimes_{\Prism_S} S \to t_{G^\vee} \to 0.$$ We recall that the modules $w_{G}$ and $t_{G^\vee}$ are locally free. Let us choose a splitting $H^2_{\Prism} (BG) \otimes_{\Prism_S} S \simeq \omega_{G} \oplus t_{G^\vee}.$ Since the surjection $\Prism_S \to S$ is henselian (see \cite[Lem.~4.28]{alb}), it is possible to choose an isomorphism $H^2_{\Prism}(BG) \simeq W \oplus T$, such that $W \otimes_{\Prism_S} S \simeq \omega_G$, $T \otimes_{\Prism_S} S \simeq t_{G^\vee}$, and lifting the isomorphism $H^2_{\Prism} (BG) \otimes_{\Prism_S} S \simeq \omega_{G} \oplus t_{G^\vee}.$ It follows that under these identifications, $F^1_\w{N} = (\mathrm{Fil}^1_{\mathrm{Nyg}} \Prism_S \otimes_{\Prism_S} T) \oplus W.$ This induces a projection map $\pi: F^1_{\w{N}} \to W$, which admits a section $\tau_\pi: W \to F^1_{\w{N}}.$ We also have a map $\rho: T \to H^2_{\Prism}(BG)$ obtained by composing $T \xrightarrow{ (0, \mathrm{id})} W \oplus T \simeq H^2_{\Prism}(BG).$

Let us define $$G^k_{\w{N}}:= (\mathrm{Fil}^k_{\mathrm{Nyg}} \Prism_S \otimes_{\Prism_S}T) \oplus (\mathrm{Fil}^{k-1}_{\mathrm{Nyg}} \Prism_S \otimes_{\Prism_S}W).$$ Using the map $\rho$ and the section $\tau_\pi$, we obtain a map $G^k_{\w{N}} \to F^k_{\w{N}}$ of (decreasing) filtered objects. We will show that it is an isomorphism. Since the map we constructed induces isomorphism on underlying objects, we need to check that it induces isomorphism on graded pieces. Let us now choose a perfectoid ring $R$ mapping surjectively onto $S$. The graded pieces for the left hand side are given by 
\begin{equation}\label{claritin}
T^k:=\w{gr}^k(G^\bullet_{\w{N}}) \simeq (\mathrm{Fil}^k_{\w{conj}} \overline{\Prism}_S \left \{k\right \}\otimes_S t_{G^\vee}) \oplus (\mathrm{Fil}^{k-1}_{\w{conj}} \overline{\Prism}_S \left \{k-1\right \}\otimes_S \omega_{G}).
\end{equation}
By the short exact sequence in \cref{useinprooflter}, the graded pieces of the right hand side are given by $\gr^k (F^\bullet_{\w{N}}) \simeq F^k_{\w{conj}} \left \{k \right \}.$ Therefore, we obtain a map $T^k \left \{ -k \right \} \to {F}^k_{\mathrm{conj}}$ of (increasing) filtered objects, which we will show to be an isomorphism. Since the source and target of the map are both zero for $k<0$, it suffices to check isomorphism of the induced map on graded pieces. In other words, we need to prove that the induced map
\begin{equation}\label{toomuch}
 (\wedge^k \mathbb{L}_{S/R}[-k] \otimes_S t_{G^\vee} )\oplus   (\wedge^{k-1} \mathbb{L}_{S/R}[-k+1] \otimes_S \omega_{G}) \to (F^k_{\mathrm{conj}}/F^{k-1}_{\mathrm{conj}})\left \{k\right \} 
\end{equation}
is an isomorphism for $k \ge 0$. Note that by choice, the map $\rho$ induces a map $t_{G^\vee} \to H^2_{\Prism}(BG) \otimes_{\Prism_S} S$, which is a section to the canonical composite map
$$H^2_{\Prism}(BG) \otimes_{\Prism_S} S \to   H^2 (BG, \mathcal{O})) \simeq F^0_{\mathrm{conj}} \simeq t_{G^\vee}.$$
Thus, for $k=0$, the map in \cref{toomuch} is the natural isomorphism. Therefore, for $k \ge 1$, the map $(\wedge^k \mathbb{L}_{S/R}[-k] \otimes_S t_{G^\vee} ) \to (F^k_{\mathrm{conj}}/F^{k-1}_{\mathrm{conj}})\left \{k\right \} $ in \cref{toomuch} is the same as the canonical map in the short exact sequence in \cref{laaa}. By our choice of $\tau_\pi$, the induced map $\overline{\tau}: \omega_G \to (F^1_{\mathrm{conj}}/ F^0_{\mathrm{conj}})\left \{1\right\}$ in \cref{toomuch} (for $k=1$) is a section to the canonical map $\overline{\gamma}^1: (F^1_{\mathrm{conj}}/ F^0_{\mathrm{conj}})\left \{1\right\} \to \omega_G$. Since we have the following diagram
\begin{center}
\begin{tikzcd}
{\wedge^{k-1} \mathbb{L}_{S/R}[-k+1] \otimes_S (F^1_{\mathrm{conj}}/ F^0_{\mathrm{conj}})\left \{1\right\}} \arrow[rr] \arrow[rrd, "\mathrm{id}\otimes \overline{\gamma}^1"] &  & (F^k_{\mathrm{conj}}/F^{k-1}_{\mathrm{conj}})\left \{k\right \} \arrow[d, "\cref{laaa}"] \\
{\wedge^{k-1} \mathbb{L}_{S/R}[-k+1] \otimes_S \omega_{G}} \arrow[u, "\mathrm{id} \otimes \overline{\tau}"] \arrow[rr, "\mathrm{id}"]                                        &  & {(\wedge^{k-1} \mathbb{L}_{S/R}[-k+1] \otimes_S \omega_{G})}                            
\end{tikzcd}
\end{center} for $k \ge 1$, the induced map $(\wedge^{k-1} \mathbb{L}_{S/R}[-k+1] \otimes_S \omega_{G}) \to (F^k_{\mathrm{conj}}/F^{k-1}_{\mathrm{conj}})\left \{k\right \}$ in \cref{toomuch} gives a section to the surjection in \cref{laaa}. Thus the map \cref{toomuch} is an isomorphism for $k$. Therefore,
 $G^k_\mathrm{N} \to F^k_{\mathrm{N}}$ is an isomorphism; i.e., 
\begin{equation}\label{cocacola11}
F^k_{\w{N}} \simeq (\mathrm{Fil}^k_{\mathrm{Nyg}} \Prism_S \otimes_{\Prism_S}T) \oplus (\mathrm{Fil}^{k-1}_{\mathrm{Nyg}} \Prism_S \otimes_{\Prism_S}W).
\end{equation}{}

We will now verify that the prismatic weak $F$-gauge given by $H^2 (\mathrm{Fil}^\bullet_{\mathrm{Nyg}}R\Gamma_{\Prism}(BG))$ is indeed a prismatic $F$-gauge. To do so, we need to verify the condition \cref{who1}. In our situation, using \cref{cocacola11}, checking \cref{who1} amounts to checking that the canonical map of filtered objects
\begin{equation}
 \theta_n: (T \otimes_{\Prism_S, \varphi} I^n \Prism_S) \oplus (W\otimes_{\Prism_S, \varphi} I^{n-1} \Prism_S) \to I^n H^2_\Prism (BG)   
\end{equation}are isomorphisms. Note that since the Frobenius on prismatic cohomology $\varphi^* H^2_\Prism(BG) \to H^2_\Prism(BG)$ is an isomorphism after inverting $I$ (\cref{bkproperty}), the filtered map $\theta_\bullet$ induces an isomorphism on underlying objects. By \cref{below}, it suffices to check that the induced map on associated graded objects are isomorphic. However, that follows from the isomorphism $T^k \left \{ - k \right \} \to F^k_{\mathrm{conj}}$ shown above.

Since we have shown $\mathcal{M}(G)$ is a prismatic $F$-gauge over $S$, the fact that it is a vector bundle of height $h$ as a quasicoherent sheaf on $\Spf(S)^\syn$ can be checked by pulling back along $\Spf(S)^\Nyg \to \Spf(S)^\syn;$ but that follows from \cref{cocacola11}.
\end{proof}{}
The following lemma was used in the above proof.

\begin{lemma}\label{below}
 Let $\Fil^\bullet A$ be any $\mathbb{Z}$-indexed filtered object in a stable $\infty$-category $\mathcal{C}$ such that the underlying object and the associated graded object of $\Fil^\bullet A$ are both zero. Then $\Fil^\bullet A \simeq 0.$
\end{lemma}{}

\begin{proof}
 Since the associated graded object is zero, the maps in the diagram $\ldots \to \Fil^n A \to \Fil^{n-1} A \to \ldots$ are all isomorphisms. The underlying object being zero implies that colimit of the above diagram is zero. This implies that $\Fil^n A = 0.$    
\end{proof}{}

\begin{remark}\label{3van}
By \cref{useinprooflter}, for a $p$-divisible group $G$ over a quasiregular
semiperfectoid algbra $S$, we have  $H^3 (\mathrm{Fil}^k_{\mathrm{Nyg}}R\Gamma_{\Prism}(BG))=0$ for all $k.$ Using the $E_2$-spectral sequence (where the $\w{Ext}$-groups are computed in the quasisyntomic topos)
$$E_2^{i,j}=\mathrm{Ext}^i (H^{-j} (\mathbb{Z}[BG]), \mathrm{Fil}^k_{\mathrm{Nyg}}\Prism_{(\cdot)} ) \implies H^{i+j} (\mathrm{Fil}^k_{\w{Nyg}}(R\Gamma_{\Prism}(BG)),$$ the vanishing $H^3 (\mathrm{Fil}^k_{\mathrm{Nyg}}R\Gamma_{\Prism}(BG))=0$ implies that 
\begin{equation}\label{ext2van}
E_2^{2,1}=\mathrm{Ext}^2 (G, \mathrm{Fil}_{\w{Nyg}}^k\Prism_{(\cdot)})=0.
\end{equation}
To see this, note that since $\Fil^k_{\Nyg} \Prism_{(\cdot)}$ is derived $p$-complete, we have $\mathcal{H}om( G, \Fil^k_{\Nyg} \Prism_{(\cdot)})=0$. From this, we obtain that $E_2^{0,2}= \mathrm{Hom}(G \wedge G, \Fil^k_{\Nyg} \Prism_{(\cdot)})=0$; the latter follows from the vanishing $\mathcal{H}om( G, \Fil^k_{\Nyg} \Prism_{(\cdot)})=0$ in the same way as in the proof of \cref{imp}. Further, $E_2^{4,0}= \mathrm{Ext}^4 (\mathbb{Z}, \Fil^k_{\Nyg} \Prism_{(\cdot)} )= H^4 (\Fil^k_{\Nyg} \Prism_S)=0$, since $\Fil^k_{\Nyg} \Prism_S$ is discrete. Therefore, by the spectral sequence, we obtain that $E_2^{2,1} \simeq E_{\infty}^{2,1}$, which must be zero since $E_{\infty}^{2,1}$ is a subquotient of $H^3 (\mathrm{Fil}^k_{\mathrm{Nyg}}R\Gamma_{\Prism}(BG))$. Similarly, one also has 
\begin{equation}\label{ext1}
H^2 (\mathrm{Fil}^{\bullet}_{\w{Nyg}}R\Gamma_{\Prism} (BG)) \simeq \mathrm{Ext}^1(G, \mathrm{Fil}^{\bullet}_{\mathrm{Nyg}}\Prism_{(\cdot)}), 
\end{equation} as prismatic $F$-gauges.

\end{remark}{}

\begin{remark}[Base change]\label{pdivbase}
Let $S \to S'$ be a map of qrsp algebras. Let $G$ be a $p$-divisible group over $S$ and let $G_{S'}$ be the $p$-divisible group over $S'$ obtained by base change. Let $f: \mathrm{Spf}(S')^{\mathrm{syn}} \to \Spf(S)^\syn$ denote the natural map. By \cref{basechangedm} and the proof of \cref{dualo1}, one obtains that the natural map $f^* \mathcal{M}(G) \to \mathcal{M}(G_{S'})$ must be an isomorphism. Alternatively, since the source and the target of the natural map are both prismatic $F$-gauges, one can also invoke \cref{comppp} to conclude that the map is an isomorphism.  
\end{remark}{}
Finally, we can define the prismatic Dieudonn\'e $F$-gauge associated to a $p$-divisible group.
\begin{definition}[Dieudonn\'e $F$-gauge of a $p$-divisible group]\label{d}
 Let $T$ be a quasisyntomic ring and $G$ be a $p$-divisible group over $T.$ By \cref{pdivbase}, for every qrsp algebra $S$ over $T,$ we obtain a vector bundle $\mathcal{M}(G_S)$ that is compatible under pullback for every map $S \to S'$ of qrsp algebras over $T.$ By \cref{si}, this data descends to a prismatic $F$-gauge over $T$, which we denote as $\mathcal{M}(G)$ and call it the (prismatic) Dieudonn\'e $F$-gauge of $G.$ By \cref{dualo1}, it follows that $\mathcal{M}(G)$ is naturally a vector bundle over $\Spf(T)^\syn$ of rank equal to the height of $G.$
\end{definition}{}

\begin{remark}\label{l}
 By the above definition, if $T$ is a quasisyntomic ring and $S$ is a qrsp algebra over $T,$ then the pullback of the prismatic $F$-gauge $\mathcal{M}(G)$ to $\mathrm{Spf}\,(S)^\syn$ is canonically isomorphic to $\mathcal{M}(G_S).$ Since the pullback of a prismatic $F$-gauge to $\mathrm{Spf}(T)^{\mathrm{Nyg}}$ is a prismatic gauge in the sense of \cref{vanc}, by the first paragraph in the proof of \cref{equivalence}, it follows that the functor determined by
$$(T)_\qrsp \ni S \mapsto H^2 (\mathrm{Fil}^\bullet_{\mathrm{Nyg}}R\Gamma_{\Prism}(BG_S))$$ is a sheaf of filtered modules over $\Fil^\bullet \Prism_{(\cdot)}.$ 
\end{remark}{}

\begin{remark}[Compatibility with \cite{alb}]\label{iasbeer} As explained in \cref{ias1}, for a $p$-divisible group $G$ over a qrsp algebra, the prismatic $F$-gauge $\mathcal{M}(G)$ has Hodge--Tate weights in $[0,1].$ By \cref{comppp}, one can associate an admissible prismatic Dieudonn\'e module to $\mathcal{M}(G).$ By \cite[Thm.~1.6]{Mon1} this is isomorphic to the admissible prismatic Dieudonn\'e module associated to $G$ by \cite{alb}.
\end{remark}{}

\begin{example}
Let $T$ be a quasisyntomic ring and let us consider the $p$-divisible group $\mathbb{Q}_p/\mathbb{Z}_p$ over $T.$ The prismatic Dieudonn\'e $F$-gauge of $\mathbb{Q}_p/\mathbb{Z}_p$ is naturally isomorphic to the structure sheaf $\mathcal{O}_{T^{\syn}}$. This follows from \cref{ext1}.    
\end{example}{}

\begin{example}
Let $T$ be a quasisyntomic ring and let us consider the $p$-divisible group $\mu_{p^\infty}$ over $T.$ The prismatic Dieudonn\'e $F$-gauge of $\mu_{p^\infty}$ is naturally isomorphic to the Breuil--Kisin twist $\mathcal{O}\left \{-1\right \}.$ This follows from \cref{comppp}, \cref{dualo1} and \cite[Cor.~1.3]{actaar}.   
\end{example}{}

\subsection{Prismatic Dieudonn\'e $F$-gauges associated to abelian schemes}

Now we will discuss the construction of prismatic Dieudonn\'e $F$-gauges associated to abelian schemes. To this end, let $S$ be a qrsp algebra. Let $A$ be an abelian scheme over $S.$  Note that $\mathbb{L}_{BA/S} \simeq \Omega^1_{A/S} [-1]$. Further, since $p: A \to A$ is a surjection on $(S)_\qsyn$ and $\mathbb{G}_a$ is derived $p$-complete, one has $\mathcal{H}om_{(S)_\qsyn}(A, \mathbb
G_a)=0.$ Using the latter vanishing, similar to the proof of \cref{imp}, we obtain a natural isomorphism $H^2 (BA, \mathcal{O}) \simeq \mathrm{Ext}^1 (A, \mathbb{G}_a)$. By the moduli description of the dual abelian variety $A^\vee$, $\mathrm{Ext}^1 (A, \mathbb{G}_a)$ is naturally isomorphic to the tangent space of $A^\vee$ at the origin, which we note by $t_{A^\vee}$.

\begin{definition}
    In the above set up, we define $\mathcal{M}(A)$ to be the effective prismatic weak $F$-gauge over $S$ obtained from $H^2 (\mathrm{Fil}^\bullet_{\mathrm{Nyg}}R\Gamma_{\Prism}(BA))$ (see \cref{fl}).
\end{definition}{}

\begin{proposition}\label{vectorbundleabsch}
In the above set up, $\mathcal{M}(A)$ is a prismatic $F$-gauge on $S$ and is a vector bundle of rank $2\dim A$ when viewed as a quasicoherent sheaf on $\mathrm{Spf}(S)^\syn.$
\end{proposition}

\begin{proof}
  Follows in the same way as in the proof of \cref{dualo1} by replacing $\omega_G$ by $\Omega^1_A$ and $t_{G^\vee}$ by $t_{A^\vee}$. Alternatively, one can deduce it directly from the statement of  \cref{dualo1}, as we explain. Note that, $\Fil^k _{\mathrm{Nyg}}R\Gamma_{\Prism}(BA)) \simeq R\mathrm{Hom}_{(S)_\qsyn} (\mathbb{Z}[BA], \mathrm{Fil}^k_{\Nyg} \Prism_{(\cdot)} )$ (see the discussion before \cref{torsion} or the discussion before \cite[Lem.~3.2]{actaar}). Let $A[p^\infty]$ denote the $p$-divisible group associated to $A$. As $\mathrm{Fil}^k_{\Nyg} \Prism_{(\cdot)}$ is derived $p$-complete, by \cite[Lem.~3.2]{actaar}, we have 
 \begin{equation}\label{reducetopdic}
      \Fil^k _{\mathrm{Nyg}}R\Gamma_{\Prism}(BA))  \simeq \Fil^k _{\mathrm{Nyg}}R\Gamma_{\Prism}(BA[p^\infty])).
 \end{equation}{}
Therefore, applying \cref{dualo1} finishes the proof.
  \end{proof}{}

\begin{construction}[Dieudonn\'e $F$-gauge of abelian schemes]\label{starbc}
Let $S$ be a qrsp algebra. Let $A$ be an abelian scheme over $S.$ If $S \to S'$ is a map of qrsp algebras, and $A_{S'}$ is the abelian scheme over $S'$ obtained by base change, it follows again that the natural map $f^* \mathcal{M}(A) \to \mathcal{M}(A_{S'})$ is an isomorphism, where $f: \mathrm{Spf}(S')^{\mathrm{syn}} \to \Spf(S)^\syn$ denotes the natural map. As a consequence, if $T$ is a quasisyntomic ring and $A$ is an abelian scheme over $T$, by \cref{si}, one obtains a vector bundle on $\Spf(T)^\syn$ of rank $2\dim A$ which we denote as $\mathcal{M}(A)$.
\end{construction}{}

\begin{remark}\label{abbb}
Similar to \cref{l}, by the above definition, if $T$ is a quasisyntomic ring and $S$ is a qrsp algebra over $T,$ then the pullback of $\mathcal{M}(A)$ to $\mathrm{Spf}\,(S)^\syn$ is canonically isomorphic to $\mathcal{M}(A_S).$ By the first paragraph in the proof of \cref{equivalence}, it follows that the functor determined by
\begin{equation}\label{smallworld}
(T)_\qrsp \ni S \mapsto H^2 (\mathrm{Fil}^\bullet_{\mathrm{Nyg}}R\Gamma_{\Prism}(BA_S))\end{equation}is a sheaf of filtered modules over $\Fil^\bullet \Prism_{(\cdot)}.$ 
\end{remark}{}

\begin{remark}\label{cafe}
Let $A$ be an abelian scheme over a qrsp algebra $S.$ All the $\mathrm{Ext}$-groups below are computed in the quasisyntomic topos of $S.$ Similar to \cref{3van}, we have $H^1(\Fil^\bullet_{\Nyg}R\Gamma_{\Prism}(BA))=0,$ which implies $\mathrm{Hom}(A, \mathrm{Fil}_{\w{Nyg}}^k\Prism_{(\cdot)})=0.$ By \cref{useinprooflter} and \cref{reducetopdic}, one also has $$H^3(\Fil^\bullet_{\Nyg}R\Gamma_{\Prism}(BA))=0,$$ which implies, by the $E_2$ spectral sequence as in \cref{3van}, that we have
\begin{equation}\label{ext2vanab}
\mathrm{Ext}^2 (A, \mathrm{Fil}_{\w{Nyg}}^k\Prism_{(\cdot)})=0.
\end{equation}for all $k$.
Further, one also has 
\begin{equation}\label{ext1ab}
H^2 (\mathrm{Fil}^{\bullet}_{\w{Nyg}}R\Gamma_{\Prism} (BA)) \simeq \mathrm{Ext}^1(A, \mathrm{Fil}^{\bullet}_{\mathrm{Nyg}}\Prism_{(\cdot)}), 
\end{equation} as prismatic $F$-gauges.
\end{remark}{}

\begin{remark}\label{cafeforgm}
By \cref{inducedmapchernpower}, we have  $H^1(\Fil^\bullet_{\Nyg}R\Gamma_{\Prism}(B\mathbb{G}_m))=0$ and $H^3(\Fil^\bullet_{\Nyg}R\Gamma_{\Prism}(B\mathbb{G}_m))=0$. Similar to \cref{cafe}, we have $$H^2 (\mathrm{Fil}^{\bullet}_{\w{Nyg}}R\Gamma_{\Prism} (B\mathbb{G}_m)) \simeq \mathrm{Ext}^1(\mathbb{G}_m, \mathrm{Fil}^{\bullet}_{\mathrm{Nyg}}\Prism_{(\cdot)}),$$ which is further isomorphic to the prismatic $F$-gauge $\mathcal{O} \left \{-1\right \}$ by \cref{bgm}.
\end{remark}{}

Before proceeding further, we record a few observations.

\begin{construction}\label{chcago}
Let $T$ be a quasisyntomic ring. Let $\mathcal{F}$ be a quasisyntomic sheaf of abelian groups on $(T)_\qsyn$ such that $p: \mathcal{F} \to \mathcal{F}$ is surjective. Let $T_p(\mathcal{F}):= \varprojlim \mathcal{F}[p^n].$ Then we have a natural short exact sequence $0 \to T_p(\mathcal{F}) \to \varprojlim_p \mathcal{F} \to \mathcal{F} \to 0.$ For any derived $p$-complete sheaf of abelian groups $\mathcal{G}$, we have $R\mathcal{H}om (\varprojlim_p \mathcal{F}, \mathcal{G})=0$. To see the latter, note that multiplication by $p$ is an isomorphism on $\varprojlim_p \mathcal{F}$; therefore, $R\mathcal{H}om (\varprojlim_p \mathcal{F}, \mathcal{G}) \simeq  \varprojlim_p R\mathcal{H}om (\varprojlim_p \mathcal{F}, \mathcal{G}) \simeq R\mathcal{H}om (\varprojlim_p \mathcal{F}, \varprojlim_p \mathcal{G})$. But since $\mathcal{G}$ is derived $p$-complete, we have $\varprojlim_p \mathcal{G} = 0.$ This gives the desired vanishing, which produces a natural isomorphism
\begin{equation}\label{fight}
    R\mathcal{H}om(T_p(\mathcal{F}), \mathcal{G}) \simeq R\mathcal{H}om(\mathcal{F}, \mathcal{G}[1]).
\end{equation}As a consequence of \cref{fight}, we obtain the following isomorphisms, which will be used in the proof of \cref{abdual}. 
\begin{enumerate}
    \item $\mathrm{Hom}(T_p(A), \mathrm{Fil}^{\bullet}_{\mathrm{Nyg}}\Prism_{(\cdot)}) \simeq \mathrm{Ext}^1(A, \mathrm{Fil}^{\bullet}_{\mathrm{Nyg}}\Prism_{(\cdot)})$ for any abelian scheme $A$ over $T$.

    \item $\mathrm{Hom}(\mathbb{Z}_p(1), \mathrm{Fil}^{\bullet}_{\mathrm{Nyg}}\Prism_{(\cdot)}) \simeq \mathrm{Ext}^1(\mathbb{G}_m, \mathrm{Fil}^{\bullet}_{\mathrm{Nyg}}\Prism_{(\cdot)}) \simeq H^2 (\mathrm{Fil}^{\bullet}_{\w{Nyg}}R\Gamma_{\Prism} (B\mathbb{G}_m)).$

    \item $\mathcal{H}om_{(T)_\qsyn}(T_p(A), \mathbb{Z}_p(1)) \simeq \mathcal{E}xt^1_{(T)_\qsyn}(T_p(A), \mathbb{Z}_p(1)) \simeq T_p(A^\vee)$, for an abelian scheme $A$ over $T$, where the latter isomorphism is from \cref{tatemoduleabelianscheme}. 
\end{enumerate}{}
 
\end{construction}{}

\begin{remark}[De Rham realizations]
Let $A$ be an abelian scheme over a quasisyntomic ring $T$ of characteristic $p$. Note that $(T)_\qrsp \ni S \mapsto R\Gamma_{\mathrm{dR}}(S) $ defines a discrete sheaf of rings denoted by $\mathcal{O}^{\mathrm{dR}}$. Similar to the proof of \cref{vectorbundleabsch}, one has $R\Gamma_{\mathrm{dR}}(BA) \simeq R\Gamma_{\mathrm{dR}}(B A[p^\infty]).$ By the same argument in the proof of \cite[Prop.~3.28]{Mon1}, $H^*_{\mathrm{dR}}(BA)$ is a symmetric algebra in degree $2$ (also see Prop.~4.40 loc. cit.). In particular, $H^1_{\mathrm{dR}}(BA) \simeq H^3 _{\mathrm{dR}}(BA)=0.$ By the same strategy as in \cref{cafe}, we obtain $$\mathrm{Ext}^2 (A, \mathcal{O}^{\mathrm{dR}})=0\,\,\,  \text{and}\,\,\, H^2_{\mathrm{dR}}(BA) \simeq \mathrm{Ext}^1 (A, \mathcal{O}^{\mathrm{dR}}).$$The natural map $\mathbb{Z}[A] \to A$ of abelian sheaves induces a natural map 
$$\mathrm{Ext}^1 (A, \mathcal{O}^{\mathrm{dR}}) \to \mathrm{Ext}^1 (\mathbb{Z}[A], \mathcal{O}^{\mathrm{dR}}) \simeq H^1_{\mathrm{dR}}(A),$$ which is known to be an isomorphism, by using the technique of Breen--Deligne resolution (see \cite[Thm.~2.5.6]{bbm} or \cite[Thm.~4.60]{alb}). Combining with \cref{chcago}, we obtain a natural isomorphism
\begin{equation}\label{3-1}\mathrm{Hom} (T_p(A), \mathcal{O}^{\mathrm{dR}}) \simeq H^1_{\mathrm{dR}}(A)
   \end{equation} 
that will be used in the proof of \cref{abdual}. The above maps also give a natural isomorphism \begin{equation}\label{2-1}
    H^2_{\mathrm{dR}}(BA) \simeq H^1_{\mathrm{dR}}(A).
\end{equation}{}
\end{remark}{}

\begin{proposition}[Duality compatibility for abelian schemes]\label{abdual} Let $A$ be an abelian scheme over a quasisyntomic ring $T.$ Let $A^\vee$ denote the dual abelian scheme. Then we have a natural isomorphism
$$\mathcal{M}(A)^* \left \{-1 \right \} \simeq \mathcal{M}(A^\vee)$$ of prismatic $F$-gauges over $T.$
\end{proposition}{}
\begin{proof}
 By quasisyntomic descent, it is enough to construct a natural isomorphism in the case when $T$ is qrsp. Note that we have a canonical isomorphism (\cref{chcago} (3))
 \begin{equation}\label{canonicalidentificationoftatemodules}
       \mathcal{H}om_{(T)_\qsyn}(T_p(A), \mathbb{Z}_p(1)) \simeq T_p(A^\vee).
 \end{equation}{}
By functoriality, the above map constructs a natural map
 \begin{equation}
  T_p(A^\vee) \to \mathcal{H}om_{\syn} \left(\Hom (\mathbb{Z}_p(1), \Fil^\bullet_{\Nyg}\Prism_{(\cdot)}), \Hom (T_p(A), \Fil^\bullet_{\Nyg}\Prism_{(\cdot)}) \right)
  \end{equation} in the category of abelian sheaves on $(T)_\qrsp.$ This produces a natural map
  \begin{equation}\label{importantmapshouldhaveusedbefore}
  T_p(A^\vee) \to \mathcal{H}om_{\syn} (\mathcal{O}\left\{-1\right\}, \mathcal{M}(A)) \simeq \mathcal{H}om_{\syn} (\mathcal{M}(A)^*\left\{-1\right\}, \Fil^\bullet \Prism_{(\cdot)}).    
  \end{equation}
Now, there is a canonical ``evaluation map"
 \begin{equation}\label{veryimportantmap}
   \mathcal{M}(A)^* \left \{-1\right\} \to  \mathrm{Hom}_{(T)_\qrsp} (\mathcal{H}om_{\syn} (\mathcal{M}(A)^*\left\{-1\right\}, \Fil^\bullet \Prism_{(\cdot)}), \Fil^\bullet \Prism_{(\cdot)}).  
  \end{equation}
  Composing with \cref{importantmapshouldhaveusedbefore}, we obtain a map
  \begin{equation}\label{abelianschemedualitymap}
    \mathcal{M}(A)^* \left \{-1\right\} \to \Hom_{(T)_\qrsp}(T_p(A^\vee), \Fil^\bullet \Prism_{(\cdot)}) \simeq   \mathcal{M}(A^\vee). 
  \end{equation}
We will argue that the above map is an isomorphism. Since both sides are prismatic $F$-gauges with Hodge--Tate weights in $\left \{0,1 \right \}$ and the map we have constructed is compatible with the natural Frobenius maps, by \cref{comppp} it suffices to show that the induced natural map
\begin{equation}
H^2_{\Prism} (BA)^* \left \{-1\right \} \to H^2_\Prism (BA^\vee) 
\end{equation} of locally free $\Prism_T$-modules is an isomorphism. Since $\Prism_T \to T$ is henselian, it would suffice to show that the induced natural map
$$H^2_\Prism(BA)^* \left \{-1 \right\} \otimes_{\Prism_T} T \to H^2_\Prism (BA^\vee )\otimes_{\Prism_T} T$$ of locally free $T$-modules is an isomorphism. It suffices to check this after base changing to an arbitrary map $T \to k$, where $k$ is a perfect field of characteristic $p$. Under \cref{3-1} and \cref{2-1}, we are thus reduced to checking that the induced map
\begin{equation}\label{firstmap1}
 H^1_{\mathrm{dR}}(A_k)^* \to H^1_{\mathrm{dR}} (A^\vee_k)   \end{equation}is an isomorphism of $k$-vector spaces. We will give a geometric reconstruction of this map, under which this statement will be classical. Let $\mathcal{P} $ denote the tautological line bundle on $A_k \times A_k^\vee$. Let ${c}_1 (\mathcal{P}) \in H^2_{\mathrm{dR}}(A_k\times A_k^\vee)$ denote its Chern class. Let $c_1(\mathcal{P})^{1,1} \in H^1_{\mathrm{dR}}(A_k) \otimes_k H^1_{\mathrm{dR}}(A^\vee_k)$ denote the projection of $c_1(\mathcal{P})$ under the Kunneth map $$\pi_{1,1}: H^2_{\mathrm{dR}}(A_k\times A_k^\vee) \to H^1_{\mathrm{dR}}(A_k) \otimes_k H^1_{\mathrm{dR}}(A^\vee_k). $$ Then $c_1(\mathcal{P})^{1,1}$ induces a canonical map 
 \begin{equation}\label{chernclassmapduality}
 H^1_{\mathrm{dR}}(A_k)^* \to H^1_{\mathrm{dR}} (A^\vee_k).
 \end{equation}By \cref{lemmaforcomp}, \cref{firstmap1} and \cref{chernclassmapduality} are the same map, and they are isomorphisms by \cite[\S~5.1]{bbm}. This finishes the proof.
\end{proof}

\begin{lemma}\label{lemmaforcomp}
 The maps \cref{firstmap1} and \cref{chernclassmapduality} constructed above are the same.   
\end{lemma}{}
\begin{proof}
 This proof uses the ideas from \cite[Lem.~5.2.5]{bbm}. To simplify notations, we will use $A$ to denote $A_k$. Note the following commutative diagram

 \begin{center}
\begin{tikzcd}
{\mathrm{Ext}^1(A \otimes_{\mathbb{Z}}^L A^\vee, \mathbb{G}_m)} \arrow[d, "d_1"] \arrow[rr, "t_1"]                 &  & {H^1(A \times A^\vee, \mathbb{G}_m)} \arrow[d, "c_1"]        \\
{\mathrm{Ext}^1(A \otimes_{\mathbb{Z}}^L A^\vee, \mathcal{O}^{\mathrm{dR}}[1])} \arrow[rr, "t_2"] \arrow[d, "d_2"] &  & H^2_{\mathrm{dR}}(A \times A^\vee) \arrow[d, "{\pi_{1,1}}"]  \\
{\mathrm{Ext}^1(A^\vee, R\mathcal{H}om(A, \mathcal{O}^{\mathrm{dR}}[1]))} \arrow[rr, "t_3"]                        &  & H^1_{\mathrm{dR}}(A) \otimes_k H^1_{\mathrm{dR}}(A^\vee)
\end{tikzcd}
 \end{center}
 All the Ext groups above are taken in $(k)_\qrsp$. The maps $t_1$ and $t_2$ are induced by the composite canonical map $\mathbb{Z}[A \times A^\vee] \simeq \mathbb{Z}[A] \otimes^L_{\mathbb{Z}}\mathbb{Z}[A^\vee]  \to A \otimes^L_{\mathbb{Z}}A^\vee$ in $(k)_\qrsp$; see the proof of \cite[Lem.~5.2.5]{bbm}. The maps $ d_2$ and $t_3$ are the natural maps as in the commutative diagram in loc. cit. Note that there is a canonical map $\mathrm{ch}_{\mathrm{dR}}:\mathbb{G}_m \to \mathcal{O}^{\mathrm{dR}}[1]$ induced via left Kan extending the composite map $d \log: \mathbb{G}_m(S) \to \Fil^1_{\mathrm{Hodge}} R\Gamma_{\mathrm{dR}}(S)[1] \simeq R\Gamma_{\mathrm{dR}}(S)[1]$ from smooth $k$-algebras $S$ to $(k)_\qsyn$, and then restricting to $(k)_\qrsp$. The map $d_1$ above is induced from $\mathrm{ch}_{\mathrm{dR}}.$
 
 Since $\mathcal{P}$ arises via a canonical isomorphism $D_{\mathcal{P}}: A \simeq \tau_{\ge 0}\mathcal{R}Hom_{(k)_\qsyn}(A^\vee, \mathbb{G}_m[1])$ (cf.~\cite[Lem.~5.2.3]{bbm}), there is a canonical element $\mathcal{P'}$ such that $t_1(\mathcal{P}')= \mathcal{P}$.
  The above diagram gives $c_1(\mathcal{P})^{1,1} = t_3 (d_2 (d_1(\mathcal{P}'))).$

Now there is a canonical element in $\mathrm{Hom}_{(k)_\qrsp}(\mathbb{Z}_p(1), \Prism_{(\cdot)}) \simeq H^2_{\Prism}(B\mathbb{G}_m)$ arising from the prismatic Chern class of the tautological line bundle on $B\mathbb{G}_m$, which corresponds to a map $\mathbb{Z}_p(1) \to \Prism_{(\cdot)}$. By \cite[Thm.~7.6.2]{BhaLu}, the composition $\mathbb{G}_m \to \mathbb{Z}_p(1)[1] \to \Prism_{(\cdot)}[1] \to \mathcal{O}^{\mathrm{dR}}[1]$ agrees with $\mathrm{ch}_{\mathrm{dR}}$. Further, there is a map
\begin{equation}\label{bilinearmap}
  T_p: {\mathrm{Ext}^1(A \otimes_{\mathbb{Z}}^L A^\vee, \mathbb{G}_m)} \to \mathrm{Hom} (T_p(A) \otimes_\mathbb{Z} T_p(A^\vee), \mathbb{Z}_p(1)), 
\end{equation} as we will construct. We have a canonical identification 
$\mathrm{Ext}^1 (A\otimes^L_{\mathbb{Z}}A^\vee, \mathbb{G}_m) \simeq \mathrm{Hom}(A, A)$ induced by $D_\mathcal{P}$. By functoriality, we have a natural map $\Hom (A, A) \to \Hom(T_p(A), T_p(A)).$ Further, \cref{canonicalidentificationoftatemodules} gives a natural identification $\Hom(T_p(A), T_p(A)) \simeq \mathrm{Hom} (T_p(A) \otimes_\mathbb{Z} T_p(A^\vee), \mathbb{Z}_p(1))$. Composing these maps, we obtain the desired map \cref{bilinearmap}. Now we have the following commutative diagram 

\begin{center}
\begin{tikzcd}
{\mathrm{Ext}^1(A \otimes_{\mathbb{Z}}^L A^\vee, \mathbb{G}_m)} \arrow[d, "d_1"] \arrow[rr, "T_p"]                    &  & {\mathrm{Hom} (T_p(A) \otimes_\mathbb{Z} T_p(A^\vee), \mathbb{Z}_p(1))} \arrow[d, "g_1"]      \\
{\mathrm{Ext}^1(A \otimes_{\mathbb{Z}}^L A^\vee, \mathcal{O}^{\mathrm{dR}}[1])} \arrow[rr, "\simeq"] \arrow[d, "d_2"] &  & {\mathrm{Hom}(T_p(A), \mathrm{Hom}(T_p(A^\vee), \mathcal{O}^{\mathrm{dR}}))} \arrow[d, "g_2"] \\
{\mathrm{Ext}^1(A^\vee, R\mathcal{H}om(A, \mathcal{O}^{\mathrm{dR}}[1]))} \arrow[rr, "t_3"]                        &  & H^1_{\mathrm{dR}}(A) \otimes_k H^1_{\mathrm{dR}}(A)                                 
\end{tikzcd}

\end{center} In the above, the two lower horizontal maps are induced by \cref{fight}. The map $g_1$ is induced by the map $\mathbb{Z}_p(1) \to \mathcal{O}^{\mathrm{dR}}$ mentioned above, and $g_2$ is induced by \cref{3-1}.

By construction, $g_2(g_1(T_p(\mathcal{P}'))) \in H^1_{\mathrm{dR}}(A_k) \otimes_k H^1_{\mathrm{dR}}(A^\vee_k)$ defines \cref{firstmap1}. By commutativity of the diagram, it is the same as $ t_3 (d_2 (d_1(\mathcal{P}')))= c_1(\mathcal{P})^{1,1},$ as desired.
\end{proof}


\subsection{Prismatic Dieudonn\'e $F$-gauges associated to finite flat group schemes}
In this subsection, we will construct certain prismatic Dieudonn\'e $F$-gauges associated to finite flat group schemes, and prove our main theorems (\cref{mainthm22}, \cref{mainthm2}). We will also prove compatibility with Cartier duality (\cref{dualoo}), expression for cohomology with coefficients in finite flat group schemes in terms of prismatic Dieudonn\'e $F$-gauges (\cref{earth}), and use it to reconstruct Galois representations associated to group schemes (\cref{finalprop}). 

\begin{construction}\label{cafe4}
Let $T$ be a quasisyntomic ring and let $\mathrm{FFG}(T)$ denote the category of finite locally free commutative group schemes over $S$ of $p$-power rank. Let $G \in \mathrm{FFG}(T).$ The functor determined by 
\begin{equation}\label{sh}
 (T)_{\mathrm{qrsp}} \ni S \mapsto (\tau_{[-2,-3]} \mathrm{Fil}^\bullet_{\mathrm{Nyg}} R\Gamma_{\Prism}(B^2G_S))[3]   
\end{equation}defines a presheaf of filtered modules over the filtered ring $\mathrm{Fil}^\bullet \Prism_{(\cdot)}.$ We let $\mathcal{M}(G)$ denote the sheafification of this functor for the quasisyntomic topology on $(T)_{\mathrm{qrsp}}$ as a sheaf of derived $(p, \mathscr I)$-complete filtered module over the filtered ring $\Fil^\bullet \Prism_{(\cdot)}.$ Let $\mathcal{M}(G)^\mathrm{u}$ denote the underlying object of the filtered object $\mathcal{M}(G).$ There is a natural map
$\varphi: \mathcal{M}(G) \to \mathscr I^\bullet \otimes_{\Prism_{(\cdot)}} \mathcal{M}(G)^{\w{u}}$ of filtered $\Fil^\bullet \Prism_{(\cdot)}$-modules induced from the natural Frobenius map defined at the level of presheaves in \cref{sh}. We will prove in \cref{fffff} that $(\mathcal{M}(G), \varphi)$ satisfies conditions $(1)$ and $(2)$ from \cref{fgaugeff} and therefore, defines a prismatic $F$-gauge, which we would again denote as $\mathcal{M}(G).$
\end{construction}
Before that, we will need the following preparation.

\begin{proposition}\label{cafe3}
 Let $S$ be a quasiregular semiperfectoid algebra and let $G \in \mathrm{FFG}(S).$ Then we have a natural isomorphism of prismatic weak $F$-gauges over $S$ $$\tau_{\ge 0}R\mathrm{Hom}_{(S)_{\qrsp}}(G, \Fil^\bullet \Prism_{(\cdot)}[1]) \simeq (\tau_{[-2,-3]}\mathrm{Fil}^\bullet_{\Nyg} R\Gamma_{\Prism} (B^2 G))[3].$$  
\end{proposition}{}
\begin{proof}
Let $P$ be an ordinary abelian group. We have a natural map $\mathbb{Z} \to \mathbb{Z}[B^2P]$, arising from the map $(0) \to B^2 P$. The cofiber of $\mathbb{Z} \to \mathbb{Z}[B^2P]$ will be denoted as $\mathbb{Z}^{\w{red}}[B^2P].$ There is a natural map $\mathbb{Z}^{\w{red}}[B^2P] \to P[2],$ which is the $2$-truncation. This induces a natural composite map $\mathbb Z[B^2 P] \to \mathbb{Z}^{\w{red}}[B^2P] \to P[2].$ Applying this to $G$ regarded as an abelian group object of the topos associated to $(S)_\qrsp,$ we obtain a natural map $$R\mathrm{Hom}_{(S)_\qrsp}(G, \Fil^\bullet \Prism_{(\cdot)}[1]) \to R\mathrm{Hom}_{(S)_\qrsp}(\mathbb Z[B^2G], \Fil^\bullet \Prism_{(\cdot)})[3].$$ This induces a natural map 
$$\tau_{\ge 0}R\mathrm{Hom}_{(S)_{\qrsp}}(G, \Fil^\bullet \Prism_{(\cdot)}[1]) \to (\tau_{[-2,-3]}\mathrm{Fil}^\bullet_{\Nyg} R\Gamma_{\Prism} (B^2 G))[3].$$
To show that it is an isomorphism, it suffices to show that the induced maps
$$\mathrm{Hom}_{(S)_{\qrsp}} (G, \Fil^\bullet \Prism_{(\cdot)}) \to H^2 (\mathrm{Fil}^\bullet_{\Nyg} R\Gamma_{\Prism} (B^2 G))) $$ and $$\mathrm{Ext}^1_{(S)_{\qrsp}} (G, \Fil^\bullet \Prism_{(\cdot)}) \to H^3 (\mathrm{Fil}^\bullet_{\Nyg} R\Gamma_{\Prism} (B^2 G)))$$ are isomorphisms. By construction, the canonical map $\mathbb{Z}[B^2 G] \to G[2]$ induces a natural isomorphism on $H^{-2}$. Therefore, the above two maps are the same as the ones induced by the $E_2$-spectral sequence 
$$E_2^{i,j}=\mathrm{Ext}^i_{(S)_\qrsp}(H^{-j} (\mathbb{Z}[B^2G]), \Fil^\bullet \Prism_{(\cdot)} ) \implies H^{i+j}(\mathrm{Fil}^\bullet_{\Nyg} R\Gamma_{\Prism} (B^2 G))),$$ and they are isomorphisms by the same arguments as in the proof of \cref{finiteflatrev}.
\end{proof}{}

\begin{remark}\label{ecc22}
Let $G$ be a $p$-divisible group over a quasisyntomic ring $T.$ Let $S \in (T)_\qrsp.$ An argument using the $E_2$-spectral sequence similar to above shows that $$H^2 (\mathrm{Fil}^\bullet_{\Nyg} R\Gamma_{\Prism} (B^2 G_S))) \simeq \mathrm{Hom}_{(S)_{\qrsp}} (G, \Fil^\bullet \Prism_{(\cdot)})$$ and $$H^3 (\mathrm{Fil}^\bullet_{\Nyg} R\Gamma_{\Prism} (B^2 G_S))) \simeq \mathrm{Ext}^1_{(S)_{\qrsp}} (G, \Fil^\bullet \Prism_{(\cdot)} ).$$ By \cref{3van}, it follows that there is an isomorphism $$(\tau_{[-2,-3]}\mathrm{Fil}^\bullet_{\Nyg} R\Gamma_{\Prism} (B^2 G_S))[3] \simeq \mathcal{M}(G_S)$$ of prismatic $F$-gauges over $S.$ Therefore, if we applied \cref{cafe4} to the case of $p$-divisible groups, the resulting construction would be equivalent to \cref{d}.
\end{remark}{}

\begin{remark}\label{cafe1}
For an abelian scheme $A$ over a quasisyntomic ring $T$ and $S \in (T)_{\qrsp},$ using \cref{cafe} and exactly following the steps in \cref{ecc22}, we obtain that 
$$\tau_{\ge 0}R\mathrm{Hom}_{(S)_{\qrsp}}(A_S, \Fil^\bullet \Prism_{(\cdot)}[1]) \simeq (\tau_{[-2,-3]}\mathrm{Fil}^\bullet_{\Nyg} R\Gamma_{\Prism} (B^2 A_S))[3] \simeq \mathcal{M}(A_S)$$ as prismatic $F$-gauges over $S.$ Therefore, if we applied \cref{cafe4} to the case of abelian schemes, the resulting construction would be equivalent to \cref{d}.
\end{remark}{}

\begin{lemma}\label{atias}
 Let $T$ be a quasisyntomic ring and $G \in \mathrm{FFG}(T).$ Assume that there is a short exact sequence $$0 \to G \to A' \to A \to 0,$$ where $A, A'$ are abelian schemes over $T.$ Then $\mathcal{M}(G)$ is a prismatic $F$-gauge over $T$ and we have a fiber sequence
$$\mathcal{M}(A) \to \mathcal{M}(A') \to \mathcal{M}(G)$$ of prismatic $F$-gauges over $T$.
\end{lemma}{}

\begin{proof}Let $S \in (T)_{\qrsp}$. We have a fiber sequence of filtered $\Fil^\bullet_{\Nyg}\Prism_S$-modules 
\begin{equation*}
\begin{split}
  R\mathrm{Hom}_{(S)_{\qrsp}}(A, \Fil^\bullet \Prism_{(\cdot)}[1]) \to R\mathrm{Hom}_{(S)_{\qrsp}}(A', \Fil^\bullet \Prism_{(\cdot)}[1]) \\ \to R\mathrm{Hom}_{(S)_{\qrsp}}(G, \Fil^\bullet \Prism_{(\cdot)}[1]).  
\end{split}{}    
\end{equation*}
By the vanishing in \cref{ext2vanab}, we have a fiber sequence 
\begin{equation}\label{cafe2}\begin{split}
\tau_{\ge 0}R\mathrm{Hom}_{(S)_{\qrsp}}(A, \Fil^\bullet \Prism_{(\cdot)}[1]) \to \tau_{\ge 0}R\mathrm{Hom}_{(S)_{\qrsp}}(A', \Fil^\bullet \Prism_{(\cdot)}[1]) \\ \to \tau_{\ge 0}R\mathrm{Hom}_{(S)_{\qrsp}}(G, \Fil^\bullet \Prism_{(\cdot)}[1]). \end{split} 
\end{equation}
By \cref{starbc} and \cref{cafe1}, the presheaf of filtered $\Fil^\bullet \Prism_{(\cdot)}$-modules on $(T)_\qrsp$ (equipped with the natural Frobenius) determined by the association
$$(T)_\qrsp \ni S \mapsto \tau_{\ge 0}R\mathrm{Hom}_{(S)_{\qrsp}}(A_S, \Fil^\bullet \Prism_{(\cdot)}[1])$$ is a prismatic $F$-gauge (similarly for $A'$), it follows from \cref{cafe2} that the presheaf of filtered $\Fil^\bullet \Prism_{(\cdot)}$-modules on $(T)_\qrsp$ (equipped with the natural Frobenius) determined by the association
$$(T)_\qrsp \ni S \mapsto \tau_{\ge 0}R\mathrm{Hom}_{(S)_{\qrsp}}(G_S, \Fil^\bullet \Prism_{(\cdot)}[1])$$ is a prismatic $F$-gauge. By \cref{cafe3}, we conclude that $\mathcal{M}(G)$ is a prismatic $F$-gauge and one also has a fiber sequence as desired.
\end{proof}{}

\begin{construction}\label{verys1}
 Let $T$ be a quasisyntomic ring and let $G \in \mathrm{FFG}(T).$ Let us consider the full subcategory of $(T)_\qrsp$ spanned by $S \in (T)_\qrsp$ with the property that there exists a short exact sequence $$0 \to G_S \to A' \to A \to 0,$$ where $A, A'$ are abelian schemes over $S.$ We denote this full subcategory by $(T)^G_\qrsp.$ If $S \in (T)_\qrsp$ and $S \to S_1$ is a quasisyntomic cover with $S_1 \in (T)^G_\qrsp$ and $S \to S_2$ is any map in $(T)_\qrsp$, then the ($p$-completed) pushout is again in $(T)^G_\qrsp$ by the proof of \cite[Lem.~4.27]{BMS2}.

In particular, equipped with the quasisyntomic topology, we see that $(T)^G_\qrsp$ forms a site. Using a theorem of Raynaud \cite[Thm.~3.1.1]{bbm} combined with \cite[Lem.~4.28]{BMS2}, it follows that for any $S \in (T)_{\qrsp},$ there is a quasisyntomic cover $S \to S',$ such that $S' \in (T)^G_\qrsp.$ Also, by \cite[Lem.~4.30]{BMS2}, we have that if $S \to S'$ is a quasinsyntomic cover with $S \in (T)_\qrsp$ and $S' \in (T)^G_\qrsp,$ then all terms of the Cech nerve $S'^\bullet \in (T)^G_\qrsp.$ In particular, $(T)^G_\qrsp$ forms a basis for the site $(T)_\qrsp.$
\end{construction}

\begin{proposition}\label{fffff}Let $T$ be a quasisyntomic ring and $G \in \mathrm{FFG}(T).$ Then $(\mathcal{M}(G), \varphi)$ from \cref{cafe4} defines a prismatic $F$-gauge over $T.$
\end{proposition}{}

\begin{proof}
Note that the association 
\begin{equation}\label{atias1}(T)^G_\qrsp \ni S \mapsto (\tau_{[-2,-3]} \mathrm{Fil}^\bullet_{\mathrm{Nyg}} R\Gamma_{\Prism}(B^2G_S))[3]
\end{equation}{}determines a sheaf of $(p, \mathcal{I})$-complete filtered modules $\Fil^\bullet \mathscr{N}$ over the filtered ring $\Fil^* \Prism_{(\cdot)}$ equipped with a $\Fil^\bullet \Prism_{(\cdot)}$-linear Frobenius map 
$$\varphi:  \mathrm{Fil}^\bullet \mathscr N \to \mathscr{I}^\bullet \mathscr N$$ that satisfies natural analogues of $(1)$ and $(2)$ from \cref{fgaugeff}. This follows from the fiber sequence \cref{cafe2}, \cref{abbb}, the base change property for Dieudonn\'e $F$-gauges of abelian schemes in \cref{cafe2}, and \cref{atias}. In the above, we abuse notation slightly and use $\mathscr{I}$ and $\mathrm{Fil}^\bullet \Prism_{(\cdot)}$ to denote their restriction to $(T)^G_\qrsp.$

 By \cref{cafe4}, $\mathcal{M}(G)$ (along with the Frobenius) is obtained by sheafifiying the presheaf on $(T)_\qrsp$ described \cref{sh}. By the previous paragraph, it follows that the latter presheaf restricted to $(T)^G_\qrsp$ is a sheaf, and for $S \in (T)^G_\qrsp,$ we have $\mathcal{M}(G) (S) \simeq (\tau_{[-2,-3]} \mathrm{Fil}^\bullet_{\mathrm{Nyg}} R\Gamma_{\Prism}(B^2G_S))[3]$ is a prismatic $F$-gauge over $S.$

Further, as in \cref{presen}, by descent, we have
\begin{equation}\label{need}
  F\text{-Gauge}_{\Prism}(T) \simeq \lim_{S \in (T)^G_{\qrsp}} F\text{-Gauge}_{\Prism}(S).   
\end{equation}{}
Therefore, by the previous paragraph, the association in \cref{atias1} determines a prismatic $F$-gauge, which we denote as $\mathcal{M}'(G)$ and view it as in \cref{fgaugeff}. Note that by construction, there is a natural map $\mathcal{M}(G) \to \mathcal{M}'(G)$ of sheaf of filtered modules over $\Fil^\bullet \Prism_{(\cdot)}$ on $(T)_\qrsp,$ compatibly with the Frobenius. Since this map is an isomorphism when restricted to $(T)^G_\qrsp,$ it must be an isomorphism. Since $\mathcal{M}'(G)$ is a prismatic $F$-gauge, $\mathcal{M}(G)$ satisfies conditions $(1)$ and $(2)$ from \cref{fgaugeff} and therefore, is a prismatic $F$-gauge.
\end{proof}{}

\begin{definition}[Dieudonn\'e $F$-gauge of a finite locally free group scheme] Let $T$ be quasisyntomic ring and let $G \in \mathrm{FFG}(T).$ We call the prismatic $F$-gauge over $T$ given by  $\mathcal{M}(G)$ from \cref{fffff} the (prismatic) Dieudonn\'e $F$-gauge of $G.$
\end{definition}{}

\begin{remark}
We can view the prismatic $F$-gauge $\mathcal{M}(G)$ as an object of $D_{\mathrm{qc}}(\Spf(T)^\syn)$ or in the sense of \cref{fgaugeff}. If $T' \to T$ is a map of quasisyntomic rings inducing a map $f: \mathrm{Spf}(T')^\syn \to \Spf (T)^\syn,$ then by construction we have $f^* \mathcal{M}(G) \simeq \mathcal{M}(G_{T'}).$
\end{remark}{}

\begin{remark}\label{ant}
 Let $T$ be a quasisyntomic ring and $G \in \mathrm{FFG}(T).$ As a consequence of the proof of \cref{fffff}, we see that the prismatic $F$-gauge $\mathcal{M}(G),$ when viewed as a sheaf of filtered modules (equipped with a Frobenius) on $(T)^G_\qrsp,$ is given by the formula \cref{atias1}. Using \cref{cafe3}, one may rewrite the latter as 
\begin{equation}\label{atias111}(T)^G_\qrsp \ni S \mapsto \tau_{\ge 0}R\mathrm{Hom}_{(S)_{\qrsp}}(G, \Fil^\bullet \Prism_{(\cdot)}[1]).
\end{equation}
\end{remark}{}

\begin{remark}\label{stella1}Although it is out of the scope of our current paper, we give sketch of another description of the prismatic Dieudonn\'e $F$-gauge functor. Let $S$ be a qrsp ring and $G \in \mathrm{FFG}(S).$ Then $\mathcal{M}(G)$ as in \cref{cafe4}, can also be written as $(\tau_{[-2,-3]}Rv_* \mathcal{O})[3]$, where $v: B^2 G^\syn \to S^\syn$ denotes the natural map. Similar to the proof of \cref{cafe3}, we obtain a natural map
\begin{equation}\label{stella}
    R\mathcal{H}om_{S^\syn} (G^{\w{syn}}[2], \mathbb G_a) \to R\mathcal{H}om_{S^\syn} (\mathbb{Z}[B^2G^{\w{syn}}], \mathbb G_a)\simeq Rv_* \mathcal{O}. 
\end{equation}
This induces natural maps
$\mathcal{E}xt^1(G^\mathrm{syn}, \mathbb{G}_a) \to \mathcal{E}xt^{3}(\mathbb{Z}[B^2 G^{\w{syn}}], \mathbb{G}_a)$ and $\mathcal{H}om(G^\mathrm{syn}, \mathbb{G}_a) \to \mathcal{E}xt^1(G^\mathrm{syn}, \mathbb{G}_a)$, which are isomorphisms by a $E_2$-spectral sequence argument similar to \cref{cafe3}. Note that this spectral sequence arises from a certain filtration on $\mathbb{Z}[B^2P]$ for an animated abelian group $P$. This is obtained as animation of the Postnikov filtration $\tau_{\le *} \mathbb{Z}[B^2 P^0]$ for discrete abelian groups $P^0$. As a consequence, $\gr^0 \mathbb{Z}[B^2P] \simeq \mathbb{Z}$, $\gr^1 \mathbb{Z}[B^2P] \simeq 0, \gr^2 \mathbb{Z}[B^2P] \simeq P$ and $\gr^3 \mathbb{Z}[B^2P] \simeq 0,$ as the same statements hold (functorially) for any discrete abelian group $P^0$ (in which case $\gr^i$ is the same as $i$-th homology). Therefore, \cref{stella} induces an isomorphism
\begin{equation}
\tau_{[-2,-3]} R\mathcal{H}om_{S^\syn} (G^{\w{syn}}[2], \mathbb G_a) \simeq \tau_{[-2,-3]}Rv_* \mathcal{O}.    
\end{equation}This implies that for a quasisyntomic ring $T$, and $G \in \mathrm{FFG}(T),$ we have 
\begin{equation}\label{stella3}
  \mathcal{M}(G) \simeq \tau_{\ge 0} R\mathcal{H}om_{S^\syn}(G^\syn, \mathbb{G}_a[1]).  
\end{equation}{}

\end{remark}{}

\begin{proposition}\label{ecc}
 Let $S$ be a quasisyntomic ring and $G \in \mathrm{BT}(S).$ Let $G_n := G[p^n].$ Then the fiber sequence $G_n \to G \xrightarrow{p^n} G$ induces a fiber sequence
\begin{equation}\label{cofiber}
\mathcal{M}(G) \xrightarrow{p^n} \mathcal{M}(G) \to \mathcal{M}(G_n)
\end{equation} of prismatic $F$-gauges over $S.$
\end{proposition}
\begin{proof}
In order to prove this, one may work locally on $S.$ In particular, we may work in the site $(S)^{G_n}_\qrsp.$ Using the description in \cref{ant}, the desired fiber sequence \cref{cofiber} follows from the vanishing \cref{ext2van} and \cref{ext1} (\textit{cf}. \cref{cafe2}).
\end{proof}{}

\begin{remark}\label{ecc2}
Let $S$ be a quasisyntomic ring and $G \in \mathrm{BT}(S).$ Then the natural map $\mathcal{M}(G) \to \varprojlim \mathcal{M}(G_n)$ is an isomorphism of prismatic $F$-gauges. This follows from \cref{cofiber} since $\mathcal{M}(G)$ is derived $p$-complete.\end{remark}{}

\begin{remark}\label{ecc57}
 Let $S$ be a quasisyntomic ring and $G, G' \in BT(S).$ Assume that we have a fiber sequence $H \to G' \to G,$ where $H \in \mathrm{FFG}(S).$ By the same argument as in \cref{ecc}, we have a fiber sequence $\mathcal{M}(G) \to \mathcal{M}(G') \to \mathcal{M}(H)$ or prismatic $F$-gauges over $S.$   
\end{remark}{}

We will now explain how one can recover cohomology with coefficients in $G \in \mathrm{FFG}(T)$ in terms of quasicoherent cohomology of the prismatic Dieudonn\'e $F$-gauge of the Cartier dual $G^\vee$.

\begin{proposition}[Cohomology with coefficients in group schemes]\label{earth} Let $T$ be a quasisyntomic ring and let $G \in \mathrm{FFG}(T).$ There is a natural isomorphism
$$R\Gamma_{\mathrm{qsyn}}(T, G) \simeq R\Gamma (\mathrm{Spf}(T)^\syn, \mathcal{M}(G^\vee) \left \{1 \right \}).$$
\end{proposition}{}

\begin{proof}
Note that both sides satisfy quasisyntomic descent, viewed as a functor on quasisyntomic rings. Therefore, it suffices to produce a natural isomorphism when $T$ is qrsp. By the construction of $\mathcal{M}(G)$ via sheafification in \cref{cafe4} and \cref{sick5} for $n=1,$ we have a natural map (see \cref{finiteflat} for the notation of the source) of quasisyntomic sheaves on $(T)_{\qrsp}$
as below.
\begin{equation}\label{vert1}
 \begin{matrix}
\mathrm{Fib}\left((\tau_{[-3,-2]}\mathrm{Fil}^1_{\mathrm{Nyg}}R\Gamma_{\Prism}(B^2 G^\vee_{(\cdot)})\left \{1\right\}^\sharp)[3] \xrightarrow{\varphi_1 - \mathrm{can}} (\tau_{[-3,-2]} R\Gamma_{\Prism}(B^2 G^\vee_{(\cdot)})\left \{1\right\}^\sharp)[3]\right) \\ \downarrow\\

R\Gamma (\mathrm{Spf}(T)^\syn, \mathcal{M}(G^\vee) \left \{1 \right \})
\end{matrix}   
\end{equation}{}
By \cref{finiteflat}, this constructs a natural map of quasisyntomic sheaves on $(T)_\qrsp$ 
\begin{equation}\label{verysick}
 R\Gamma_{\mathrm{qsyn}}(S, G) \to R\Gamma (\mathrm{Spf}(S)^\syn, \mathcal{M}(G^\vee) \left \{1 \right \}),   
\end{equation}where $S \in (T)_\qrsp.$ To prove that \cref{verysick} is an isomorphism, by quasisyntomic descent, it is enough to prove it for $S \in (T)^G_\qrsp$ (\cref{verys1}). But then by the sheafiness of the functor in \cref{atias1}, it follows that the downward vertical map \cref{vert1} is an isomorphism. Therefore, by \cref{finiteflat}, the map \cref{verysick} is an isomorphism. This finishes the proof.
\end{proof}{}

\begin{proposition}[Cohomology with coefficients in Tate modules]\label{jac}
Let $S$ be a quasisyntomic ring and let $G \in \mathrm{BT}(S).$ There is a natural isomorphism
$$R\Gamma_{\mathrm{qsyn}}(S, T_p(G)) \simeq R\Gamma (\mathrm{Spf}(S)^\syn, \mathcal{M}(G^\vee) \left \{1 \right \}).$$
\end{proposition}{}
\begin{proof}
Similar to the above proof, it suffices to produce a natural isomorphism when $S$ is qrsp. In this case, it follows from an analogous argument by using \cref{p-div} and sheafiness of the functor in \cref{smallworld}. Alternatively, one may directly deduce it from \cref{ecc}, \cref{ecc2} and \cref{earth}.
\end{proof}{}

\begin{construction}\label{impnotation}
 Let $T$ be a quasisyntomic ring and let $\mathcal{M},\,\mathcal{N}$ be two objects of $D_{\mathrm{qc}}(\Spf(T)^\syn).$ Then the functor $(T)_\qrsp \to D(\mathbb Z)$ determined by the association
$$(T)_\qrsp \ni S \mapsto R\Hom_{D_{\mathrm{qc}}(\Spf(S)^\syn)}\left(\mathcal{M}|_{\Spf(S)^\syn)}, \mathcal{N}|_{\Spf(S)^\syn)} \right)$$ is a $D(\mathbb Z)$-valued quasisyntomic sheaf; we denote this sheaf by $R\mathcal{H}om_{\syn} (\mathcal{M}, \mathcal{N}).$ 
\end{construction}{}

\begin{proposition}[Duality compatibility for finite flat group schemes]\label{dualoo}
    Let $T$ be a quasisyntomic ring and $G \in \mathrm{FFG}(T).$ Then $\mathcal{M}(G)$ is a dualizable object of $F\mathrm{\w{-}Gauge}_{\Prism}(T).$ Further, if $G^\vee$ is the Cartier dual of $G,$ then we have a natural isomorphism
$$ \mathcal{M}(G)^* \left \{-1 \right \} [1]\simeq  \mathcal{M}(G^\vee).$$ 
\end{proposition}{}

\begin{proof}
 Dualizability of $\mathcal{M}(G)$ can be checked locally and follows from \cref{atias} and 
\cref{starbc} (\textit{cf}.~\cref{verys1} and \cref{need}).   

For the latter part, since both sides satisfy quasisyntomic descent, it is enough to produce a natural isomorphism when $T$ is qrsp. Note that the functor determined by
$$(T)_\qrsp \ni S \mapsto \Fil^\bullet \Prism_{S}$$ is a sheaf with values in filtered objects of $D(\mathbb Z).$
Therefore, we have a natural isomorphism
$$R\Hom_{(T)_\qrsp}\left(R\mathcal{H}om_{(T)_\qrsp}(\mathbb Z, G^\vee), \Fil^\bullet \Prism_{(\cdot)}[1]\right) \simeq R\mathrm{Hom}_{(T)_\qrsp} \left(G^\vee, \Fil^\bullet \Prism_{(\cdot)}[1]\right),$$
where the $D(\mathbb Z)$-valued sheaf $R\mathcal{H}om_{(T)_\qrsp}(\mathbb Z, G^\vee)$ is determined by the association $$(T)_\qrsp \ni S \mapsto R\Gamma_{\qsyn}(S, G^\vee).$$ By \cref{earth}, we have an isomorphism
\begin{equation*}
 \begin{split}
  R\Hom_{(T)_\qrsp}\left(R\mathcal{H}om_{\syn}(\mathcal{O}, \mathcal{M}(G)\left \{1 \right \}), \Fil^\bullet \Prism_{(\cdot)}[1]\right) \\  \simeq R\Hom_{(T)_\qrsp}\left(R\mathcal{H}om_{(T)_\qrsp}(\mathbb Z, G^\vee), \Fil^\bullet \Prism_{(\cdot)}[1]\right).   
 \end{split}   
\end{equation*}
By dualizability of $\mathcal{M}(G),$ we have a natural isomorphism $$R\mathcal{H}om_{\syn}(\mathcal{O}, \mathcal{M}(G)\left \{1 \right \}) \simeq R\mathcal{H}om_{\syn}(\mathcal{M}(G)^*\left \{-1 \right \}, \mathcal O).$$

Note that we have a natural evaluation map (the counit of the adjunction in 
\cref{adjoinnn})
$$\mathcal{M}(G)^*\left \{-1 \right \}[1] \to R\Hom_{(T)_\qrsp}\left(R\mathcal{H}om_{\syn}(\mathcal{M}(G)^*\left \{-1 \right \}, \mathcal O), \Fil^\bullet \Prism_{(\cdot)}[1]\right)$$ of prismatic weak $F$-gauges over $T$ (\cref{earthq1}). Combining with the above isomorphisms, we have a natural map $$ \mathcal{M}(G)^*\left \{-1 \right \}[1] \to R\mathrm{Hom}_{(T)_\qrsp} \left(G^\vee, \Fil^\bullet \Prism_{(\cdot)}[1]\right)$$ of prismatic weak $F$-gauges over $T.$ Repeating the same argument for any $S \in (T)_\qrsp,$ we obtain a natural map
\begin{equation}\label{naturalmap}
 \mathcal{M}(G_S)^*\left \{-1 \right \}[1] \to R\mathrm{Hom}_{(S)_\qrsp} \left(G_S^\vee, \Fil^\bullet \Prism_{(\cdot)}[1]\right)   
\end{equation} 
of prismatic weak $F$-gauges over $S.$ Note that by \cref{abdual}, \cref{atias} and \cref{verys1}, the prismatic $F$-gauge $\mathcal{M}(G)^*\left \{-1\right\}[1]$ viewed as a sheaf on $(T)_\qrsp$ of filtered module over $\Fil^\bullet \Prism_{(\cdot)}$ equipped with a Frobenius map is sheafification of the functor determined by
$$(T)_\qrsp \ni S \mapsto \tau_{\ge 0}(\mathcal{M}(G_S)^*\left \{-1 \right \}[1]).$$
By \cref{cafe4} and \cref{cafe3}, the prismatic $F$-gauge $\mathcal{M}(G^\vee)$ is obtained from sheafification of the functor determined by 
 $$(T)_\qrsp \ni S \mapsto \tau_{\ge 0}R\mathrm{Hom}_{(S)_\qrsp} \left(G_S^\vee, \Fil^\bullet \Prism_{(\cdot)}[1]\right)$$ of prismatic $F$-gauges over $T.$ Thus, \cref{naturalmap} constructs a natural map
$$\mathcal{M}(G)^*\left \{-1 \right \}[1] \to \mathcal{M}(G^\vee).$$ Since one can locally embed $G$ into abelian schemes, by using \cref{abdual}, \cref{atias} and \cref{verys1}, we see that the above map must be an isomorphism. This finishes the proof.
\end{proof}{}

\begin{proposition}[Duality compatiblity for $p$-divisible groups]\label{grading}
    Let $G$ be a $p$-divisible group of height $h$ over a quasisyntomic ring $T$. If $G^\vee$ is the Cartier dual of $G,$ then we have a natural isomorphism
$$ \mathcal{M}(G^\vee) \simeq \mathcal{M}(G)^* \left \{-1 \right \}$$ of prismatic $F$-gauges over $T.$
\end{proposition}{}

\begin{proof}
Let $G_n = G[p^n].$ We have a fiber sequence $G_n \to G \xrightarrow{p^n} G$. By \cref{cofiber}, we have a fiber sequence
$$ \mathcal{M}(G) \xrightarrow{p^n} \mathcal{M}(G) \xrightarrow{} \mathcal{M}(G_n).$$ Dualizing, we obtain that $\mathcal{M}(G_n)^* \simeq (\mathcal{M}(G)^*/p^n) [-1].$ Now, we have
$$\mathcal{M}(G^\vee) \simeq \varprojlim \mathcal{M}(G_n^\vee) \simeq \varprojlim \mathcal{M}(G_n)^* \left \{-1 \right \} [1],$$ where the last step follows from \cref{dualoo}. Further, $$\varprojlim \mathcal{M}(G_n)^* \left \{-1 \right \} [1] \simeq \varprojlim (\mathcal{M}(G)^*\left \{-1 \right \})/p^n \simeq \mathcal{M}(G)^* \left \{ -1 \right \}.$$This finishes the proof.\end{proof}

\begin{proposition}\label{mainthm22}
Let $T$ be a quasisyntomic ring. The prismatic Dieudonn\'e $F$-gauge functor $$\mathcal{M}: \mathrm{FFG}(T)^\mathrm{op} \to F\text{-}\mathrm{Gauge}^{\mathrm{}}_{\Prism}(T)$$ determined by $G \mapsto \mathcal{M}(G)$ is fully faithful.  
\end{proposition}{}

\begin{proof}
 We use $\Hom$ to denote the connective truncation of $R\Hom$ (i.e., $\tau_{\ge 0} R\Hom$) and refer to it as the mapping space. We would like to prove that the natural map 
\begin{equation}\label{smallworldagain}
\Hom_{\mathrm{FFG}(T)}(H, G) \to \mathrm{Hom}_{F\text{-}\mathrm{Gauge}^{\mathrm{}}_{\Prism}(T)} (\mathcal{M}(G), \mathcal{M}(H))    
\end{equation}{}
is an isomorphism. For notational simplicity, we will sometimes omit the subscript indicating the category where maps are being taken. Note that for any $G \in \mathrm{FFG}(S),$ the prismatic $F$-gauge $\mathcal{M}(G)$ is effective by construction. It follows that the mapping spaces on the right hand side of \cref{smallworldagain} can also be computed in the category $F\text{-}\mathrm{Mod}_{\Fil{^\bullet}\Prism}(T)^{\mathrm{eff}}$ (\cref{smallworagain1}). Note that by \cref{adjoinnn}, $R\mathcal{H}om_{(S)_\qrsp}(H, \mathrm{Fil}^\bullet \Prism_{(\cdot)})$ can naturally be viewed as an object of $F\text{-}\mathrm{Mod}_{\Fil{^\bullet}\Prism}(T)^{\mathrm{eff}}.$
\end{proof}{}

\begin{lemma}\label{cocacola}
We have an isomorphism 
$$\mathrm{Hom}_{F\text{-}\mathrm{Mod}_{\Fil{^\bullet}\Prism}(T)^{\mathrm{eff}}} (\mathcal{M}(G), \mathcal{M}(H))$$ $$ \simeq \Hom_{F\text{-}\mathrm{Mod}_{\Fil{^\bullet}\Prism}(T)^{\mathrm{eff}}}(\mathcal{M}(G), R\mathcal{H}om_{(T)_\qrsp}( H, \mathrm{Fil}^\bullet \Prism_{(\cdot)}[1])).$$    
\end{lemma}{}

\begin{proof}
We will introduce some notations for the proof of the lemma. Note that the site $(T)_\qrsp^{G \times H}$ (\cref{verys1}) forms a basis for $(T)_\qrsp.$ Let us define $\mathcal{C}$ to be the $\infty$-category of sheaves on $(T)_\qrsp^{G \times H}$ of derived $(p, \mathscr I)$-complete $\mathbb Z$-indexed filtered modules $\mathrm{Fil}^{\bullet} \mathscr{N}$ over the sheaf of filtered ring $\mathrm{Fil}^{\bullet} \Prism_{(\cdot)}$ on $(T)_{\mathrm{qrsp}}^{G\times H}$ equipped with a $\mathrm{Fil}^{\bullet} \Prism_{(\cdot)}$-linear Frobenius map 
$$\varphi:  \mathrm{Fil}^\bullet \mathscr N \to \mathscr{I}^\bullet \Prism_{(\cdot)}\otimes_{\Prism_{(\cdot)}} \mathscr N =: \mathscr{I}^\bullet \mathscr N,$$ with the property that the natural maps $\Fil^i \mathscr N \to \Fil^{i-1}\mathscr N$ are isomorphisms for all $i \le 0.$ The natural restriction functor induces an equivalence $u^{-1}:F\text{-}\mathrm{Mod}_{\Fil{^\bullet}\Prism}(T)^{\mathrm{eff}} \simeq \mathcal{C}.$

Let us define $\mathcal{D}$ to be the analogous presheaf category; there is a fully faithful functor $\mathcal{C} \to \mathcal{D}.$ The left hand side of the lemma is then equivalent to $\mathrm{Hom}_{\mathcal{D}}(u^{-1}\mathcal{M}(G), u^{-1}\mathcal{M}(H)).$ Note that by \cref{ant}, the object $u^{-1}\mathcal{M}(H)$ is given by the following association $$(T)^{G \times H}_{\qrsp} \ni S \mapsto \tau_{\ge 0}R\mathrm{Hom}_{(S)_{\qrsp}}(H, \Fil^\bullet \Prism_{(\cdot)}[1]),$$ and similarly for $u^{-1}\mathcal{M}(G).$ Further, the association $$(T)^{G \times H}_{\qrsp} \ni S \mapsto R\mathrm{Hom}_{(S)_{\qrsp}}(H, \Fil^\bullet \Prism_{(\cdot)}[1]),$$ defines an object of $\mathcal{C},$ which is simply $u^{-1}R\mathcal{H}om_{(T)_\qrsp}(H, \Fil^\bullet \Prism_{(\cdot)})$
from \cref{adjoinnn}. Thus, we have
\begin{align*}
&\mathrm{Hom}_{F\text{-}\mathrm{Mod}_{\Fil{^\bullet}\Prism}(T)^{\mathrm{eff}}} (\mathcal{M}(G), \mathcal{M}(H))\\&\xrightarrow{\sim} \mathrm{Hom}_{\mathcal{D}}(u^{-1}\mathcal{M}(G), u^{-1}\mathcal{M}(H))\\&\xrightarrow{\sim} \mathrm{Hom}_{\mathcal{D}}(u^{-1}\mathcal{M}(G), u^{-1}R\mathcal{H}om (H, \Fil^\bullet \Prism_{(\cdot)}[1]))\\&\xrightarrow{\sim} \mathrm{Hom}_{\mathcal{C}}(u^{-1}\mathcal{M}(G), u^{-1}R\mathcal{H}om (H, \Fil^\bullet \Prism_{(\cdot)}[1]))\\&\xrightarrow{\sim}\Hom_{F\text{-}\mathrm{Mod}_{\Fil{^\bullet}\Prism}(T)^{\mathrm{eff}}}(\mathcal{M}(G), R\mathcal{H}om_{(T)_\qrsp}( H, \mathrm{Fil}^\bullet \Prism_{(\cdot)}[1])). 
\end{align*} In the above, second isomorphism follows because $u^{-1}\mathcal{M}(G)$ is a (pre)sheaf of connective filtered modules; the rest follows from the equivalence $u^{-1}:F\text{-}\mathrm{Mod}_{\Fil{^\bullet}\Prism}(T)^{\mathrm{eff}} \simeq \mathcal{C}.$
\end{proof}{}
Back to proof of the proposition, we have
\begin{align*}
&\Hom(\mathcal{M}(G), R\mathcal{H}om_{(T)_\qrsp}( H, \mathrm{Fil}^\bullet \Prism_{(\cdot)}[1]))\\&\xrightarrow{\sim} \Hom_{(T)_\qrsp}(H, R\mathcal{H}om_{\syn}(\mathcal{M}(G), \mathrm{Fil}^\bullet \Prism_{(\cdot)}[1]))\\&\xrightarrow{\sim}\Hom_{(T)_\qrsp} (H, R\mathcal{H}om_{\syn}(\mathcal{O}, \mathcal{M}(G)^*[1]))\\&\xrightarrow{\sim}\Hom_{(T)_\qrsp}(H, R\mathcal{H}om_{\syn}(\mathcal{O}, \mathcal{M}(G^\vee)\left\{1\right\}))\\&\xrightarrow{\sim}\Hom_{(T)_\qrsp}(H, R\mathcal{H}om_{(T)_\qrsp}(\mathbb Z, G))\\&\xrightarrow{\sim}\Hom_{(T)_\qrsp}(H,G).
\end{align*}
Here, the first isomorphism follows directly from the adjunction in \cref{adjoinnn} (after taking opposite categories into account), the second isomorphism uses dualizability of $\mathcal{M}(G)$, the third isomorphism uses the duality compatibility formula in \cref{dualoo}, the fourth isomorphism uses the formula for cohomology with coefficients in $G$ (\cref{earth}), and the final isomorphism uses the fact that $H$ is connective as a presheaf of abelian groups and $\tau_{\ge 0}R\mathcal{H}om (\mathbb Z, G) \simeq G,$ as presheaves on $(T)_\qrsp.$ These chain of isomorphisms and \cref{cocacola} imply that \cref{smallworldagain} is an isomorphism, which finishes the proof.

\begin{proposition}\label{mainthm2}
Let $S$ be a quasisyntomic ring. The prismatic Dieudonn\'e $F$-gauge functor $$\mathcal{M}: \mathrm{BT}(S)^{\mathrm{op}} \to F\text{-}\mathrm{Gauge}^{\mathrm{vect}}_{\Prism}(S)$$ is fully faithful.  
\end{proposition}

\begin{proof}
This is proven similarly to the proof of \cref{mainthm22}. Indeed, let $G,H\in \mathrm{BT}(S).$ As before, let $\Hom$ denote the connective truncation of $R\Hom.$
We have
\begin{align*}
 \mathrm{Hom}(\mathcal{M}(G), \mathcal{M}(H))&\xrightarrow{\sim} \Hom(\mathcal{M}(G), R\mathcal{H}om_{(S)_\qrsp}(H, \Fil^\bullet \Prism_{(\cdot)}[1]))\\&\xrightarrow{\sim}\Hom_{(S)_\qrsp} (H, R\mathcal{H}om_{\syn}(\mathcal{O}, \mathcal{M}(G)^*[1]))\\&\xrightarrow{\sim}\Hom_{(S)_\qrsp}(H, R\mathcal{H}om_{\syn}(\mathcal{O}, \mathcal{M}(G^\vee)\left\{1\right\}[1]))\\&\xrightarrow{\sim}\Hom_{(S)_\qrsp}(H[-1], R\mathcal{H}om_{(S)_\qrsp}(\mathbb Z, T_p(G))),
\end{align*}{}
where the first isomorphism follows similar to \cref{cocacola}, the second isomorphism uses the adjunction from \cref{adjoinnn} and dualizability of $\mathcal{M}(G),$ the third isomorphism uses the duality compatibility from \cref{grading} and the fourth one uses \cref{jac}.

Note that we have a fiber sequence $T_p(H) \to \varprojlim_{p} H \to H.$ Using the fact that multiplication by $p$ is an isomorphism on $\varprojlim_p H$ and $R\mathcal{H}om_{(S)_\qrsp}(\mathbb Z, T_p(G))$ is a  derived $p$-complete sheaf on $(S)_\qrsp,$ it further follows that 
\begin{align*}
&\Hom_{(S)_\qrsp}(H[-1], R\mathcal{H}om_{(S)_\qrsp}(\mathbb Z, T_p(G)))\\&\xrightarrow{\sim}  \Hom_{(S)_\qrsp}(T_p(H), R\mathcal{H}om_{(S)_\qrsp}(\mathbb Z, T_p(G)))\\&\xrightarrow{\sim}\Hom_{(S)_\qrsp}(T_p(H), T_p(G)) \\&\xrightarrow{\sim} \Hom_{\mathrm{BT}(S)} (H, G). 
\end{align*} We explain the last isomorphism in detail. Note that, since $T_p(G)$ is derived $p$-complete, by \cref{fight}, we have $R\mathrm{Hom}_{(S)_\mathrm{qsyn}}(T_p(H), T_p(G)) \simeq  R\mathrm{Hom}_{(S)_\mathrm{qsyn}}(H, T_p(G)[1])$. Using the fiber sequence $\varprojlim_{p} G \to G \to T_p(G)[1],$ it suffices to show that $R\mathrm{Hom}(H, \varprojlim_p G)\simeq R\varprojlim_n R\mathrm{Hom}(H[p^n], \varprojlim_p G) \simeq 0$, which is true because multiplication by $p^n$ is an isomorphism on $\varprojlim_p G$ but zero on $H[p^n]$. This finishes the proof.
\end{proof}{}

\begin{remark}[Essential image of $p$-divisible groups, cf.~\cite{kisin}]\label{ias1}
 Let $S$ be a qrsp algebra and let $G$ be a $p$-divisible group over $S.$ By \cref{dualo1}, we know that $\mathcal{M}(G)$ is a vector bundle on $S^{\w{syn}}.$ Using the description from \cref{dualo1} \cref{claritin}, it follows that if we denote the underlying filtered $\Fil^\bullet_\Nyg \Prism_R$-module of $\mathcal{M}(G)$ as $\Fil^\bullet M$, then 
\begin{equation}\label{}
\mathrm{gr}^\bullet M \otimes_{\mathrm{gr}^\bullet_{\Nyg} \Prism_S} S \simeq t_{G^\vee} \oplus \omega_G,   
\end{equation}where as a graded $S$-module, $t_{G^\vee}$ is in weight $0$ and $\omega_G$ is in weight $1.$ By \cref{guolii}, we conclude that $\mathcal{M}(G)$ has Hodge--Tate weights in $[0,1].$ By quasisyntomic descent, if $G$ is a $p$-divisible group over a quasisyntomic ring $T,$ it follows that $\mathcal{M}(G)$ has Hodge--Tate weights in $[0,1].$ Therefore the functor from \cref{mainthm2} refines to a fully faithful functor
\begin{equation}\label{thisisequiv}
 \mathcal{M}: \mathrm{BT}(T)^{\mathrm{op}} \to  F\mathrm{\text{-}{Gauge}}^{\mathrm{vect}}_{[0,1]}(T),   
\end{equation} where the latter denotes the category 
of vector bundles on $\Spf(T)^\syn$ of Hodge--Tate weights in $[0,1].$ The functor in \cref{thisisequiv} is an \textit{equivalence} of categories. Indeed, in order to check that the functor is essentially surjective, by descent, we can reduce to the case when $T$ is a qrsp algebra (see \cite[Prop.~A2]{alb}). In that case, under the equivalence in \cref{comppp}, using the argument in the second paragraph of the proof of \cite[Thm.~4.9.5]{alb}, one can further reduce to the case when $T$ is perfectoid, when the claim follows from \cite[Theorem 17.5.2]{berk}.
\end{remark}

\begin{remark}\label{clari}
Let $A$ be an abelian scheme over a quasisyntomic ring $S.$ Similar to \cref{ias1}, it follows that $\mathcal{M}(A)$ is a vector bundle over $\Spf(S)^\syn$ with Hodge--Tate weights in $[0,1].$    
\end{remark}{}



\begin{remark}[Essential image of finite locally free $p$-power rank group schemes] Let $S$ be a quasisyntomic ring. Let $D_{\w{perf}}(\Spf(S)^{\syn})$ denote the full subcategory of dualizable objects of $D_{\mathrm{qc}}(\Spf(S)^\syn)$. We will determine the full subcategory of $D_{\w{perf}}(\Spf(S)^{\syn})$ spanned by the essential image of finite locally free commutative group schemes of $p$-power rank over $S.$ Let $\mathrm{Vect}^{\mathrm{iso}}_{[0,1]}(S^{\w{syn}})$ be the full subcategory of $D_{\w{perf}}(S^{\w{syn}})$ spanned by objects $M$ satisfying the following properties:

\begin{enumerate}
    \item There exists a quasisyntomic cover $(S_i)_{i \in I}$ of $S$ such that $M_i:=M|_{{S_i}^{\w{syn}}}$ is isomorphic to $\mathrm{cofib}(V_i \xrightarrow{f_i} V'_i),$ for some $V_i, V_i' \in \mathrm{Vect}(S^{\w{syn}})$ and $f_i: V_i \to V'_i.$

\item The vector bundles $V_i, V_i'$ appearing above have Hodge--Tate weights in $[ 0, 1]$ for all $i \in I.$

\item The map $f_i$ appearing above has the property that it is an isomorphism when viewed in the category $\mathrm{Vect}(S^{\w{syn}})[\frac 1 p].$
\end{enumerate}{}
Let $G \in \mathrm{FFG}(S).$ By a result of Raynaud \cite[Thm.~3.1.1]{bbm}, one can Zariski locally realize $G$ as a kernel of an isogeny $A' \to A$ of abelian varieties. By \cref{atias}, it follows that (locally) we have a fiber sequence $$ \mathcal{M}(A) \to \mathcal{M}(A') \to \mathcal{M}(G).$$ Since (\cref{clari}), $\mathcal{M}(A)$ and $\mathcal{M}(A')$ are vector bundles of Hodge--Tate weights in $[0,1],$ it follows that $\mathcal{M}(G)$ indeed lies in $\mathrm{Vect}^{\mathrm{iso}}_{[0,1]}(S^{\w{syn}}).$ Now we can formulate the following.

\begin{proposition}\label{laststatement}
   The functor
\begin{equation}\label{ias}
 \mathcal{M}: \mathrm{FFG}(S)^{\mathrm{op}} \to \mathrm{Vect}^{\mathrm{iso}}_{[0,1]}(S^{\syn})\end{equation}induces an equivalence of categories.
\end{proposition}{}
\begin{proof}
Let us introduce some notations for the proof. For $M \in \mathrm{Vect}^{\mathrm{iso}}_{[0,1]}(S^{\w{syn}}),$ we define $T(M)$ to be the $D(\mathbb{Z})$-valued sheaf on $(S)_\qsyn$ given by $R\mathcal{H}om_{\syn}(\mathcal{O}, M \left \{1 \right \})$ (see \cref{impnotation}). We will denote $T^0 (M):= \tau_{\ge 0} T(M),$ which is a sheaf of abelian groups on $(S)_\qsyn.$ For $G \in \mathrm{FFG}(S),$ we denote $\mathcal{M}^\vee (G):= \mathcal{M}(G^\vee).$

It would be enough to prove that if $M \in \mathrm{Vect}^{\mathrm{iso}}_{[0,1]}(S^{\syn}),$ then the quasisyntomic sheaf $T^0 (M)$ is representable by a group scheme and lies in $\mathrm{FFG}(S),$ and $\mathcal{M}^\vee (T^0 (M)) \simeq M.$ To this end, by descent (see the proof of \cite[Prop.~A2]{alb}), we may work locally and assume without loss of generality that there exists $V, V' \in \mathrm{Vect}(S^{\w{syn}})$ with Hodge--Tate weights in $[0,1]$ and a map $f: V \to V'$ such that we have a fiber sequence $V \to V' \to M$ and $f$ is an isomorphism when viewed in the category $\mathrm{Vect}(S^{\w{syn}})[\frac 1 p]$. By construction, we have a fiber sequence \begin{equation}\label{fuld1}
  T(V) \to T(V') \to T(M)  
 \end{equation}By \cref{ias1}, the map $f$ corresponds to an isogeny $\underline{f}: G' \to G$ of $p$-divisible groups such that $\mathcal{M}(G'), \mathcal{M}(G)$ identify with $V',V$ respectively. We have a fiber sequence 
\begin{equation}\label{pii}
H^\vee \to G^\vee \to G'^\vee  
\end{equation}
where $H$ is a finite locally free commutative group scheme of $p$-power rank. In particular, it follows that $H^\vee$ is killed by a power of $p.$ Applying derived $p$-completion to \cref{pii}, we obtain a fiber sequence $$H^\vee \to T_p(G^\vee)[1] \to T_p(G'^\vee)[1],$$ which maybe rewritten as a fiber sequence 
\begin{equation}\label{ias22}
 T_p(G^\vee) \to T_p(G'^\vee) \to H^\vee   
\end{equation}of quasisyntomic sheaves of abelian groups on $S$.

 Using \cref{fuld1}, \cref{ias22} and \cref{jac}, it follows that the $D(\mathbb Z)$-valued sheaf on $(S)_\qsyn$ determined by $$(S)_\qsyn \ni A \mapsto R\Gamma_{\qsyn}(A, H^\vee)$$ is naturally isomorphic to $T(M).$ This shows that $T^0 (M) \simeq H^\vee$ as quasisyntomic sheaf of abelian groups. Now, $\mathcal{M}^\vee (T^0 (M)) \simeq \mathcal{M}^\vee(H^\vee) \simeq \mathcal{M}(H) \simeq \mathrm{cofib} (\mathcal{M}(G) \to \mathcal{M}(G')),$ where the last isomorphism follows from the fiber sequence $H \to G' \to G$ and \cref{ecc57}. Since $\mathcal{M}(G'), \mathcal{M}(G)$ naturally identify with $V',V$ respectively, and we have a fiber sequence $V \to V' \to M,$ we see that $\mathcal{M}(H) \simeq M.$ Thus, we obtain $\mathcal{M}^\vee (T^0 (M)) \simeq M,$ which finishes the proof.
\end{proof}

\end{remark}

\begin{remark}
In a future joint work with Madapusi, we will show that the  natural functor from the category $\mathrm{Vect}^{\mathrm{iso}}_{[0,1]}(S^{\syn})$ to the category of $p$-power torsion perfect complexes on $\Spf(S)^\syn$ with Tor-amplitiude in homological degrees $[0,1]$ and with Hodge--Tate weights in $[0,1]$ is an equivalence.
\end{remark}

\begin{remark}[\'Etale realizations of prismatic Dieudonn\'e $F$-gauges]We now discuss how to recover local systems (on the generic fiber) associated to $p$-divisible groups or finite locally free commutative group schemes of $p$-power rank $G$ from their prismatic Dieudonn\'e $F$-gauge $\mathcal{M}(G).$ 

Let $S$ be a quasisyntomic ring. As constructed in \cite[Cons.~6.3.1]{fg}, there is a symmetric monoidal \'etale realization functor
$$T_{\mathrm{\acute{e}t}}: \mathrm{Perf}(\Spf(S)^\syn) \to D^b_{\mathrm{lisse}} (\mathrm{Spa}(S[ 1 /p], S), \mathbb{Z}_p),$$ where we think of $\mathrm{Spa}(S[1/p], S)$ as a presheaf on affinoid perfectoid spaces over $\mathbb{Q}_p$ and $ D^b_{\mathrm{lisse}} (\mathrm{Spa}(S[ 1 /p], S), \mathbb{Z}_p)$ is the full subcategory of $D^b_{\mathrm{pro}\w{-}\mathrm{\acute{e}t}}(\mathrm{Spa}(S[1/p], S), \mathbb{Z}_p)$ with lisse cohomology (see \cite[3.1]{bscam}). Note that any $\mathcal{F} \in D^b_{\mathrm{lisse}}(\mathrm{Spa}(S[1/p], S), \mathbb{Z}_p)$ automatically satisfies $v$-descent (in the derived sense). Let us denote $\Spf(S)_{\eta}:= \mathrm{Spa}(S[1/p], S).$ For any $G \in \mathrm{FFG}(S),$ we will identify $\Te(\mathcal{M}(G)\left \{1 \right \})$ with the generic fiber $G^\vee_{\eta}$ of the Cartier dual $G^\vee$.
\end{remark}{}
\begin{proposition}\label{finalprop}
Let $S$ be a quasisyntomic ring and $G \in \mathrm{FFG}(S)$. In the above set up, we have an isomorphism
$$\Te(\mathcal{M}(G)\left \{1 \right \}) \simeq G^\vee_{\eta}.$$   
\end{proposition}{}
\begin{proof}
Let $D((S)_\qrsp, \mathbb{Z}_p)$ denote the category of $p$-complete objects in the derived category of $\mathbb{Z}_p$-modules on the site $(S)_\qrsp$ and let $D_{v}(\mathrm{Spf}(S)_\eta, \mathbb{Z}_p)$ denote the category of $p$-complete objects in the derived category of $\mathbb{Z}_p$-modules on the $v$-site of $\mathrm{Spf}(S)_\eta.$ There is a canonical functor $u^*: D((S)_\qrsp, \mathbb{Z}_p) \to D_{v}(\mathrm{Spf}(S)_\eta, \mathbb{Z}_p)$ determined by (sheafifying) $\mathcal{F} \mapsto \mathcal{G}:= \left ( \mathrm{Spa}(R, R^+) \mapsto \mathcal{F}(R^+)  \right).$ We note that $u^*$ has a right adjoint $u_*: D_{v}(\mathrm{Spf}(S)_\eta, \mathbb{Z}_p) \to D((S)_\qrsp, \mathbb{Z}_p)$ determined by
$$\mathcal{G} \mapsto \mathcal{F}:= \left((S)_\qrsp \ni R \mapsto \mathcal{F}(\mathrm{Spa}(R_{\mathrm{perfd}}[1/p], R_{\mathrm{perfd}} )\right )).$$ In the above, $R_{\mathrm{perfd}}$ is the universal perfectoid ring that receives a map from $R$ (see \cite[Thm.~1.12]{BS19}). The fact that the above assignment defines a sheaf follows from the observations that $(S)_\qrsp \ni R \mapsto R_{\mathrm{perfd}}$ takes $v$-covers to $v$-covers and preserves Tor-independent pushouts -- the first observation follows from \cite[Lem.~8.8, Prop.~8.10]{BS19}, where the conclusions remain valid by replacing arc-covers by $v$-covers, and the second observation follows from \cite[Prop.~8.13]{BS19}. Note that by construction, $u^*u_*$ is naturally isomorphic to identity.

By construction of the functor $\Te$ (see \cite[Cons.~6.3.1,~6.3.2]{fg}, \cite[Cor.~3.7]{BS19}), it follows that 
\begin{equation}\label{rainy}
\Te(\mathcal{M}(G)) \simeq u^*( \mathcal{M}(G)  [1/ \mathscr I]^{\wedge}_p)^{\varphi = 1}   
\end{equation}
Using the Artin--Schreier sequence for the tilt of the structure sheaf of $\Spf(S)_{\eta}$ along with \cite[Prop.~8.5, Prop.~8.8]{diamond}, it follows that $u_* \mathbb{Z}/p^n \simeq \mu_* \mathbb{Z}/p^n$ (in the category $D((S)_\qrsp, \mathbb{Z}_p)$) where $\mu_* \mathbb{Z}/p^n$ denotes (quasisyntomic sheafification of) the derived pushforward along the map $\mu: \Spf(S)_{\eta, \mathrm{\acute{e}t}} \to \Spf(S)_{\mathrm{\acute{e}t}}$ of topoi. By \cite[Thm.~9.1]{BS19}, we have an isomorphism $u_* \mathbb{Z}/p^n \simeq (\Prism_{(\cdot)}[1/ \mathscr I]/p^n)^{\varphi= 1}.$ Passing to limits, we have 
\begin{equation}\label{hon}
 u_* \mathbb{Z}_p \simeq (\Prism_{(\cdot)}[1/ \mathscr I]^{\wedge}_p)^{\varphi= 1}.   
\end{equation}
 Note that we have $\mathcal{E}xt^2 (G_{\eta}, \mathbb{Z}_p)=0$; in fact, since every finite locally free commutative group scheme of $p$-power rank over a base where $p$ is invertible must be finite \'etale\footnote{To deduce this, by \cite[Tag~02GM,~Tag~02VN]{stacks}, one can reduce to the base being an algebraically closed field, which is classical, e.g., see \cite[\S~14.4]{waterhouse}. }, by the proof of \cref{w}, for any affinoid perfectoid space $X$ over $\Spf(S)_\eta$, a class in $\mathrm{Ext}^2 (G_{\eta, X}, \mathbb{Z}_p)$ is killed by a finite etale cover of $X,$ which must again be affinoid perfectoid by the almost purity theorem \cite[Thm.~1.10]{pspaces}. This implies that
$\mathcal{H}^2 (u^*R\mathcal{H}om (G, u_* \mathbb{Z}_p))=0$ via adjunction and the fact that $u^*u_* \simeq \mathrm{id}$.
Therefore, we have a fiber sequence of (derived) sheaves 
\begin{equation*}\begin{split}
    \tau_{\ge 0}u^*R\mathcal{H}om (G, u_*\mathbb{Z}_p[1]) \to \tau_{\ge 0}   u^*R\mathcal{H}om (G, \Prism_{(\cdot)}[1/\mathscr I]^{\wedge}_p[1]) \\ \xrightarrow{\varphi - 1}\tau_{\ge 0} u^*   R\mathcal{H}om (G, \Prism_{(\cdot)}[1/\mathscr I]^{\wedge}_p[1]).
\end{split}{}
\end{equation*}    
By construction of $\mathcal{M}(G)$, and the above fiber sequence, it follows that $$u^* (\mathcal{M} (G) [1/ \mathscr I]^{\wedge}_p)^{\varphi = 1} \simeq  \tau_{\ge 0}u^*R\mathcal{H}om (G, u_*\mathbb{Z}_p[1]).$$ 
Combining with \cref{rainy}, and using adjunction, this implies that
\begin{equation}\label{rainy1}
 \Te (\mathcal{M}(G)) \simeq    \tau_{\ge 0}R\mathcal{H}om (G_\eta, \mathbb{Z}_p[1]).
\end{equation} By \cref{hon}, $\Te(\mathcal{O}) \simeq \mathbb{Z}_p.$ Applying \cref{rainy1} to $G= \mu_{p^n}$ and taking limits, we obtain
$\Te(\mathcal{O}\left \{-1\right \}) \simeq \mathbb{Z}_p(-1).$ Since $\Te$ is symmetric monoidal, we have $\Te(\mathcal{O}\left \{1\right \}) \simeq \mathbb{Z}_p(1).$ Using the fact that $\Te$ is symmetric monoidal again, we have an isomorphism
\begin{equation}\label{rainy11}
 \Te (\mathcal{M}(G)\left \{1\right \}) \simeq    \tau_{\ge 0}R\mathcal{H}om (G_\eta, \mathbb{Z}_p(1)[1]).
\end{equation}

Further, since $G_\eta$ is killed by a power of $p$, and $\mathbb{Z}_p$ is $p$-torsion free, we have an isomorphism of sheaves $\tau_{\ge 0} R\mathcal{H}om_{\Spf(S)_{\eta, v}}(G_\eta, \mathbb{Z}_p[1]  ) \simeq \mathcal{E}xt^1_{\Spf(S)_{\eta, v}}(G_\eta, \mathbb{Z}_p ).$ Finally, since $G^\vee_{\eta} \simeq \mathcal{E}xt^1_{\Spf(S)_{\eta, v}}(G_\eta, \mathbb{Z}_p(1) )$ (cf.~\cref{imp}), we obtain the desired statement.
\end{proof}{}

\begin{remark}We end this paper by noting some consequences of \cref{finalprop}. By using \cref{ecc} and taking limits, we obtain that (cf.~\cite[\S~5.4]{alb})\footnote{Note that the statement we prove is formulated at the derived level, and therefore is stronger than loc.~cit.} for a quasisyntomic ring $S$, and a $G \in \mathrm{BT}(S)$, we have 
$$\Te(\mathcal{M}(G)\left \{1 \right \}) \simeq R\varprojlim \Te(\mathcal{M}(G)\left \{1 \right \})/p^n \simeq R\varprojlim \Te(\mathcal{M}(G)\left \{1 \right \}/p^n) \simeq R\lim G_\eta^\vee[p^n] \simeq T_p(G^\vee_{\eta}),$$ where the last isomorphism follows from repleteness of the $v$-site (see proof of \cite[Prop.~14.10]{diamond} and \cref{notconven} (9)). Further, note that using \cite[Prop.~2.14, Cor.~3.7]{bscam}, for $G \in \mathrm{FFG}(S),$ we obtain
$$R\Gamma_{\qsyn}(\Spf(S), \mathcal{M}(G) \left \{1 \right \}[1/\mathscr I]^\wedge _p)^{\varphi= 1} \simeq R\Gamma_{{\mathrm{pro}\w{-}\mathrm{\acute{e}t}}}(\mathrm{Spf}(S)_\eta, G^\vee_{\eta}),$$ and similarly, for $G \in \mathrm{BT}(S),$ we obtain
$$R\Gamma_{\qsyn}(\Spf(S), \mathcal{M}(G) \left \{1 \right \}[1/\mathscr I]^\wedge _p)^{\varphi= 1} \simeq R\Gamma_{{\mathrm{pro}\w{-}\mathrm{\acute{e}t}}}(\mathrm{Spf}(S)_\eta, T_p(G^\vee_{\eta})).$$

\end{remark}{}
\newpage

\bibliographystyle{amsalpha}
\bibliography{main}

\end{document}